\documentclass[11pt,reqno,twoside]{amsart}
\usepackage{graphicx}
\usepackage{amssymb}
\usepackage{epstopdf}
\usepackage[asymmetric,top=3.5cm,bottom=3.3cm,left=3.1cm,right=3.1cm]{geometry}
\usepackage{booktabs} % for much better looking tables
\usepackage{array} % for better arrays (eg matrices) in maths
\usepackage{paralist} % very flexible & customisable lists (eg. enumerate/itemize, etc.)
\usepackage{verbatim} % adds environment for commenting out blocks of text & for better verbatim
\usepackage{subfig} % make it possible to include more than one captioned figure/table in a single float
\usepackage{tabularx}
\usepackage{amsmath,amsfonts,amsthm,mathrsfs,amssymb,cite}
\usepackage[usenames]{color}
\usepackage{bm}
\usepackage{tabularx}
\usepackage{longtable}
\usepackage{booktabs} % for much better looking tables

\newtheorem{thm}{Theorem}[section]

\newtheorem{lem}{Lemma}[section]
\newtheorem{prop}{Proposition}[section]
\theoremstyle{definition}
\newtheorem{defn}{Definition}[section]
\theoremstyle{remark}

\newtheorem{rem}{Remark}[section]
\numberwithin{equation}{section}

\newcommand{\R}{\mathbb{R}}
\newcommand{\C}{\mathbb{C}}

\newcommand{\bmf}[1]{{\mathbf{#1}}}
\def\bsi{{\mathrm{i}}}

\def\Oh{{\mathcal  O}}

\newcommand{\rmd}{\mathrm{d}}

\newcommand{\tabincell}[2]{\begin{tabular}{@{}#1@{}}#2\end{tabular}}

\allowdisplaybreaks

\usepackage{hyperref}
\usepackage{xcolor}
\hypersetup{
  colorlinks,
  linkcolor={blue!80!black},
  urlcolor={blue!80!black},
  citecolor={blue!80!black}
}
\usepackage[capitalize,noabbrev]{cleveref}

\title[Generalized Holmgren's principle to Lam\'e's operator and applications]{Further results on generalized Holmgren's principle to the Lam\'e operator and applications}

\author{Huaian Diao}
\address{School of Mathematics and Statistics, Northeast Normal University,
Changchun, Jilin 130024, China.}
\email{hadiao@nenu.edu.cn}

\author{Hongyu Liu}
\address{Department of Mathematics, City University of Hong Kong, Kowloon, Hong Kong SAR, China.}
\email{hongyu.liuip@gmail.com; hongyliu@cityu.edu.hk}

\author{Li Wang}
\address{School of Mathematics and Statistics, Northeast Normal University,
Changchun, Jilin 130024, China.}
\email{1322771004@qq.com}

\date{} % Activate to display a given date or no date (if empty),
\begin{document}
\maketitle

\begin{abstract}

In our earlier paper \cite{DLW}, it is proved that a homogeneous rigid, traction or impedance condition on one or two intersecting line segments together with a certain zero point-value condition implies that the solution to the Lam\'e system must be identically zero, which is referred to as the generalized Holmgren principle (GHP). The GHP enables us to solve a longstanding inverse scattering problem of determining a polygonal elastic obstacle of general impedance type by at most a few far-field measurements. In this paper, we include all the possible physical boundary conditions from linear elasticity into the GHP study with additionally the soft-clamped, simply-supported as well as the associated impedance-type conditions. We derive a comprehensive and complete characterisation of the GHP associated with all of the aforementioned physical conditions. As significant applications, we establish novel unique identifiability results by at most a few scattering measurements not only for the inverse elastic obstacle problem but also for the inverse elastic diffraction grating problem within polygonal geometry in the most general physical scenario. We follow the general strategy from \cite{DLW} in establishing the results. However, we develop technically new ingredients to tackle the more general and challenging physical and mathematical setups. It is particularly worth noting that in \cite{DLW}, the impedance parameters were assumed to be constant whereas in this work they can be variable functions.

\medskip

\medskip



\end{abstract}

\section{Introduction}

This paper is continuation of our earlier paper \cite{DLW} on a novel unique continuation principle for the Lam\'e operator and its applications to several challenging inverse elastic problems. We begin by briefly introducing the background and motivation of our study and referring to \cite{DLW} for more related discussions. 

Let $\Omega\subset\mathbb{R}^2$ be a bounded open set. $\bmf u=(u_\ell)_{\ell=1}^2\in L^2(\Omega)^2$ is said to be a (generalised) Lam\'e eigenfunction if
\begin{equation}\label{eq:lame}
-\mathcal{L} (\bmf{u})=\kappa \bmf{u}\ \text { in }\ \Omega,\ \ \kappa \in\mathbb{R}_+, 
\end{equation}
where
\begin{equation}\label{eq:pdo1}
\mathcal{L}(\bmf u) :=\mu \Delta \bmf u+(\lambda+\mu) \nabla(\nabla \cdot \bmf u )
\end{equation}
and $\mathcal{L}$ is known as the Lam\'e operator that arises in the theory of linear elasticity. Here $\lambda,\mu$ are the Lam\'e constants satisfying the following strong convexity condition
\begin{equation}\label{eq:convex}
\mu>0 \text { and } \lambda+\mu>0.
\end{equation}
In \eqref{eq:lame}, it is noted that we do not assume any boundary condition on $\partial\Omega$ for $\bmf{u}$ since we shall be mainly concerned with a local property of $\mathbf{u}$. \eqref{eq:lame} arises in the study of the time-harmonic elastic wave scattering with $\mathbf{u}$ and $\omega:=\sqrt{\kappa}$ respectively signifying the elastic displacement field and the angular frequency. We next present six trace operators associated with $\bmf{u}$ in \eqref{eq:lame} that arise in different physical scenarios in the theory of linear elasticity. To that end, we let $\Gamma_h\Subset\Omega$ be a closed connected line segment, where $h\in\mathbb{R}_+$ signifies the length of the line segment. Let
\begin{equation}\label{eq:nutau}
{\nu }=(\nu_1,\nu_2)^\top \mbox{ and }
 \boldsymbol{\tau}=(-\nu_2,\nu_1)^\top
\end{equation}
respectively, signify the unit normal and tangential vectors to $\Gamma_h$. The traction $T_\nu\bmf{u}$ on $\Gamma_h$ is defined by
\begin{equation}\label{eq:Tu}
 T_{\bmf{\nu }}\bmf{u}=2\mu\partial_{\bmf{\nu }}\bmf{u}+\lambda\bmf{\nu} \left(\nabla \cdot \bmf{u}\right)+\mu\boldsymbol{\tau}(\partial_{2}u_{1}-\partial_{1}u_{2}), \end{equation}
where
$$
\nabla\bmf u:=\begin{bmatrix}
	\partial_1 u_1 & \partial_2 u_1 \cr
	\partial_1 u_2 & \partial_2 u_2
\end{bmatrix},\quad \partial_\nu\bmf u:=\nabla\bmf u \cdot  \nu, \quad \partial_j u_i:=\partial u_i/\partial x_j.
$$
Set
\begin{align}
&\mathcal{B}_1(\mathbf{u})=T_\nu\mathbf{u}\big |_{\Gamma_h}, \hspace*{1.9cm} \mathcal{B}_2(\mathbf{u})=\mathbf{u}\big |_{\Gamma_h}, \label{eq:b12}\\
& \mathcal{B}_3(\mathbf{u})\big |_{\Gamma_h }=\left(
\begin{array}{c}
{\nu} \cdot  \mathbf{u} \\
\boldsymbol{\tau} \cdot T_{\mathbf{\nu}} \bmf{u} \\
\end{array}
\right), \quad
\mathcal{B}_4(\mathbf{u}) \big |_{\Gamma_h }=\left(
\begin{array}{c}
\boldsymbol{\tau} \cdot \bmf{u} \\
{\nu} \cdot T_{\mathbf{\nu}} \mathbf{u} \\
\end{array}
\right),\label{eq:tf}\\
&\mathcal{B}_5(\mathbf{u})=\mathcal{B}_1(\mathbf{u})+\boldsymbol{\eta}\mathcal{B}_2(\mathbf{u}),\quad \mathcal{B}_6(\mathbf{u})=\mathcal{B}_3(\mathbf{u})+\boldsymbol{\eta}\mathcal{B}_4(\mathbf{u})\label{eq:b56}
\end{align}
where $\boldsymbol{\eta}\in L^\infty(\Gamma_h)$ and $ \boldsymbol \eta\equiv\hspace*{-4mm}\backslash\, 0$. Physically, the condition $\mathcal{B}_j(\mathbf{u})=0$, $j=1,\ldots, 6$, corresponds, respectively, to the case that $\Gamma_h$ is a traction-free, rigid, soft-clamped, simply-supported, impedance or generalized-impedance line segment. 

It is noted that $\mathbf{u}$ is real-analytic in $\Omega$ since $\mathcal{L}$ is an elliptic PDO with constant coefficients. The classical Holmgren's uniqueness principle (HUP) states that if any two different conditions from $\mathcal{B}_j(\mathbf{u})=\mathbf{0}$, $1\leq j\leq 4$, are satisfied on $\Gamma_h$, then $\mathbf{u}\equiv \mathbf{0}$ in $\Omega$; see \cite{DLW} for more relevant background discussion on this aspect. In \cite{DLW}, we show that under ultra weaker conditions in two scenarios, the Holmgren principle still holds, which is referred to as the generalized Holmgren principle (GHP). In the first scenario, it is proved that if only one homogeneous condition is fulfilled on $\Gamma_h$ plus a certain zero point-value condition, say $\boldsymbol{\tau}^\top \nabla \bmf{u} |_{\bmf{x}={\bmf{x}}_0 }  {\nu}=\bmf{0}$ for any fixed $\bmf{x}_0\in\Gamma_h$ in the case that $\Gamma_h$ is rigid, then one must have $\mathbf{u}\equiv \mathbf{0}$ in $\Omega$. That is, one of the two homogeneous conditions on $\Gamma_h$ in the Holmgren principle can be replaced by a point-value condition. In the other scenario, two intersecting line segments $\Gamma_h^+$ and $\Gamma_h^-$ are involved. It is proved that if two homogeneous conditions are respectively fulfilled on $\Gamma_h^\pm$ (the two homogeneous conditions can be the same to or different from each other), then one generically has $\mathbf{u}\equiv \mathbf{0}$. It is noted that in the degenerate case where the two line segments are collinear and the two homogeneous conditions are different, the GHP is reduced to the standard Holmgren principle. However, in \cite{DLW}, only $\mathcal{B}_1, \mathcal{B}_2$ and $\mathcal{B}_5$ were considered. In this paper, we shall include all the boundary traces from \eqref{eq:b12}--\eqref{eq:b56} into the GHP study, and derive a comprehensive and complete characterisation of the GHP associated with all the possible physical situations. In addition, there are two more extensions that are worthy of noting: first, in \cite{DLW}, the impedance parameter $\boldsymbol{\eta}$ is always assumed to be constant, and in this paper, it can be a variable function; second, in \cite{DLW}, a certain geometric restriction was required on the intersecting angle in the case of two intersecting line segments, and in this paper, we completely relax such a restriction. In principle, we follow the general strategy from \cite{DLW} in establishing the new GHP results. However, we develop technically new ingredients to tackle the more general and challenging physical and mathematical setups.

It is remarked that the HUP holds in a more general geometric setup where the line segment $\Gamma_h$ can be relaxed to be an analytic curve. Indeed, it is a special case of the more general Unique Continuation Principle (UCP) for elliptic PDOs. But the HUP requires the analyticity of the solution, which is also the case in our study. This is the main reason that we require $\Gamma_h$ located inside the domain $\Omega$, and refer to the new uniqueness principle as the GHP. We believe the GHP generically does not hold if $\Gamma_h$ is curved. On the other hand, our study should be able to be extended to the higher dimensional case. But as shall be seen that even for the two-dimensional case associated with flat lines, the study involves highly technical and lengthy arguments. Hence, we defer the aforementioned extensions and generalisations to our future study. 

The GHP is strongly motivated by our study of a longstanding geometrical problem in the inverse scattering theory. It is concerned with the unique determination of the geometrical shape of an obstacle by minimal/optimal or at least formally-determined scattering measurements, which constitutes an open problem in the literature. In recent years, global uniqueness results with formally-determined scattering data have been achieved in determining polygonal/polyhedral obstacles of the soft-clamped or simply-supported type \cite{ElschnerYama2010,LiuXiao}. The mathematical machinery therein is mainly based on certain reflection and path arguments that are of a global nature. The GHP enables us to develop a completely local argument in resolving the unique determination problem for polygonal obstacles of the traction-free, rigid or impedance type \cite{DLW}. In this paper, using the newly established results for the GHP, we can prove the global uniqueness in determining the shape of a general polygonal obstacle as well as its boundary impedance (if there is any) by formally-determined data (indeed at most a few scattering measurements) in the most general physical scenario. We are aware that in the non-polygonal/polyhedral setting with formally-determined scattering data, only local uniqueness/stability results were derived in the literature for rigid or traction-free obstacles, where one needs to a-priori know that the possible obstacles cannot deviate too much \cite{GM,RSS}. In addition to the inverse elastic obstacle problem, we also consider the inverse elastic diffraction grating problem of determining unbounded polygonal structures and establish several novel unique identifiability results by at most a few scattering measurements. 

The rest of the paper is organized as follows. In Section \ref{sect:2}, we present an overview and summary of the GHP results for the convenience of the readers. Sections \ref{sect:3} and \ref{sec:4} are respectively devoted to the GHP in the presence of a single homogeneous line segment and two homogeneous line segments. In Section \ref{sect:5}, we present the unique identifiability results for the inverse elastic obstacle problem and the inverse diffraction grating problem.

\section{Overview and summary of the GHP results}\label{sect:2}

In order to have a complete and comprehensive study, the derivation of the GHP results in this paper is lengthy and technically involved. In order ease the reading of the audience, we provide a brief summary of the GHP results in this section. We first fix several notations as well as the geometric setup in our study.

%We first introduce two important definitions. Before that, recall that ${\nu}$ and $\boldsymbol{\tau}$ are the unit normal and tangential vectors to a line segment $\Gamma_h$, which are defined in \eqref{eq:nutau}.  Let us denote
%
%
%  \begin{equation}\label{eq:tf}
%\mathcal{B}_3(\mathbf{u})\big |_{\Gamma_h }=\left(
%\begin{array}{c}
%{\nu} \cdot  \mathbf{u} \\
%\boldsymbol{\tau} \cdot T_{\mathbf{\nu}} \bmf{u} \\
%\end{array}
%\right), \quad
%\mathcal{B}_4(\mathbf{u}) \big |_{\Gamma_h }=\left(
%\begin{array}{c}
%\boldsymbol{\tau} \cdot \bmf{u} \\
%{\nu} \cdot T_{\mathbf{\nu}} \mathbf{u} \\
%\end{array}
%\right),
%\end{equation}
%where $\mathbf u$ is  a generalized Lam\'e eigenfunction to \eqref{eq:lame} associated with an eigenvalue $\kappa\in\mathbb{R}_+$ and $ T_{\bmf{\nu }}\bmf{u}$ is defined in \eqref{eq:Tu}.
%
%The  rigid, traction-free and impedance line of $\mathbf u$ were introduced in \cite{DLW}. In Definition \ref{def:1}, we review the aforementioned  three type line and introduce three new line of  $\mathbf u$, which shall be frequently used in our study. Before Definitions \ref{def:1} and \ref{def:2}, we introduce the class $\mathcal A$. 
%

\begin{defn}\label{def:class1p}
    Suppose that $\psi(r)$ is a complex-valued function for $r \in \Sigma:=[0, r_0]$, where $r_0\in\mathbb{R}_+$. $\psi$ is said to belong to the class $\mathcal{A}(\Sigma)$ if it allows an absolutely convergent series representation as follows
    \begin{equation}\label{eq:series1}
    \psi(r)=a_0+\sum_{j=1}^\infty a_j r^j,
    \end{equation}
    where $a_0\in\mathbb{C}\backslash\{0\}$ and $a_j \in \mathbb C$. 
    \end{defn}
    
  \begin{defn}\label{def:class1}
  Let $\Gamma_h\Subset\Omega$ be a closed line segment and $\boldsymbol{\eta}(\mathbf{x})$, $\mathbf{x}\in\Gamma_h$, be a complex-valued function. For a given $\mathbf{x}_0\in\Gamma_h$, we treat $\boldsymbol{\eta}(\mathbf{x})$ as a function in terms of the distance away from $\mathbf{x}_0$ along the line, namely $r:=|\mathbf{x}-\mathbf{x}_0|$. $\boldsymbol{\eta}$ is said to belong to the class $\mathcal{A}(\Gamma_h)$ if for any $\mathbf{x}\in\Gamma_h$, there exists a small neighborhood of $\mathbf{x}_0$ on the line segment, say $\Sigma_{\epsilon(\mathbf{x}_0)}=[0, \epsilon(\mathbf{x}_0)]$ such that $\boldsymbol{\eta}\in\mathcal{A}(\Sigma_{\epsilon(\mathbf{x}_0)})$. 
\end{defn}

Throughout the rest of the paper, we always assume that the impedance parameters in \eqref{eq:b56} belong to the class $\boldsymbol{\eta}\in\mathcal{A}(\Gamma_h)$. In fact, one can easily verify that if $\boldsymbol{\eta}$ is analytic on $\Gamma_h$ and nonvanishing at any point, then it belongs to $\mathcal{A}(\Gamma_h)$. Nevertheless, we give the explicit definition for clarity and in particular, we shall frequently make use of the series form \eqref{eq:series1} in our subsequent arguments. The GHP that we shall establish is a local property, and we shall mainly prove that $\mathbf{u}$ is identically vanishing in a sufficiently small neighbourhood around a point on $\Gamma_h$. Hence, without loss of generality, we can simply assume that the series expansion \eqref{eq:series1} holds on the whole line segment if $\boldsymbol{\eta}\in \mathcal{A}(\Gamma_h)$. Moreover, in order to ease the exposition, we shall not distinguish between a line and a line segment in what follows, which should be clear from the context. 
    
\begin{defn}\label{def:1}
Recall the rigid, traction-free, soft-clamped, simply-supported, impedance and generalized-impedance lines introduced after \eqref{eq:b12}--\eqref{eq:b56}.
 Set ${\mathcal R}_\Omega^{\kappa}$, ${\mathcal T}_\Omega^{\kappa}  $,  ${\mathcal I}_\Omega^{\kappa} $, ${\mathcal G}_\Omega^{\kappa}$, ${\mathcal F}_\Omega^{\kappa} $ and ${\mathcal H}_\Omega^{\kappa}$ to respectively denote the sets of rigid, traction-free, impedance, soft-clamped, simply-supported and generalized-impedance lines in $\Omega$ of $\bmf{u}$.
 \end{defn}

\begin{defn}\label{def:2}
Recall that the unit normal vector $\nu $ and the tangential vector $\boldsymbol{\tau}$ to $\Gamma_h$ are defined in \eqref{eq:nutau}. Define
\begin{subequations}\notag
\begin{align}
	{\mathcal S}\left(   {\mathcal R}_\Omega^{\kappa} \right):=& \{\Gamma_h \in  {\mathcal R}_\Omega^{\kappa}  ~|~ \exists \bmf{x}_0 \in \Gamma_h \mbox{ such that }
	%\boldsymbol{\tau} \cdot \partial_{\nu} \bmf{u} |_{\bmf{x}={\bmf{x}}_0 }=\bmf{0}
	\boldsymbol{\tau}^\top \nabla \bmf{u} |_{\bmf{x}={\bmf{x}}_0 }  {\nu}=\bmf{0} \},\label{eq:defa} \\
	{\mathcal S}\left(   {\mathcal T}_\Omega^{\kappa} \right):=& \{\Gamma_h \in  {\mathcal T}_\Omega^{\kappa}  ~|~ \exists \bmf{x}_0 \in \Gamma_h \mbox{ such that }  \bmf{u}(\bmf{x}_0)=\bmf{0},\, %\mbox{ and }
	%\boldsymbol{\tau} \cdot \partial_{\nu} \bmf{u} |_{\bmf{x}={\bmf{x}}_0 }=\bmf{0}
	\boldsymbol{\tau}^\top \nabla \bmf{u} |_{\bmf{x}={\bmf{x}}_0 }  {\nu}=\bmf{0}
	 \},\label{eq:defb}\\
	{\mathcal S}\left(   {\mathcal I}_\Omega^{\kappa} \right):=& \{\Gamma_h \in  {\mathcal I}_\Omega^{\kappa}  ~|~ \exists \bmf{x}_0 \in \Gamma_h \mbox{ such that }  \bmf{u}(\bmf{x}_0)=\bmf{0},\,  %\mbox{ and }
	 %\boldsymbol{\tau} \cdot \partial_{\nu} \bmf{u} |_{\bmf{x}={\bmf{x}}_0 }=\bmf{0}
	 \boldsymbol{\tau}^\top \nabla \bmf{u} |_{\bmf{x}={\bmf{x}}_0 }  {\nu}=\bmf{0} \},\label{eq:defc}\\
    {\mathcal S}\left(   {\mathcal G}_\Omega^{\kappa} \right):=& \{\Gamma_h \in  {\mathcal G}_\Omega^{\kappa}  ~|~ \exists \bmf{x}_0 \in \Gamma_h \mbox{ such that } {\nu}^\top \nabla \bmf{u} |_{\bmf{x}={\bmf{x}}_0 }  {\nu} =0  \},\label{eq:def1} \\
    {\mathcal S}\left(   {\mathcal F}_\Omega^{\kappa} \right):=& \{\Gamma_h \in  {\mathcal F}_\Omega^{\kappa}  ~|~ \exists \bmf{x}_0 \in \Gamma_h \mbox{ such that }  \boldsymbol{\tau}^\top \nabla \bmf{u} |_{\bmf{x}={\bmf{x}}_0 }  {\nu} =0 .\label{eq:def2}
\end{align}
\end{subequations}
They are referred to as the singular sets of ${\mathcal R}_\Omega^{\kappa}$, ${\mathcal T}_\Omega^{\kappa}  $,  ${\mathcal I}_\Omega^{\kappa} $, ${\mathcal G}_\Omega^{\kappa}$, ${\mathcal F}_\Omega^{\kappa} $, respectively. 
\end{defn}

\begin{rem}
Taking $\mathcal{S}\left(   {\mathcal R}_\Omega^{\kappa} \right)$ as an example, the presence of a singular rigid line $\Gamma_h\in \mathcal{S}\left(   {\mathcal R}_\Omega^{\kappa} \right)$ means that the homogeneous condition $\mathbf{u}|_{\Gamma_h}=0$ and the point-value condition $\boldsymbol{\tau}^\top \nabla \bmf{u} |_{\bmf{x}={\bmf{x}}_0 }  {\nu}=\bmf{0}$ are fulfilled for the line $\Gamma_h$ and a given point $\mathbf{x}_0\in\Gamma_h$. It is proved in \cite{DLW} that the presence of a singular rigid line implies that $\mathbf{u}\equiv\mathbf{0}$ in $\Omega$. In what follows, we shall show that the presence of any other singular lines also ensures $\mathbf{u}\equiv\mathbf{0}$ in $\Omega$. It is pointed out that we shall also introduce the singular set of the generalised-impedance lines in what follows, namely $\mathcal{S}(\mathcal{H}_\Omega^\kappa)$. But its definition involves the Fourier coefficients of $\mathbf{u}$, which we would need to introduce first.  
\end{rem}

%Compared with the homogeneous lines introduced in Definition~\ref{def:1}, we further assume that certain additional conditions on a specific point are satisfied for the singular lines in Definition~\ref{def:2}. In Section \ref{sect:3} we shall prove that if $\Omega$ contains a singular line of a (generalised) Lam\'e eigenfunction $\bmf{u}$, then $\bmf{u}$ is identically zero. We prove this by quantitatively characerizing $\bmf{u}$ in the phase space across the lines. This is the reason that we call them (microlocally) singular lines. Furthermore, we show that the generic intersections of the homogeneous lines of Definition~\ref{def:1} shall also generate microlocal singularities, which prevent the occurrence of such intersections unless $\bmf{u}$ is trivially zero. In this article, we provide a comprehensive characterization of all those cases. 

\begin{figure}%[h!]
	\centering
	\includegraphics[width=0.3\textwidth]{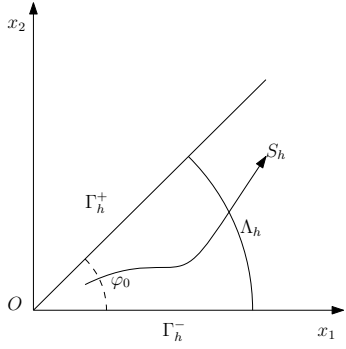}
	\caption{Schematic of the geometry of two intersecting lines with an angle $\varphi_0$ with $0< \varphi_0\leq \pi $.}
	%\label{fig:geometry}
	\label{fig1}
\end{figure}

We next introduce the geometric setup of our study. Let two line segments respectively be defined by (see Fig.~\ref{fig1} for a schematic illustration):
\begin{equation}\label{eq:gamma_pm}
	\begin{split}
		\Gamma_h^+&=\{\bmf{x} \in \mathbb R^2~|~\bmf{x}=r\cdot (\cos \varphi_0, \sin \varphi_0 )^\top ,\quad 0\leq r\leq h,\quad 0<\varphi_0\leq 2\pi   \}, \\
		\Gamma_h^-&=\{\bmf{x} \in \mathbb R^2~|~\bmf{x}=r\cdot (1, 0 )^\top,\quad 0\leq r\leq h   \},\ \ h\in\mathbb{R}_+.
	\end{split}
	\end{equation}
Clearly, the intersecting angle between $\Gamma_h^+$ and $\Gamma_h^-$ is
\begin{equation}\label{eq:angle1}
\angle(\Gamma_h^+,\Gamma_h^{-})=\varphi_0, \quad 0< \varphi_0 \leq 2\pi.
\end{equation}
We should emphasize that if the intersection of $\Gamma_h^+$ and $\Gamma_h^-$  degenerates, i.e., $\varphi_0=\pi$ or $\varphi_0=2\pi$, then  $\Gamma_h^+$ and $\Gamma_h^-$ are actually lying on a same line. In this degenerate case, we let $\Gamma_h$ denote the line, which corresponds to the case with a single line in our subsequent study. Moreover, $\Gamma_h^\pm$ shall be the homogeneous lines in Definition~\ref{def:1} or the singular lines in Definition~\ref{def:2}. Since the PDO $\mathcal{L}$ defined in \eqref{eq:lame} is invariant under rigid motions, for any two of such lines that are intersecting in $\Omega$ (or one line in the degenerate case), we can always  have two lines as introduced in \eqref{eq:gamma_pm}  after a straightforward coordinate transformation such that the homogeneous conditions in Definitions~\ref{def:1} and \ref{def:2} are still satisfied on $\Gamma_h^\pm$. Furthermore, we assume that $h\in\mathbb{R}_+$ is sufficiently small such that $\Gamma_h^\pm$ are contained entirely in $\Omega$. If $\Gamma_h^\pm$ are impedance or generalized-impedance  lines, we assume that the impedance parameters on $\Gamma_h^\pm$ are respectively  $\boldsymbol{ \eta}_1$ and $\boldsymbol{\eta}_2$, where $\boldsymbol{ \eta}_1 \in \mathcal A(\Gamma_h^+ )$ and $\boldsymbol{\eta}_2\in \mathcal A(\Gamma_h^- )$.  As also noted before that $\bmf{u}$ is analytic in $\Omega$, it is sufficient for us to consider the case that $0<\varphi_0\leq \pi$. In fact, if $\pi  < \varphi_0 \leq 2\pi$, we see that $\Gamma_h^+$ belongs to the half-plane of $x_2<0$ (see Fig.~\ref{fig1}). Let $\widetilde{\Gamma}_h^+$ be the extended line segment of length $h$ in the half-plane of $x_2>0$. By the analytic continuation, we know that $\widetilde{\Gamma}_h^+$ is of the same type of $\Gamma_h^+$, namely $\bmf{u}$ satisfies the same homogeneous condition on $\widetilde{\Gamma}_h^+$ as that on $\Gamma_h^+$. Hence, instead of studying the intersection of $\Gamma_h^+$ and $\Gamma_h^-$, one can study the intersection of $\widetilde{\Gamma}_h^+$ and $\Gamma_h^-$. Clearly, $\angle(\widetilde{\Gamma}_h^+, \Gamma_h^-)\in (0,\pi]$.

With the above geometric setup, our argument shall be mainly confined in a neighborhood of the origin, namely $\mathbf{x}_0=\mathbf{0}$. We shall heavily rely on the Fourier expansion in terms of the radial wave functions of the solution $\bmf{u}$ to \eqref{eq:lame} around the origin, which we also fix in the following to facilitate our subsequent use. Most of the results can be conveniently found in \cite{DLW} that are summarized from \cite{DR95,SP}. To that end, throughout the rest of the paper, we set
\begin{equation}\label{eq:kpks}
	k_{{p}}=\sqrt{ \frac{\kappa }{\lambda+2 \mu} } \text { and } k_{{s}}=\sqrt{ \frac{\kappa }{\mu}},
\end{equation}
which are known as the compressional and shear wave numbers, respectively.

\begin{lem}\cite[Lemma 2.3]{DLW} \label{lem:u2 exp}
Let $J_m(t)$ be the first-kind Bessel function of the order $m\in \mathbb{N} \cup \{0\} $ and $\bmf{x}=r(\cos \varphi, \sin \varphi )^\top \in \mathbb R^2$.  Let $\bmf{u}(\bmf{x})$ be a Lam\'e eigenfunction of \eqref{eq:lame}.  The radial wave expression of $\mathbf{u}(\bmf{x})$ to \eqref{eq:lame} at the origin can be written as
\begin{equation}\label{eq:u}
 \begin{aligned}
 \mathbf{u}(\mathbf{x})=  \sum_{m=0}^{\infty} & \left\{ \frac{k_p}{2} a_{m} \mathrm{e}^{\bsi m \varphi} \left\{J_{m-1}\left(k_{p} r\right)\mathrm{e}^{-\bsi \varphi}\mathbf{e}_1
 -J_{m+1}\left(k_{p}r\right)\mathrm{e}^{\bsi \varphi}\mathbf{e}_2 \right\}\right.\\
 & + \frac{\bsi k_s}{2} b_{m} \mathrm{e}^{\bsi m \varphi} \left\{J_{m-1}\left(k_{s} r\right)\mathrm{e}^{-\bsi \varphi}\mathbf{e}_1
  +J_{m+1}\left(k_{s} r\right)\mathrm{e}^{\bsi \varphi}
  \mathbf{e}_2
   \right\} \bigg\} .
 \end{aligned}
\end{equation}
where and also throughout the rest of the paper, $\mathbf{e}_1:=(1,\bsi)^\top \mbox{ and }\mathbf{e}_2:=(1,-\bsi)^\top$.
\end{lem}

\begin{rem}
	In view of \eqref{eq:u}, we have
%\begin{equation}\label{eq:u1}
\begin{align}
u_1 \left(\bmf{x}\right)=  \sum_{m=0} ^\infty & \Big [ \frac{k_p}{2} a_m \left(\mathrm{e}^{\bsi \left(m-1\right) \varphi} J_{m-1} \left(k_p r\right) - \mathrm{e}^{\bsi \left(m+1\right) \varphi} J_{m+1} \left(k_p r\right) \right)\nonumber \\
& + \frac{\bsi k_s}{2} b_m \left(\mathrm{e}^{\bsi \left(m-1\right) \varphi} J_{m-1} \left(k_s r\right) + \mathrm{e}^{\bsi \left(m+1\right) \varphi} J_{m+1} \left(k_s r\right) \right)  \Big ],\label{eq:u1}\\
u_2 \left(\bmf{x}\right)=  \sum_{m=0} ^\infty & \Big [ \frac{\bsi k_p}{2} a_m \left( \mathrm{e}^{\bsi \left(m-1\right) \varphi} J_{m-1} \left(k_p r\right) + \mathrm{e}^{\bsi \left(m+1\right) \varphi} J_{m+1} \left(k_p r\right) \right)\nonumber  \\
& + \frac{ k_s}{2} b_m \left(-\mathrm{e}^{\bsi \left(m-1\right) \varphi} J_{m-1} \left(k_s r\right) + \mathrm{e}^{\bsi \left(m+1\right) \varphi} J_{m+1} \left(k_s r\right)  \right) \Big ].\label{eq:u2}
\end{align}
\end{rem}

By the analyticity of $\bmf{u}$ in the interior domain of $\Omega$ and  strong continuation principle, we have the following proposition.
\begin{prop}\label{prop:1}	
Suppose $\mathbf{0}\in \Omega$ and $\bmf{u}$ has the expansion \eqref{eq:u} around the origin such  that $a_m=b_m=0$ for $\forall m \in \mathbb{N}\cup \{0\}$. Then
$$
\bmf{u} \equiv \bmf{0} \mbox{ in } \Omega .
$$
\end{prop}

The corresponding Fourier representation of  the boundary traction operator $T_{\bmf{\nu }}\mathbf u \Big |_{\Gamma^\pm_h }$ defined in \eqref{eq:Tu} can be found in the following lemma.

%The proof is postponed to Appendix.

\begin{lem}\cite[Lemma 2.4]{DLW}\label{lem:Tu1 exp}
Let $\bmf{u}(\bmf{x})$ be a Lam\'e eigenfunction to \eqref{eq:lame} with the Fourier expansion \eqref{eq:u} and $\Gamma^\pm_h$ is defined in \eqref{eq:gamma_pm}. Then $ T_\nu \mathbf{u}\Big |_{\Gamma^+_h}$ possesses the following radial wave expansion at the origin
\begin{equation}\label{eq:Tu1}
 \begin{aligned}
 T_\nu \mathbf{u}\Big |_{\Gamma^+_h} & =  \sum_{m=0}^{\infty}  \left\{\frac{\bsi k_p^2}{2} a_{m} \Big[ \mathrm{e}^{\bsi \left(m-1\right) \varphi_0}  \mu  J_{m-2}(k_p r) \mathbf{e}_1 +  \mathrm{e}^{\bsi (m -1) \varphi_0}  (\lambda+\mu) J_m(k_p r) \mathbf{e}_1\right.\\
 &  - \mathrm{e}^{\bsi (m+1) \varphi_0}  \mu J_{m+2}\left(k_{p} r\right) \mathbf{e}_2  - \mathrm{e}^{\bsi( m +1)\varphi_0}  \left(\lambda+\mu\right) J_m\left(k_{p} r\right) \mathbf{e}_2\Big] \\
 &- \frac{k_s^2}{2} b_{m}\Big[  \mathrm{e}^{\bsi \left(m-1\right) \varphi_0}  \mu J_{m-2}\left(k_{s} r\right) \mathbf{e}_1 -  \mathrm{e}^{\bsi \left(m+1\right) \varphi_0}   \mu J_{m+2}\left(k_{s} r\right) \mathbf{e}_2\Big]
 \bigg\}.
 \end{aligned}
\end{equation}
Similarly, the radial wave expansion of $ T_\nu \mathbf{u}\Big |_{\Gamma^-_h}$ at the origin is given by
\begin{equation}\label{eq:Tu2}
 \begin{split}
  T_\nu \mathbf{u}\Big |_{\Gamma^-_h}&=  \sum_{m=0}^{\infty}  \left\{-\frac{\bsi k_p^2}{2} a_{m} \mu  J_{m-2}(k_p r) \mathbf{e}_1 - \frac{\bsi k_p^2}{2} a_{m}   (\lambda+\mu) J_m(k_p r) \mathbf{e}_1\right.\\
 & + \frac{k_s^2}{2} b_{m}  \mu J_{m-2}\left(k_{s} r\right) \mathbf{e}_1 + \frac{\bsi k_p^2}{2} a_{m} \left(\lambda+\mu\right) J_m\left(k_{p} r\right) \mathbf{e}_2 \\
 &+ \frac{\bsi k_p^2}{2} a_{m}   \mu J_{m+2}\left(k_{p} r\right) \mathbf{e}_2 + \frac{k_s^2}{2} b_{m}  \mu J_{m+2}\left(k_{s} r\right) \mathbf{e}_2
 \bigg\}.
 \end{split}
\end{equation}
\end{lem}

The explicit  Fourier expansions of $\mathcal B_3(\mathbf u)$ and  $\mathcal B_4(\mathbf u)$  defined in \eqref{eq:tf} are given in the following lemma.
\begin{lem}
	Let the unit normal vector $\nu $ and the tangential vector $\boldsymbol{\tau}$ to $\Gamma_h$ are defined in \eqref{eq:nutau}. Suppose that $\bmf{u}(\bmf{x})$ is  a Lam\'e eigenfunction of \eqref{eq:lame}, $\mathcal B_3(\mathbf u)$ and $\mathcal B_4(\mathbf u)$
 are defined in \eqref{eq:tf}. We have the Fourier expansions of $\mathcal B_3(\mathbf u)$  as
 \begin{equation}\label{eq:third3}
 \begin{aligned}
\mathcal B_3(\mathbf u)|_{\Gamma_h^-} &=  \sum_{m=0}^{\infty}  \bigg\{ \left[- \frac{\bsi k_p}{2} a_{m} \big (J_{m-1}(k_p r) + J_{m+1}(k_p r)\big )\right.\\
   & +  \frac{k_s}{2} b_m \big (J_{m-1}(k_s r) - J_{m+1}(k_s r) \big )\bigg]\hat{\mathbf{e}}_1 + \left[ \frac{\bsi k_p^2}{2} a_{m} \big (-J_{m-2}(k_p r)\right.\\
    &  + J_{m+2}(k_p r) \big  )
  + \frac{k_s^2}{2} b_m \big (J_{m-2}(k_s r) + J_{m+2}(k_s r) \big )\bigg] \mu \hat{\mathbf{e}}_2 \bigg\},
 \end{aligned}
\end{equation}
\vspace*{-2mm}
 \begin{equation}\label{eq:third2}
 \begin{aligned}
\mathcal B_3(\mathbf u)|_{\Gamma_h^+} &=  \sum_{m=0}^{\infty} \mathrm{ e}^{\bsi m \varphi_0} \bigg\{ \left[ \frac{\bsi k_p}{2} a_{m} \big (J_{m-1}(k_p r) + J_{m+1}(k_p r)\big )\right.\\
   & +  \frac{k_s}{2} b_m \big (-J_{m-1}(k_s r) + J_{m+1}(k_s r) \big )\bigg]\hat{\mathbf{e}}_1 + \left[ \frac{\bsi k_p^2}{2} a_{m} \big (-J_{m-2}(k_p r)\right.\\
    &  + J_{m+2}(k_p r) \big  )
  + \frac{k_s^2}{2} b_m \big (J_{m-2}(k_s r) + J_{m+2}(k_s r) \big )\bigg] \mu \hat{\mathbf{e}}_2 \bigg\},
 \end{aligned}
\end{equation}
where  and also throughout the rest of the paper,   $\hat{\mathbf{e}}_1:=(0,1)^\top $ and $\hat{\mathbf{e}}_2:=(1,0)^\top $. Furthermore, it holds that
\begin{equation}\label{eq:forth3}
 \begin{aligned}
\mathcal B_4(\mathbf u)\big |_{\Gamma_h^-} = & \sum_{m=0}^{\infty}  \bigg\{ \left[ \frac{ k_p}{2} a_{m} (  J_{m-1}(k_p r) - J_{m+1}(k_p r) )\right.\\
   & +  \frac{\bsi k_s}{2} b_m (J_{m-1}(k_s r) + J_{m+1}(k_s r) )\bigg]\hat{\mathbf{e}}_1 \\
   & + \left[- \frac{ k_p^2}{2} a_{m} (J_{m-2}(k_p r)\mu + 2(\lambda+\mu) J_m(k_p r)+ J_{m+2}(k_p r) \mu)\right.\\
   & - \frac{\bsi k_s^2}{2} b_m (J_{m-2}(k_s r) - J_{m+2}(k_s r) )\mu \bigg]  \hat{\mathbf{e}}_2 \bigg\}.
 \end{aligned}
 \vspace*{-1mm}
\end{equation}
\begin{equation}\label{eq:forth2}
 \begin{aligned}
\mathcal B_4(\mathbf u)\big |_{\Gamma_h^+} = & \sum_{m=0}^{\infty} \mathrm{ e}^{\bsi m \varphi_0} \bigg\{ \left[ \frac{ k_p}{2} a_{m} ( - J_{m-1}(k_p r) + J_{m+1}(k_p r) )\right.\\
   & -  \frac{\bsi k_s}{2} b_m (J_{m-1}(k_s r) + J_{m+1}(k_s r) )\bigg]\hat{\mathbf{e}}_1 \\
   & + \left[- \frac{ k_p^2}{2} a_{m} (J_{m-2}(k_p r)\mu + 2(\lambda+\mu) J_m(k_p r)+ J_{m+2}(k_p r) \mu)\right.\\
   & - \frac{\bsi k_s^2}{2} b_m (J_{m-2}(k_s r) - J_{m+2}(k_s r) )\mu \bigg]  \hat{\mathbf{e}}_2 \bigg\}.
 \end{aligned}
\end{equation}
\end{lem}

\begin{proof}
On $\Gamma_h^+$, we have ${\nu}=(-\sin \varphi_0,\cos \varphi_0)^\top$ and $\boldsymbol{\tau}=(-\cos \varphi_0,-\sin \varphi_0)^\top$. By virtue of  \eqref{eq:tf}, \eqref{eq:u} and \eqref{eq:Tu1}, as well as noting
$$
\nu \cdot \mathbf e_1=\bsi {\mathrm e}^{\bsi \varphi_0},\quad \nu \cdot \mathbf e_2=-\bsi {\mathrm e}^{-\bsi \varphi_0},\quad \boldsymbol{\tau} \cdot \mathbf e_1=- {\mathrm e}^{\bsi \varphi_0},\quad \boldsymbol{\tau} \cdot \mathbf e_2=- {\mathrm e}^{-\bsi \varphi_0},
$$
and by direct calculations, we can obtain \eqref{eq:third2} and \eqref{eq:forth2}. Similarly, one can derive \eqref{eq:third3} and \eqref{eq:forth3}.
\end{proof}

\begin{lem}{\cite{CDLZ}}\label{lem:co exp}
Suppose that for $0<h \ll 1$ and $t \in (0,h)$,\[\sum_{n=0}^{\infty}\alpha_{n}J_{n}(t)=0,\]
where $J_{n}\left(t\right)$ is the n-th Bessel function of the first kind. Then \[\alpha_{n}=0,\quad n=0,1,2,\ldots \]
\end{lem}

We are now in a position to present the summary of the main GHP results. Henceforth, if $\Gamma_h^+ $ is an impedance line or a generalized-impedance line associated with the parameter $\boldsymbol{\eta}_1$,  according to Definition~\ref{def:class1}, we assume that $\boldsymbol{ \eta}_1$ is given by the following absolutely convergent series at $\mathbf 0\in \Gamma_h^+$:
\begin{equation}\label{eq:eta1 ex}
	%\begin{align}
		\boldsymbol{\eta}_1=\eta_1+\sum_{j=1}^\infty \eta_{1,j} r^j,  
		%\quad %\label{eq:eta1 ex} \\
	%	\boldsymbol{\eta}_2=\eta_{2}+\sum_{j=1}^\infty \eta_{2,j}(\theta) r^j \label{eq:eta2 ex}
	%\end{align}
\end{equation}
    where $\eta_{1}\in\mathbb{C}\backslash\{0\}$, $\eta_{1,j}\in \mathbb C$ and $r\in [0,h]$. Similarly, when $\Gamma_h^-$ is an impedance line or a generalized-impedance line associated with the parameter $\boldsymbol{\eta}_2$, we suppose that $\boldsymbol{\eta}_2$ is given by the following absolutely convergent series at $\mathbf 0\in \Gamma_h^+$:
\begin{equation}\label{eq:eta2 ex}
	%\begin{align}
		\boldsymbol{\eta}_2=\eta_2+\sum_{j=1}^\infty \eta_{2,j} r^j,  
		%\quad %\label{eq:eta1 ex} \\
	%	\boldsymbol{\eta}_2=\eta_{2}+\sum_{j=1}^\infty \eta_{2,j}(\theta) r^j \label{eq:eta2 ex}
	%\end{align}
\end{equation}
    where $\eta_{2}\in\mathbb{C}\backslash\{0\}$, $\eta_{2,j}\in \mathbb C$ and $r\in [0,h]$. 
 
We first present four theorems on the GHP where only $\mathcal{B}_1, \mathcal{B}_2$ and $\mathcal{B}_5$ are involved. They are actually extended from our paper \cite{DLW}, where the boundary impedance parameters were always assumed to be constant, but in this paper, they can be variable functions as introduced in Definition~\ref{def:class1}. 
    
\begin{thm}\label{thm:impedance line}
Let $\mathbf{u}\in L^2(\Omega)^2$ be a solution to \eqref{eq:lame}. If there exits a singular impedance line $\Gamma_h\subset\Omega$ of $\bmf{u}$ with an impedance parameter $\boldsymbol{ \eta}  \in \mathcal{A}(\Gamma_h ) $, then $\bmf{u}\equiv \bmf{ 0}$.
\end{thm}

\begin{thm}\label{thm:54}
Let $\mathbf{u}\in L^2(\Omega)^2$ be a solution to \eqref{eq:lame}. If there exist two intersecting lines $\Gamma_h^\pm$ of $\bmf{u}$ such that $\Gamma_h^-$ is a rigid line and $\Gamma_h^+$ is an impedance line with the intersecting angle $\angle(\Gamma_h^+,\Gamma_h^-)=\varphi_0\neq \pi $, where the associated impedance parameter satisfies  \eqref{eq:eta1 ex}, then $\bmf{u}\equiv \bmf{0}$.

%$
%T_{\bmf{\nu}}\bmf{u}+\eta_2 \bmf{u}=\bmf{0} \mbox{ on } \Gamma_h^+$ with a nonzero impedance constant $\eta_2\in \mathbb C$. 
%Let $\mathbf{u}$ be a lam\'e eigenfunction of \eqref{eq:lame} . Suppose that a impedance line $\Gamma_h^+$ intersects with a rigid line $\Gamma_h^-$ at $\mathbf{0}$ with the angle $\angle(\Gamma_h^+,\Gamma_h^-)=\varphi_0(\neq\pi)$. Then $a_m=b_m=0$ for $m \in \mathbb{N}\cup \{0\}$.
\end{thm}

\begin{thm}\label{thm:55}
Let $\mathbf{u}\in L^2(\Omega)^2$ be a solution to \eqref{eq:lame}. Suppose there exit two intersecting lines $\Gamma_h^\pm$ of $\bmf{u}$ such that $\Gamma_h^-$ is a traction-free line and $\Gamma_h^+$ is an impedance line associated with an impedance parameter $\boldsymbol{\eta}_2$ satisfying \eqref{eq:eta1 ex}, 
 with the property that $\angle(\Gamma_h^+,\Gamma_h^-)=\varphi_0\neq \pi $ and $\bmf{u}$ vanishes at the intersecting point, namely $\bmf{u}(\bmf{0})=\bmf{0}$, then $\bmf{u}\equiv \bmf{0}$. 
%Let $\mathbf{u}$ be a Lam\'e eigenfunction of \eqref{eq:lame}. Suppose that a impedance line $\Gamma_h^+$ intersects with a traction-line $\Gamma_h^-$ at  $\mathbf{0}$ with the angle $\angle(\Gamma_h^+,\Gamma_h^-)=\varphi_0(\neq\pi)$, Then $a_m=b_m=0$ for $m \in \mathbb{N}\cup \{0\}$.
\end{thm}

\begin{thm}\label{eq:impedance line exp}
Let $\mathbf{u}\in L^2(\Omega)^2$ be a solution to \eqref{eq:lame}. Suppose there exist two impedance lines  $\Gamma_h^\pm$ of $\bmf{u}$ such that $\angle(\Gamma_h^+,\Gamma_h^-)=\varphi_0\neq \pi$ satisfying that \begin{equation}\label{eq:lem53 cond}
	 \eta_1 \mathrm{e}^{-\bsi \varphi_0}+ \eta_2 \neq 0
\end{equation} is fulfilled and $\bmf{u}$ vanishes at the intersecting point, i.e. $\bmf{u}(\bmf{0})=\bmf{0}$, then $\bmf{u}\equiv \bmf{0}$. Here, $T_\nu \bmf{u}+\boldsymbol{ \eta}_1\bmf{u}=\mathbf{0}$ on $\Gamma_h^+$, $T_\nu \bmf{u}+\boldsymbol{ \eta}_2\bmf{u}=\mathbf{0}$ on $\Gamma_h^-$  with $\boldsymbol{ \eta}_1 \in \mathcal A(\Gamma_h^+ )$
 and  $\boldsymbol{ \eta}_2 \in \mathcal A(\Gamma_h^- )$.

%Suppose that there are two impedance lines $\Gamma_h^+:T\bmf{u}+\eta_2\bmf{u}=\mathbf{0}$ and $\Gamma_h^-:T\bmf{u}+\eta_1\bmf{u}=\mathbf{0}$. If the following condition is fulfilled $,
% Then $a_m=b_m=0$ for  $m \in \mathbb{N}\cup \{0\}$,
% where $\Gamma_h^+ \cap\Gamma_h^-=\left\{\mathbf{0}\right\} \in \Omega $.
\end{thm}

In addition to those from Theorems~\ref{thm:impedance line}--\ref{eq:impedance line exp}, there are two more results from \cite{DLW} stating that the existence of a singular rigid 
or traction-free line also ensures $\mathbf{u}\equiv\mathbf{0}$. Two remarks are in order. First, as mentioned earlier, those theorems are extended and generalised from our earlier paper \cite{DLW}. The proofs of these theorems in such as a generalised situation can be done by following the new strategies developed in this paper for the more general cases together with the arguments from \cite{DLW}. In order to save the space of the paper, we shall skip the corresponding proofs. Second, as can be seen from Theorems~\ref{thm:impedance line}--\ref{eq:impedance line exp}, if there are two line segments involved, we always assume that they are not collinear. Indeed, as discussed earlier, in the degenerate case where the two lines are collinear, the corresponding study is either reduced to the single line case or the standard HUP. This remark also applies to our subsequent study for the other cases. 

By including the additional boundary trace conditions $\mathcal{B}_3, \mathcal{B}_4$ and $\mathcal{B}_6$ into the study, we can establish the GHP results in a similar flavour to Theorems~\ref{thm:impedance line}--\ref{eq:impedance line exp}. We summarize the corresponding findings in Table \ref{table:main results}, where the first column lists the presence of one or two homogeneous line segments, the second column lists the required additional conditions and the third column is the GHP conclusions.

\begin{table} [!htbp]
	\centering
	{	
		\begin{tabular}{|c|c|c|}
			\hline
		 	&\tabincell{c}{ Assumptions} & Conclusions\\
			\hline
	$\Gamma_h \in  {\mathcal S} (\mathcal R^\kappa_\Omega )$ & / & $\mathbf u\equiv \mathbf 0$  \\
			\hline
			$\Gamma_h \in  {\mathcal S} (\mathcal T^\kappa_\Omega )$ &/ & $\mathbf u\equiv \mathbf 0$ \\
			\hline
			$\Gamma_h \in  {\mathcal S} (\mathcal I^\kappa_\Omega )$&/& $\mathbf u\equiv \mathbf 0$  \\
			\hline
			$\Gamma_h \in  {\mathcal S} (\mathcal G^\kappa_\Omega )$&/ & $\mathbf u\equiv \mathbf 0$ \\
			\hline
			$\Gamma_h \in  {\mathcal S} (\mathcal F^\kappa_\Omega )$&/ & $\mathbf u\equiv \mathbf 0$ \\
			\hline
	$ \Gamma_h^+ \in   \mathcal H^\kappa_\Omega $& \tabincell{c}{%$\mathbf 0\in \Gamma_h$ such that 
	${\nu}^\top \nabla \bmf{u} {\nu} |_{\bmf{x}={\bmf{0}} }= \boldsymbol{\tau}^\top \nabla \bmf{u} {\nu} |_{\bmf{x}={\bmf{0}} }=0$; \\ $b_0=b_1=b_2=0$;\\ $\eta_1 \neq \pm \bsi \mbox{ and } \frac{\pm \sqrt{(\lambda + 3\mu)(\lambda + \mu)} - \mu \bsi}{\lambda + 2 \mu}$ }& $\mathbf u\equiv \mathbf 0$\\
			\hline
			$\Gamma_h^+,\Gamma_h^- \in   \mathcal R^\kappa_\Omega $& $\angle(\Gamma_h^+,\Gamma_h^-)\neq \pi $& $\mathbf u\equiv \mathbf 0$\\
			\hline
			$\Gamma_h^+,\Gamma_h^- \in   \mathcal T^\kappa_\Omega $& $\angle(\Gamma_h^+,\Gamma_h^-)\neq \pi $ and $\mathbf u(\mathbf 0)=\mathbf 0$& $\mathbf u\equiv \mathbf 0$\\
			\hline
			$\Gamma_h^-\in   \mathcal R^\kappa_\Omega,\Gamma_h^+ \in   \mathcal T^\kappa_\Omega $& $\angle(\Gamma_h^+,\Gamma_h^-)\neq \pi $ & $\mathbf u\equiv \mathbf 0$\\
			\hline
			$\Gamma_h^-\in   \mathcal R^\kappa_\Omega,\Gamma_h^+ \in   \mathcal I^\kappa_\Omega $& $\angle(\Gamma_h^+,\Gamma_h^-)\neq \pi $ & $\mathbf u\equiv \mathbf 0$\\
			\hline
			$\Gamma_h^-\in   \mathcal I^\kappa_\Omega,\Gamma_h^+ \in   \mathcal I^\kappa_\Omega $& \tabincell{c}{$\angle(\Gamma_h^+,\Gamma_h^-)=\varphi_0\neq \pi $, $\mathbf u(\mathbf 0)=\mathbf 0$; \\ $\eta_1 \mathrm{e}^{-\bsi \varphi_0}+ \eta_2 \neq 0$ } & $\mathbf u\equiv \mathbf 0$\\
			\hline
			$\Gamma_h^-\in   \mathcal 
			R^\kappa_\Omega,\Gamma_h^+ \in   \mathcal G^\kappa_\Omega $& $\angle(\Gamma_h^+,\Gamma_h^-)\neq \pi $ & $\mathbf u\equiv \mathbf 0$\\
			\hline
%			$\Gamma_h^-\in   \mathcal 
%			G^\kappa_\Omega,\Gamma_h^+ \in   \mathcal G^\kappa_\Omega $& $\angle(\Gamma_h^+,\Gamma_h^-)\neq \pi $, $a_0=0$ & $\mathbf u\equiv \mathbf 0$\\
%			\hline
			$\Gamma_h^-\in   \mathcal 
			T^\kappa_\Omega,\Gamma_h^+ \in   \mathcal G^\kappa_\Omega $& $\angle(\Gamma_h^+,\Gamma_h^-)\neq \pi $ & $\mathbf u\equiv \mathbf 0$\\
			\hline
			$\Gamma_h^-\in   \mathcal 
			I^\kappa_\Omega,\Gamma_h^+ \in   \mathcal G^\kappa_\Omega $& $\angle(\Gamma_h^+,\Gamma_h^-)\neq \pi $ & $\mathbf u\equiv \mathbf 0$\\
			\hline
			$\Gamma_h^-\in   \mathcal 
			R^\kappa_\Omega,\Gamma_h^+ \in   \mathcal F^\kappa_\Omega $& $\angle(\Gamma_h^+,\Gamma_h^-)\neq \pi $ & $\mathbf u\equiv \mathbf 0$\\
			\hline
%			$\Gamma_h^-\in   \mathcal 
%			F^\kappa_\Omega,\Gamma_h^+ \in   \mathcal F^\kappa_\Omega $& $\angle(\Gamma_h^+,\Gamma_h^-)\neq \pi $, $b_0=0$ & $\mathbf u\equiv \mathbf 0$\\
%			\hline
			$\Gamma_h^-\in   \mathcal 
			I^\kappa_\Omega,\Gamma_h^+ \in   \mathcal F^\kappa_\Omega $& $\angle(\Gamma_h^+,\Gamma_h^-)\neq \pi $ & $\mathbf u\equiv \mathbf 0$\\
			\hline
			$\Gamma_h^-\in   \mathcal 
			I^\kappa_\Omega,\Gamma_h^+ \in   \mathcal F^\kappa_\Omega $& $\angle(\Gamma_h^+,\Gamma_h^-)\neq \pi $ & $\mathbf u\equiv \mathbf 0$\\
			\hline
			$\Gamma_h^-\in   \mathcal 
			R^\kappa_\Omega,\Gamma_h^+ \in   \mathcal H^\kappa_\Omega $&\tabincell{c}{$\angle(\Gamma_h^+,\Gamma_h^-)\neq \pi $; either $\eta_1\neq -\frac{\bsi \mu \mathrm{e}^{2 \bsi \varphi_0}}{\lambda+\mu(1+\mathrm{e}^{2 \bsi \varphi_0})}$ \\ or $a_0= 0, \quad \eta_1\neq \mathrm i$} & $\mathbf u\equiv \mathbf 0$\\
			\hline
			$\Gamma_h^-\in   \mathcal 
			T^\kappa_\Omega,\Gamma_h^+ \in   \mathcal H^\kappa_\Omega $& $\angle(\Gamma_h^+,\Gamma_h^-)\neq \pi $ and $\eta_1 \neq \mathrm i$ & $\mathbf u\equiv \mathbf 0$\\
			\hline
			$\Gamma_h^-\in   \mathcal 
			I^\kappa_\Omega,\Gamma_h^+ \in   \mathcal H^\kappa_\Omega $& $\angle(\Gamma_h^+,\Gamma_h^-)\neq \pi $ and $\eta_2 \neq \mathrm i$ & $\mathbf u\equiv \mathbf 0$\\
			\hline
			$\Gamma_h^-\in   \mathcal 
			G^\kappa_\Omega,\Gamma_h^+ \in   \mathcal F^\kappa_\Omega $& $\angle(\Gamma_h^+,\Gamma_h^-)\neq \pi $ & $\mathbf u\equiv \mathbf 0$\\
			\hline
			$\Gamma_h^-\in   \mathcal 
			G^\kappa_\Omega,\Gamma_h^+ \in   \mathcal H^\kappa_\Omega $& $\angle(\Gamma_h^+,\Gamma_h^-)\neq \pi $ & $\mathbf u\equiv \mathbf 0$\\
			\hline
			$\Gamma_h^-\in   \mathcal 
			H^\kappa_\Omega,\Gamma_h^+ \in   \mathcal H^\kappa_\Omega $& $\angle(\Gamma_h^+,\Gamma_h^-)\neq \pi $, $\eta_1=\eta_2=\eta\neq \pm \bsi,-\frac{\bsi m}{m+2}$ for $\forall m\in \mathbb N$ & $\mathbf u\equiv \mathbf 0$\\
			\hline
			$\Gamma_h^-\in   \mathcal 
			H^\kappa_\Omega,\Gamma_h^+ \in   \mathcal F^\kappa_\Omega $& $\angle(\Gamma_h^+,\Gamma_h^-)\neq \pi $ & $\mathbf u\equiv \mathbf 0$\\
			\hline
		\end{tabular}
	}
	\medskip
	\caption{}
	\label{table:main results}
\end{table}

\section{GHP with the presence of a single singular line}\label{sect:3}

In this section, we prove that if $\Omega$ contains  a singular soft-clamped, simply-supported or generalized-impedance  line $\Gamma_h$ as introduced in Definition~\ref{def:2} of a generic Lam\'e eigenfunction $\bmf{u}$ to \eqref{eq:lame},  then $\bmf{u}$ is identically zero. According to our discussion made at the beginning of Section~\ref{sect:2}, we can assume that the point $\bmf{x}_0$ involved in Definition~\ref{def:2} is the origin, namely $\bmf{x}_0=\bmf{0}$. It is clear that the unit normal vectors $ \nu $  to $\Gamma_h^+$ as defined in \eqref{eq:gamma_pm} is $\pm (-\sin  \varphi_0,\cos  \varphi_0)^\top$. In this paper, we choose $\nu=(-\sin  \varphi_0,\cos  \varphi_0)^\top $. In such a case, the following conditions involved in Definition~\ref{def:2}
\begin{equation}\notag
% \bmf{u}(\bmf{x}_0 )=\bmf{0} , %\,   \partial_{ \nu} \bmf{u} |_{\bmf{x}={\bmf{x}}_0 }=\bmf{0},
% \,
  {\nu}^\top \nabla \bmf{u} {\nu} |_{\bmf{x}={\bmf{x}}_0 }=0 ,\,  \boldsymbol{\tau}^\top \nabla \bmf{u} {\nu} |_{\bmf{x}={\bmf{x}}_0 }=0
\end{equation}
turn out to be
\begin{equation}\label{eq:32 cond}
% \bmf{u}(\bmf{0} )=\bmf{0} , %\,  \partial_{\boldsymbol \nu} \bmf{u}  |_{ \mathbf x=\bmf{0} }=\bmf{0},
% \,
{\nu}^\top \nabla \bmf{u} \boldsymbol{\nu}|_{ \mathbf x=\bmf{0} }=0 , \, \boldsymbol{\tau}^\top \nabla \bmf{u} {\nu} |_{ \mathbf x=\bmf{0} }=0 ,
\end{equation}
where $\boldsymbol{\nu}=(-\sin {\varphi_0},\cos{\varphi_0})^\top$ and  $\boldsymbol{\tau}=(-\cos {\varphi_0},-\sin{\varphi_0})^\top$.

Similar to Definition \ref{def:2}, we introduce the {\it singular generalized-impedance line} of $\mathbf u$ in the following definition.
\begin{defn}\label{def:3}
Let $\mathbf{0}\in  \Gamma_h$ be a generalized-impedance line of $\mathbf u$ defined in Definition \ref{def:1}, where $\mathbf u$ has the Fourier expansion \eqref{eq:u}  at the origin $\mathbf 0$ with the coefficients $a_\ell$ and $b_\ell$, $\ell \in \mathbb N \cup \{0\}$. If  the Fourier coefficients satisfy
\begin{equation}\label{eq:def3}
	b_0=b_1=b_2=0,
\end{equation}
and
\begin{equation}\label{eq:sg cond}
{\nu}^\top \nabla \bmf{u} {\nu}|_{ \mathbf x=\bmf{0} }=\boldsymbol{\tau}^\top \nabla \bmf{u} {\nu} |_{ \mathbf x=\bmf{0} }=0 ,
\end{equation}
 we call $\Gamma_h$ a {\it singular generalized-impedance line} of $\mathbf u$. The set of the singular generalized-impedance line is denoted by $	{\mathcal S} \left({\mathcal H}_\Omega^{\kappa}\right).$
\end{defn}

In the following lemma, we characterize the algebraic relationship between the Fourier coefficients of \eqref{eq:u} from \eqref{eq:32 cond}, whose proof is postponed to the  Appendix.

\begin{lem}\label{lem:condition}
Let $\mathbf{u}$ be a Lam\'e eigenfunction of \eqref{eq:lame}, where   $\mathbf{u}$ has the expansion \eqref{eq:u} around the origin.  Consider the condition \eqref{eq:32 cond}  on $\Gamma_h^+$ with ${\nu}=(-\sin {\varphi_0},\cos{\varphi_0})^\top$ and  $\boldsymbol{\tau}=(-\cos {\varphi_0},-\sin{\varphi_0})^\top$. Then
\begin{subequations}
	\begin{align}
		{\nu}^\top \nabla \bmf{u} {\nu} |_{ \mathbf x=\bmf{0} }=0\ \ &\mbox{ implies }\ \ 2 k_p^2 a_0 + \mathrm e^{ 2 \bsi \varphi_0}(k_p^2 a_2 + \bsi k_s^2 b_2)=0,\label{eq:gradient1} \\
		\boldsymbol{\tau}^\top \nabla \bmf{u} {\nu}|_{ \mathbf x=\bmf{0} }=0\ \ &\mbox{ implies }\ \
2 \bsi k_s^2 b_0 + \mathrm e^{ 2 \bsi \varphi_0}(k_p^2 a_2 + \bsi k_s^2 b_2)=0.  \label{eq:gradient2}
	\end{align}
\end{subequations}

%Suppose that $\boldsymbol{\nu}^\top \nabla \bmf{u} \boldsymbol{\nu} |_{ \mathbf x=\bmf{0} }=0$, where $\boldsymbol{\nu}=(-\sin {\varphi_0},\cos{\varphi_0})^\top$, then we have
%\begin{equation}\label{eq:gradient1}
%2 k_p^2 a_0 + \mathrm e^{ 2 \bsi \varphi_0}(k_p^2 a_2 + \bsi k_s^2 b_2)=0,
%\end{equation}
% Suppose that $\boldsymbol{\tau}^\top \nabla \bmf{u} \boldsymbol{\nu}|_{ \mathbf x=\bmf{0} }  =0$, where $\boldsymbol{\nu}=(-\sin {\varphi_0},\cos{\varphi_0})^\top$ and  $\boldsymbol{\tau}=(-\cos {\varphi_0},-\sin{\varphi_0})^\top$, then we have
%\begin{equation}\label{eq:gradient2}
%2 \bsi k_s^2 b_0 + \mathrm e^{ 2 \bsi \varphi_0}(k_p^2 a_2 + \bsi k_s^2 b_2)=0,
%\end{equation}
%Suppose that $\partial_{\nu} \bmf{u}|_{ \mathbf x=\bmf{0} }  =\bmf{0}$, then we have
%\begin{equation}\label{eq:gradient3}
%\bsi k_p^2 a_0 - k_s^2 b_0=0.
%\end{equation}

\end{lem}

\subsection{The case with a singular soft-clamped line}
\begin{lem}\label{lem:third}
Let $\mathbf{u}$ be a Lam\'e eigenfunction of \eqref{eq:lame}, where   $\mathbf{u}$ has the expansion \eqref{eq:u} around the origin.  Suppose that $\Gamma_h^+$ defined in \eqref{eq:gamma_pm} such that $\Gamma_h^+\in {\mathcal G}_\Omega^{\kappa}$.  Then we have
\begin{equation}\label{eq:1}
    \bsi k_p a_1 -  k_s b_1=0,\quad
 \bsi k_p^2 a_2 - k_s^2 b_2 = 0, 
\end{equation}
and
\begin{equation}\label{eq:2}
  2 k_s ^2 b_0+ \mathrm{e}^{2 \bsi \varphi_0} (\bsi k_p^2 a_2- k_s^2  b_2)=0,\quad
   \bsi k_p ^3 a_1 - k_s^3 b_1 - \mathrm{e}^{2 \bsi \varphi_0}(\bsi k_p ^3 a_3 - k_s^3 b_3)=0.
\end{equation}
Moreover, it holds that
\begin{equation}\label{eq:a1b1}
	b_0=a_1=b_1=0.
\end{equation}
Furthermore, suppose that
\begin{equation}\label{eq:aLbL}
	a_\ell=b_\ell=0
\end{equation}
where $\ell=0,\ldots,m-1$ and $m\in {\mathbb N}$ with $m\geq 3 $, then
\begin{align}
 \label{eq:lem31}
	a_\ell=b_\ell=0, \quad \forall \ell \in \mathbb N \cup \{0\}.
	%\mathbf u\equiv \mathbf 0 \mbox{ in } \Omega.
\end{align}
\end{lem}

\begin{proof}
Since $\Gamma_h^+$ is a soft-clamped line of $\bmf{u}$, by virtue of  \eqref{eq:third2}, we have
\begin{equation}\label{eq:third}
 \begin{aligned}
\bmf{0} = & \sum_{m=0}^{\infty} \mathrm{ e}^{\bsi m \varphi_0} \bigg\{ \left[ \frac{\bsi k_p}{2} a_{m} (J_{m-1}(k_p r) + J_{m+1}(k_p r) )\right.\\
   & +  \frac{k_s}{2} b_m (-J_{m-1}(k_s r) + J_{m+1}(k_s r) )\bigg]\hat{\mathbf{e}}_1 + \left[ \frac{\bsi k_p^2}{2} a_{m} (-J_{m-2}(k_p r) + J_{m+2}(k_p r) )\right.\\
   & + \frac{k_s^2}{2} b_m (J_{m-2}(k_s r) + J_{m+2}(k_s r) )\bigg] \mu \hat{\mathbf{e}}_2 \bigg\},\quad 0 \leqslant r\leqslant h,
 \end{aligned}
\end{equation}
where  and also throughout the rest of the paper,  $\hat{\mathbf{e}}_1=(0,1)^\top $, $\hat{\mathbf{e}}_2=(1,0)^\top $. Noting $J_{-m}(t)=(-1)^m J_{m}(t)$  (cf. \cite{Abr}), we obtain that
\begin{equation}\label{eq:Jmfact}
	J_{-m}(k_pr)=-J_{m}(k_p r), \quad J_{-m }(k_s r)=-J_{m}(k_s r),\quad  m =1,2.
\end{equation}
 Using  Lemma \ref{lem:co exp} and \eqref{eq:Jmfact}, comparing the coefficient of the term $r^0$ in both sides of \eqref{eq:third}, it holds that
$$
(\bsi k_p a_1 - k_s b_1)  \hat{\mathbf{e}}_1 + ( - \bsi k_p^2 a_2 + k_s^2 b_2) \mu \mathrm{e}^{ \bsi \varphi_0} \hat{\mathbf{e}}_2= \bmf{0}.
$$
Since $\hat{\mathbf{e}}_1$ and $\hat{\mathbf{e}}_2$ are linear independent, we obtain \eqref{eq:1}. Similarly, from Lemma \ref{lem:co exp}, we compare the coefficient of the term $r^1$ in both sides of \eqref{eq:third}, then it holds that
$$
\big[2 k_s^2 b_0 + (\bsi k_p^2 a_2 - k_s^2 b_2) \mathrm{e}^{2 \bsi \varphi_0} \big] \hat{\mathbf{e}}_1 + \big[ ( \bsi k_p^3 a_1 - k_s^3 b_1)-( \bsi k_p^3 a_1 - k_s^3 b_1) \mu \mathrm{e}^{ 2 \bsi \varphi_0}\big] \mu  \mathrm{e}^{ \bsi \varphi_0} \hat{\mathbf{e}}_2 = \bmf{0}.
$$
Since $\hat{\mathbf{e}}_1$ and $\hat{\mathbf{e}}_2$ are linear independent, we readily obtain \eqref{eq:2}.

It can be seen that $b_0=0$ by utilizing  the second equation of \eqref{eq:1} and  the first equation of \eqref{eq:2}.

Next we compare the coefficient of the term $r^2$ in both sides of \eqref{eq:third} and can conclue that
\begin{equation}\nonumber
\begin{aligned}
& \big[( \bsi k_p^3 a_1 + 3 k_s^3 b_1) + ( \bsi k_p^3 a_3 - k_s^3 b_3) \mathrm{e}^{ 2 \bsi \varphi_0}\big]   \hat{\mathbf{e}}_1+ \big[( 2 \bsi k_p^4 a_2 - 2 k_s^4 b_2)\\
&\quad - ( \bsi k_p^4 a_4  - k_s^4 b_4) \mathrm{e}^{ 2 \bsi \varphi_0}\big]  \mu \mathrm{e}^{  \bsi \varphi_0} \hat{\mathbf{e}}_2= \bmf{0}.
\end{aligned}
\end{equation}
Since $\hat{\mathbf{e}}_1$ and $\hat{\mathbf{e}}_2$ are linear independent, we obtain
\begin{equation}\label{eq:3g}
\bigg\{
    \begin{array}{l}
   \bsi k_p^3 a_1 + 3 k_s^3 b_1 + ( \bsi k_p^3 a_3 - k_s^3 b_3) \mathrm{e}^{ 2 \bsi \varphi_0}=0,\\
    2 \bsi k_p^4 a_2 - 2 k_s^4 b_2 - ( \bsi k_p^4 a_4 - k_s^4 b_4) \mathrm{e}^{ 2 \bsi \varphi_0}=0.
 \end{array}
\end{equation}
%by using the fact $\hat{\mathbf{e}}_1$ and $\hat{\mathbf{e}}_2$ are linear independent.
Substituting  the second equation of \eqref{eq:2} into the first equation of \eqref{eq:3g}, one has
 \begin{equation}\label{eq:4g}
 \bsi k_p^3 a_1 +  k_s^3 b_1 = 0.
 \end{equation}
By the first equation of \eqref{eq:1} and \eqref{eq:4g}, it yields that
\begin{equation}\label{eq:3}
 \bsi k_p a_1 - k_s b_1=0,\quad
 \bsi k_p^3 a_1- k_s^3 b_1=0.
\end{equation}
Since the determinant of the coefficient matrix of \eqref{eq:3} is
$$
\left| \begin{matrix}
	\bsi k_p & - k_s\\
	\bsi k_p^3 & - k_s^3
\end{matrix} \right|=\bsi k_p k_s (k_p^2 - k_s^2) \neq 0,
$$
we can readily prove \eqref{eq:a1b1}.

Now we suppose that \eqref{eq:aLbL}  holds. Substituting \eqref{eq:aLbL} into \eqref{eq:third}, it yields that
\begin{equation}\label{eq:third1}
 \begin{aligned}
 \bmf{0} = & \sum_{m=\ell}^{\infty} \mathrm{ e}^{\bsi m \varphi_0} \bigg\{ \left[ \frac{\bsi k_p}{2} a_{m} (J_{m-1}(k_p r) + J_{m+1}(k_p r) )\right.\\
   & +  \frac{k_s}{2} b_m (-J_{m-1}(k_s r) + J_{m+1}(k_s r) )\bigg]\hat{\mathbf{e}}_1 + \left[ \frac{\bsi k_p^2}{2} a_{m} (-J_{m-2}(k_p r) + J_{m+2}(k_p r) )\right.\\
   & + \frac{k_s^2}{2} b_m (J_{m-2}(k_s r) + J_{m+2}(k_s r) )\bigg] \mu \hat{\mathbf{e}}_2 \bigg\},\quad 0 \leqslant r\leqslant h.
 \end{aligned}
\end{equation}
The lowest order of the power with respect to $r$ in the right hand of \eqref{eq:third1} is $m-1$. Comparing the coefficients of the terms $r^{m-1}$ in both sides of \eqref{eq:third1}, it holds that
$$
  ( \bsi k_p^m a_m - k_s^m b_m )  \hat{\mathbf{e}}_1 + ( - \bsi k_p^{m+1} a_{m+1} + k_s^{m+1} b_{m+1}) \mathrm{e}^{ \bsi \varphi_0} \hat{\mathbf{e}}_2= \bmf{0}.
$$
Since $\hat{\mathbf{e}}_1$ and $\hat{\mathbf{e}}_2$ are linearly independent, we obtain that
\begin{equation}\label{eq:4}
 \bsi k_p^m a_m - k_s^m b_m=0,\quad
  \bsi k_p^{m+1} a_{m+1} - k_s^{m+1} b_{m+1}=0.
\end{equation}
Comparing the coefficients of the terms $r^m$ in both sides of \eqref{eq:third1},  it holds that
\begin{equation}\nonumber
\begin{aligned}
& (\bsi k_p^{m+1} a_{m+1} - k_s^{m+1} b_{m+1}) \mathrm{e}^{\bsi \varphi_0} \hat{\mathbf{e}}_1\\
&\quad +\left(\bsi m k_p^{m+2} a_m - m k_s^{m+2} b_m - (\bsi  k_p^{m+2} a_{m+2} -  k_s^{m+2} b_{m+2})  \mathrm{e}^{2 \bsi \varphi_0}\right)\mu \hat{\mathbf{e}}_2= \bmf{0},
\end{aligned}
\end{equation}
which further gives that
%Since $\hat{\mathbf{e}}_1$ and $\hat{\mathbf{e}}_2$ are linear independent, we obtain that
\begin{equation}\label{eq:5}
 \bigg\{  \begin{array}{l} \bsi k_p^{m+1} a_{m+1} - k_s^{m+1} b_{m+1}=0,\\
 \bsi m k_p^{m+2} a_m - m k_s^{m+2} b_m - ( \bsi k_p^{m+2} a_{m+2} - k_s^{m+2} b_{m+2})\mathrm{e}^{2 \bsi \varphi_0}=0.
  \end{array}
\end{equation}
Again  comparing the coefficient of the term $r^{m+1}$ in both sides of \eqref{eq:third1}, then it holds that
\begin{equation}\nonumber
\begin{aligned}
& \left( -\bsi m k_p^{m+2} a_m + (m+2) k_s^{m+2} b_m + (\bsi k_p^{m+2} a_{m+2} - k_s^{m+2} b_{m+2}) \mathrm{e}^{ 2 \bsi \varphi_0}\right) \hat{\mathbf{e}}_1\\
& + \left( \bsi (m+1) k_p^{m+3} a_{m+1} -(m+1) m k_s^{m+3} b_{m+1} -  ( \bsi k_p^{m+3} a_{m+3} - k_s^{m+3} b_{m+3})  \mathrm{e}^{2 \bsi \varphi_0} \right) \mu \mathrm{e}^{ \bsi \varphi_0} \hat{\mathbf{e}}_2 \\
& = \bmf{0}.
\end{aligned}
\end{equation}
%Since $\hat{\mathbf{e}}_1$ and $\hat{\mathbf{e}}_2$ are linear independent, we obtain that
Therefore one has
\begin{equation}\label{eq:6}
 \bigg\{  \begin{array}{l} -\bsi m k_p^{m+2} a_m + (m+2) k_s^{m+2} b_m + (\bsi k_p^{m+2} a_{m+2} - k_s^{m+2} b_{m+2})\mathrm{e}^{2 \bsi \varphi_0}=0,\\
 \bsi (m+1) k_p^{m+3} a_{m+1} -(m+1)  k_s^{m+3} b_{m+1} - ( \bsi k_p^{m+3} a_{m+3} - k_s^{m+3} b_{m+3})\mathrm{e}^{2 \bsi \varphi_0}=0.
  \end{array}
\end{equation}
Combining the second equation of \eqref{eq:5} with the first equation of \eqref{eq:6}, it can be derived that
\begin{equation}\label{eq:7}
b_m=0.
\end{equation}
Substituting  \eqref{eq:7} into the first equation of \eqref{eq:4}, we can obtain
\begin{equation}\label{eq:8}
a_m=0.
\end{equation}
Repeating   the above procedure for \eqref{eq:7} and  \eqref{eq:8} step by step,   we can prove that $a_\ell=b_\ell=0$ for $\ell=m+1,\ldots$.
\end{proof}

\begin{thm}\label{thm:third}
Let $\mathbf{u}\in L^2(\Omega)^2$ be a solution to \eqref{eq:lame}. If there exits a singular soft-clamped line $\Gamma_h \subset \Omega$ corresponding to $\mathbf u$, then $\bmf{u}\equiv \bmf{ 0}$.
\end{thm}

\begin{proof}
Suppose that there exits a singular soft-clamped line $\Gamma_h$ of $\mathbf u$ as described at the beginning of this section. Therefore we have  ${\nu}^\top \nabla \bmf{u} {\nu}|_{ \mathbf x=\bmf{0} } =0$, which implies that \eqref{eq:gradient1} holds.

 Since $\Gamma_h$ is a soft-clamped line, from Lemma \ref{lem:third} we know that
	\begin{equation}\label{eq:9}
			a_1=b_1=0,
	\end{equation}
and \eqref{eq:1} and \eqref{eq:2} hold.
In view of the second equation of \eqref{eq:1} and \eqref{eq:gradient1}, one has
\begin{equation}\label{eq:9z}
 \bsi k_p^2 a_2 - k_s^2 b_2=0,\quad
  2 \bsi k_s^2 a_0 + k_p^2 a_2 + \bsi k_s^2 b_2=0,
\end{equation}
which implies
\begin{equation}\label{eq:10}
	a_0=0.
\end{equation}
Substituting \eqref{eq:10} into the first equation of \eqref{eq:2}, it yields that $\bsi k_p^2 a_2 - k_s^2 b_2=0$, which can be furthered substituted into  the second equation of \eqref{eq:1} for deriving $b_0=0$.

%and the first equation of \eqref{eq:2}, that is
%\begin{equation}\label{eq:10z}
% \bigg\{  \begin{array}{l} \bsi k_p^2 a_2 - k_s^2 b_2=0,\\
%  2 k_s^2 b_0 + \bsi k_p^2 a_2 -  k_s^2 b_2=0,
%  \end{array}
%\end{equation}
%  we have
%\begin{equation}\label{eq:11}
%	b_0=0.
%\end{equation}

Since we have proved that $a_\ell=b_\ell=0$ for $\ell=0,1$, the lowest order with respect to the power of $r$ in \eqref{eq:third} is 2. Using Lemma \ref{lem:co exp}, comparing the coefficients of the terms $r^2$ in both sides of \eqref{eq:third}, one can show that
$$
(\bsi k_p ^3 a_3 - k_s^3 b_3) \mathrm{e}^{\bsi \varphi_0} \hat{\mathbf{e}}_1+\left(2 \bsi k_p^4 a_2- 2 k_s^4 b_2+(-\bsi k_p^4 a_4+k_s^4 b_4)\mathrm{e}^{2 \bsi \varphi_0} \right) \mu \hat{\mathbf{e}}_2= \bmf{0}. 
$$
Since $\hat{\mathbf{e}}_1$ and $\hat{\mathbf{e}}_2$ are linearly independent, we obtain
\begin{equation}\label{eq:12}
\bsi k_p ^3 a_3 - k_s^3 b_3=0,\quad
  2 \bsi k_p^4 a_2- 2 k_s^4 b_2 - (\bsi k_p^4 a_4 - k_s^4 b_4)\mathrm{e}^{2 \bsi \varphi_0}=0.
\end{equation}
Similarly, from Lemma \ref{lem:co exp}, we compare the coefficients of the terms $r^3$ in both sides of \eqref{eq:third}, and can conclude that
\begin{equation}\nonumber
\begin{aligned}
 & \left( -2\bsi k_p^4 a_2 + 4 k_s^4 b_2 + (\bsi k_p^4 a_4 - k_s^4 b_4) \mathrm{e}^{ 2 \bsi \varphi_0}\right) \hat{\mathbf{e}}_1\\
 & + \left( 3 \bsi k_p^5 a_3 -3 k_s^5 b_3 -  ( \bsi k_p^5 a_5 - k_s^5 b_5)  \mathrm{e}^{2 \bsi \varphi_0} \right) \mu \mathrm{e}^{ \bsi \varphi_0} \hat{\mathbf{e}}_2
 = \bmf{0}.
 \end{aligned}
\end{equation}
Since $\hat{\mathbf{e}}_1$ and $\hat{\mathbf{e}}_2$ are linearly independent, we obtain that
\begin{equation}\label{eq:13}
 \bigg\{  \begin{array}{l} -2\bsi k_p^4 a_2 + 4 k_s^4 b_2 + (\bsi k_p^4 a_4 - k_s^4 b_4) \mathrm{e}^{ 2 \bsi \varphi_0}=0,\\
 3 \bsi k_p^5 a_3 -3 k_s^5 b_3 -  ( \bsi k_p^5 a_5 - k_s^5 b_5)  \mathrm{e}^{2 \bsi \varphi_0}=0.
  \end{array}
\end{equation}
Substituting  the second equation of \eqref{eq:12} into the first equation of \eqref{eq:13}, one can show that
\begin{equation}\label{eq:14}
b_2=0.
\end{equation}
Substituting \eqref{eq:14} into the second equation of \eqref{eq:1}, we can obtain $
%\begin{equation}\label{eq:15}
a_2=0.
%\end{equation}
$

By now we have proved $a_\ell=b_\ell=0$ for $\ell=0,1,2$. According to Lemma \ref{lem:third}, we know that \eqref{eq:lem31} holds. Therefore, from Proposition \ref{prop:1}, we readily have that $\bmf{u}\equiv \bmf{0}$ in $\Omega$.

The proof is complete.
\end{proof}

\subsection{The case with a singular simply-supported line}
\begin{lem}\label{lem:forth}
Let $\mathbf{u}$ be a Lam\'e eigenfunction of \eqref{eq:lame}, where   $\mathbf{u}$ has the expansion \eqref{eq:u} around the origin.  Suppose that  $\Gamma_h^+$ is defined in \eqref{eq:gamma_pm} such that $\Gamma_h^+\in {\mathcal F}_\Omega^{\kappa}$. Then we have
\begin{equation}\label{eq:16}
 \bigg\{  \begin{array}{l} 2 (\lambda+\mu) k_p^2 a_0 +(k_p^2 a_2+ \bsi k_s^2 b_2) \mu  \mathrm{e}^{2 \bsi \varphi_0}=0,\\
  k_p a_1 + \bsi k_s b_1 = 0,
  \end{array}
\end{equation}
and
\begin{equation}\label{eq:17}
\bigg\{
    \begin{array}{l}
  2 k_p^2 a_0 - (k_p^2 a_2+ \bsi k_s^2 b_2) \mathrm{e}^{2 \bsi \varphi_0}=0\\
   -(2 \lambda +\mu) k_p ^3 a_1 + \bsi k_s^3 b_1 \mu -(k_p^3 a_3+ \bsi k_s^3 b_3) \mu  \mathrm{e}^{2 \bsi \varphi_0}=0.
 \end{array}
\end{equation}
Moreover, it holds that
\begin{equation}\label{eq:a0 forth}
	a_0=0.
\end{equation}
Furthermore, suppose that
\begin{equation}\label{eq:aLbL4}
	a_\ell=b_\ell=0,
\end{equation}
where $\ell=0,\ldots,m-1$ and $m\in {\mathbb N}$ with $m\geq 3 $, then
$$
a_\ell=b_\ell=0, \quad \forall \ell \in \mathbb N \cup \{0\}.
$$
\end{lem}

\begin{proof}
Since $\Gamma_h^+$ is a simply-supported line of $\bmf{u}$, we have from \eqref{eq:forth2} that
\begin{equation}\label{eq:forth}
 \begin{aligned}
 \bmf{0} = & \sum_{m=0}^{\infty} \mathrm{ e}^{\bsi m \varphi_0} \bigg\{ \left[ \frac{ k_p}{2} a_{m} ( - J_{m-1}(k_p r) + J_{m+1}(k_p r) )\right.\\
   & -  \frac{\bsi k_s}{2} b_m (J_{m-1}(k_s r) + J_{m+1}(k_s r) )\bigg]\hat{\mathbf{e}}_1 \\
   & + \left[- \frac{ k_p^2}{2} a_{m} (J_{m-2}(k_p r)\mu + 2(\lambda+\mu) J_m(k_p r)+ J_{m+2}(k_p r) \mu)\right.\\
   & - \frac{\bsi k_s^2}{2} b_m (J_{m-2}(k_s r) - J_{m+2}(k_s r) )\mu \bigg]  \hat{\mathbf{e}}_2 \bigg\},\quad 0 \leqslant r\leqslant h,
 \end{aligned}
\end{equation}
where $\hat{\mathbf{e}}_1=(0,1)^\top $, $\hat{\mathbf{e}}_2=(1,0)^\top $.  Using  Lemma \ref{lem:co exp} and \eqref{eq:Jmfact}, and comparing the coefficients of the terms $r^0$ in both sides of \eqref{eq:forth}, we have that
$$
( k_p a_1 + \bsi k_s b_1) \mathrm{e}^{\bsi \varphi_0} \hat{\mathbf{e}}_1 + \left[2  (\lambda+\mu) k_p^2 a_0 + ( k_p^2 a_2 + \bsi k_s^2 b_2) \mu \mathrm{e}^{2 \bsi \varphi_0} \right] \hat{\mathbf{e}}_2= \bmf{0},
$$
Since $\hat{\mathbf{e}}_1$ and $\hat{\mathbf{e}}_2$ are linearly independent, we obtain \eqref{eq:16}. Similarly, from Lemma \ref{lem:co exp}, we compare the coefficients of the terms $r^1$ in both sides of \eqref{eq:forth}, and derive that
\begin{equation}\nonumber
\begin{aligned}
  &\left[ 2 k_p^2 a_0 - (k_p^2 a_2 + \bsi k_s^2 b_2) \mathrm{e}^{2 \bsi \varphi_0}\right] \hat{\mathbf{e}}_1   \\
  & + \left[-(2\lambda+\mu) k_p ^3 a_1 + \bsi k_s^3 \mu b_1 - (k_p^3 a_3 + \bsi k_s^3 b_3)\mu \mathrm{e}^{2 \bsi \varphi_0} \right] \mathrm{e}^{ \bsi \varphi_0} \hat{\mathbf{e}}_2= \bmf{0}.
\end{aligned}
\end{equation}
Since $\hat{\mathbf{e}}_1$ and $\hat{\mathbf{e}}_2$ are linearly independent, one obtains \eqref{eq:17}.

Multiplying $\mu$ on both sides of \eqref{eq:17}, then adding it to \eqref{eq:16}, by \eqref{eq:convex}, one can directly obtain \eqref{eq:a0 forth}.

Next we assume that \eqref{eq:aLbL4}  is valid.
Substituting \eqref{eq:aLbL4} into \eqref{eq:forth}, it yields that
\begin{equation}\label{eq:forth1}
 \begin{aligned}
 \bmf{0} = & \sum_{m=\ell}^{\infty} \mathrm{ e}^{\bsi m \varphi_0} \bigg\{ \left[ \frac{ k_p}{2} a_{m} ( - J_{m-1}(k_p r) + J_{m+1}(k_p r) )\right.\\
   & -  \frac{\bsi k_s}{2} b_m (J_{m-1}(k_s r) + J_{m+1}(k_s r) )\bigg]\hat{\mathbf{e}}_1 \\
   & + \left[- \frac{ k_p^2}{2} a_{m} (J_{m-2}(k_p r)\mu + 2(\lambda+\mu) J_m(k_p r)+ J_{m+2}(k_p r) \mu)\right.\\
   & - \frac{\bsi k_s^2}{2} b_m (J_{m-2}(k_s r) - J_{m+2}(k_s r) )\mu \bigg]  \hat{\mathbf{e}}_2 \bigg\},\quad 0 \leqslant r\leqslant h.
 \end{aligned}
\end{equation}
The lowest order of the power with respect to $r$ in the right hand of \eqref{eq:forth1} is $m-1$.  Comparing the coefficients of the terms $r^{m-1}$ in both sides of \eqref{eq:forth1}, it holds that
$$
  -(  k_p^m a_m + \bsi k_s^m b_m )  \hat{\mathbf{e}}_1 - (  k_p^{m+1} a_{m+1} + \bsi k_s^{m+1} b_{m+1}) \mu \mathrm{e}^{ \bsi \varphi_0} \hat{\mathbf{e}}_2= \bmf{0},
$$
which readily gives that
\begin{equation}\label{eq:18}
 k_p^m a_m + \bsi k_s^m b_m=0,\quad
  k_p^{m+1} a_{m+1} + \bsi k_s^{m+1} b_{m+1}=0.
\end{equation}
Comparing the coefficients of the terms $r^m$ in both sides of \eqref{eq:forth1},  it holds that
\begin{equation}\nonumber
\begin{aligned}
& - (k_p^{m+1} a_{m+1} + \bsi  k_s^{m+1} b_{m+1}) \mathrm{e}^{\bsi \varphi_0} \hat{\mathbf{e}}_1+  \big(((m-2)\mu-2\lambda)  k_p^{m+2} a_m \\
 & + \bsi m \mu k_s^{m+2} b_m - (  k_p^{m+2} a_{m+2} + \bsi k_s^{m+2} b_{m+2}) \mu \mathrm{e}^{2 \bsi \varphi_0}\big) \hat{\mathbf{e}}_2= \bmf{0}.
\end{aligned}
\end{equation}
Therefore we have
\begin{equation}\label{eq:19}
 \bigg\{  \begin{array}{l} k_p^{m+1} a_{m+1} + \bsi  k_s^{m+1} b_{m+1}=0,\\
 ((m-2)\mu-2\lambda)  k_p^{m+2} a_m + \bsi m \mu k_s^{m+2} b_m - (  k_p^{m+2} a_{m+2} + \bsi k_s^{m+2} b_{m+2}) \mu \mathrm{e}^{2 \bsi \varphi_0}=0,
  \end{array}
\end{equation}
by using the fact that $\hat{\mathbf{e}}_1$ and $\hat{\mathbf{e}}_2$ are linearly independent.
Finally comparing the coefficients of the terms $r^{m+1}$ in both sides of \eqref{eq:forth1},  it holds that
\begin{equation}\nonumber
\begin{aligned}
& \left( (m+2)  k_p^{m+2} a_m + \bsi m k_s^{m+2} b_m - ( k_p^{m+2} a_{m+2} + \bsi k_s^{m+2} b_{m+2}) \mathrm{e}^{ 2 \bsi \varphi_0}\right) \hat{\mathbf{e}}_1\\
& + \big( (((m-1)\mu-2\lambda) k_p^{m+3} a_{m+1} + \bsi (m+1)\mu k_s^{m+3} b_{m+1}) \mathrm{e}^{ \bsi \varphi_0}\\
&  -  (k_p^{m+3} a_{m+3} + \bsi k_s^{m+3} b_{m+3})  \mathrm{e}^{3 \bsi \varphi_0}  \mu \big)\hat{\mathbf{e}}_2 = \bmf{0},
\end{aligned}
\end{equation}
%Since $\hat{\mathbf{e}}_1$ and $\hat{\mathbf{e}}_2$ are linear independent, we obtain that
which further yields that
\begin{equation}\label{eq:20}
 \Bigg \{  \begin{array}{l}(m+2)  k_p^{m+2} a_m + \bsi m k_s^{m+2} b_m - ( k_p^{m+2} a_{m+2} + \bsi k_s^{m+2} b_{m+2}) \mathrm{e}^{ 2 \bsi \varphi_0}=0,\\
 \begin{aligned}
 &((m-1)\mu-2\lambda) k_p^{m+3} a_{m+1} + \bsi (m+1) \mu k_s^{m+3} b_{m+1}\\
 & -  (k_p^{m+3} a_{m+3} + \bsi k_s^{m+3} b_{m+3})  \mathrm{e}^{2 \bsi \varphi_0}  \mu=0.
 \end{aligned}
  \end{array}
\end{equation}

We multiply the first equation of \eqref{eq:20} by $\mu$ and subtract it from the second equation of \eqref{eq:19}, which readily gives that
\begin{equation}
\notag
%\label{eq:21}
2(\lambda+2\mu  )a_m=0.
\end{equation}
Therefore one has
\begin{equation}\label{eq:am21}
	a_m=0
\end{equation}
 by noting \eqref{eq:convex}.
Substituting \eqref{eq:am21} into the first equation of \eqref{eq:18}, we can obtain
\begin{equation}\label{eq:22}
b_m=0.
\end{equation}
Repeating the above procedure for deriving \eqref{eq:am21} and  \eqref{eq:22}, we can prove that $a_\ell=b_\ell=0$ for $\ell=m+1,\ldots$.

The proof is complete. 
\end{proof}

\begin{thm}\label{thm:fourth}
	Let $\mathbf{u}\in L^2(\Omega)^2$ be a solution to \eqref{eq:lame}. If there exits a  singular simply-supported line $\Gamma_h \subset  \Omega$ corresponding to $\mathbf u$, then $\bmf{u}\equiv \bmf{ 0}$.
\end{thm}

\begin{proof}

Suppose that there exits a singular soft-clamped line $\Gamma_h$ of $\mathbf u$ as described at the beginning of this section. Therefore we have  $\boldsymbol{\tau}^\top \nabla \bmf{u} {\nu}|_{\mathbf x=\bmf{0} }=0$, which implies that \eqref{eq:gradient2} holds.

According to Lemma \ref{lem:forth}, \eqref{eq:16} and \eqref{eq:17} hold since $\Gamma_h$ is a  soft-clamped line of $\mathbf u$.
Substituting  the first equation of \eqref{eq:16} into the first equation of \eqref{eq:17},  we have
\begin{equation}\label{eq:23}
	a_0=0.
\end{equation}
Substituting \eqref{eq:23} into \eqref{eq:17} yields that
\begin{equation}\label{eq:a_2337}
	k_p^2 a_2+\bsi k_s^2 b_2=0.
\end{equation}
Substituting \eqref{eq:a_2337} into   \eqref{eq:gradient2},  we have
\begin{equation}\label{eq:24}
	b_0=0.
\end{equation}
Using  Lemma \ref{lem:co exp}, \eqref{eq:23} and  \eqref{eq:24}, and comparing the coefficients of the terms $r^2$ in both sides of \eqref{eq:forth},  it holds that
\begin{equation}\nonumber
\begin{aligned}
& \big(3 k_p ^3 a_1 + \bsi k_s^3 b_1 -  (k_p^3 a_3 + \bsi k_s^3 b_3)\mathrm{e}^{ 2 \bsi \varphi_0}\big)\hat{\mathbf{e}}_1\\
& +\big((-2 \lambda k_p^4 a_2+ 2 \bsi \mu k_s^4 b_2)  - ( k_p^4 a_4+ \bsi k_s^4 b_4) \mu \mathrm{e}^{2 \bsi \varphi_0} \big) \mathrm{e}^{ \bsi \varphi_0}  \hat{\mathbf{e}}_2= \bmf{0} .
\end{aligned}
\end{equation}
Since $\hat{\mathbf{e}}_1$ and $\hat{\mathbf{e}}_2$ are linearly independent, we obtain
\begin{equation}\label{eq:25}
 \bigg\{  \begin{array}{l} 3 k_p ^3 a_1 + \bsi k_s^3 b_1  -  (k_p^3 a_3 + \bsi k_s^3 b_3)\mathrm{e}^{ 2 \bsi \varphi_0}=0,\\
 -2 \lambda k_p^4 a_2+ 2 \bsi \mu k_s^4 b_2 - ( k_p^4 a_4+ \bsi k_s^4 b_4) \mu \mathrm{e}^{2 \bsi \varphi_0}=0.
  \end{array}
\end{equation}
By multiplying the first equation of \eqref{eq:25} with $\mu$ and subtracting it from the second equation of \eqref{eq:17}, one can show that
\begin{equation}\label{eq:27z}
a_1=0.
\end{equation}
Substituting  \eqref{eq:27z} into the second equation of \eqref{eq:16}, we can obtain
\begin{equation}\label{eq:28z}
b_1=0.
\end{equation}

By now we have proved $a_\ell=b_\ell=0$ for $\ell=0,1$.  Similarly, from Lemma \ref{lem:co exp}, we compare the coefficients of the terms $r^3$ in both sides of \eqref{eq:forth}, and can conclude that
\begin{equation}\nonumber
\begin{aligned}
& \big(4 k_p ^4 a_2 + 2 \bsi k_s^4 b_2 -  (k_p^4 a_4 + \bsi k_s^4 b_4)\mathrm{e}^{ 2 \bsi \varphi_0}\big)\hat{\mathbf{e}}_1\\
& +\big(((\mu-2 \lambda) k_p^5 a_3 + 3 \bsi \mu k_s^5 b_3)  - ( k_p^5 a_5 + \bsi k_s^5 b_5) \mu \mathrm{e}^{2 \bsi \varphi_0} \big) \mathrm{e}^{ \bsi \varphi_0}  \hat{\mathbf{e}}_2= \bmf{0}. 
\end{aligned}
\end{equation}
Since $\hat{\mathbf{e}}_1$ and $\hat{\mathbf{e}}_2$ are linearly independent, we obtain that
\begin{equation}\label{eq:26}
 \bigg\{  \begin{array}{l} 4 k_p ^4 a_2 + 2 \bsi k_s^4 b_2 -  (k_p^4 a_4 + \bsi k_s^4 b_4)\mathrm{e}^{ 2 \bsi \varphi_0}=0,\\
 (\mu-2 \lambda) k_p^5 a_3 + 3 \bsi \mu k_s^5 b_3 - ( k_p^5 a_5 + \bsi k_s^5 b_5) \mu \mathrm{e}^{2 \bsi \varphi_0}=0.
  \end{array}
\end{equation}
Multiplying the first equation of \eqref{eq:26} by $\mu$ and subtracting it from the second equation of \eqref{eq:25}, it yields that
\begin{equation}\label{eq:27}
a_2=0.
\end{equation}
Substituting  \eqref{eq:27}, \eqref{eq:23} into the first equation of \eqref{eq:17}, we can obtain
\begin{equation}\label{eq:28}
b_2=0.
\end{equation}

 %Using Lemma \ref{lem:third}, we can prove

 Since $a_\ell=b_\ell=0$ for $\ell=0,1,2$, according to Lemma \ref{lem:third}, we have that $a_\ell=b_\ell=0, \forall \ell \in \mathbb N \cup \{0\}$. Therefore, from Proposition \ref{prop:1}, we know that $\bmf{u}\equiv \bmf{0}$ in $\Omega$.

The proof is complete.
\end{proof}

\subsection{The case with a singular generalized-impedance line }

In this subsection, we shall establish the GHP for the presence of a singular generalize impedance line. Before that, we need several auxiliary lemmas. Using \eqref{eq:gradient1} and \eqref{eq:gradient2}, we can directly obtain that

%Henceforth, according to Definition~\ref{def:class1}, we assume that $\boldsymbol{ \eta}_1$ is given by the following absolutely convergent series at $\mathbf 0\in \Gamma_h^+$:
%\begin{equation}\label{eq:eta1 ex}
%	%\begin{align}
%		\boldsymbol{\eta}_1=\eta_1+\sum_{j=1}^\infty \eta_{1,j} r^j,  
%		%\quad %\label{eq:eta1 ex} \\
%	%	\boldsymbol{\eta}_2=\eta_{2}+\sum_{j=1}^\infty \eta_{2,j}(\theta) r^j \label{eq:eta2 ex}
%	%\end{align}
%\end{equation}
%    where $\eta_{1}\in\mathbb{C}\backslash\{0\}$, $\eta_{1,j}\in \mathbb C$ and $r\in [-h,h]$.

\begin{lem}\label{lem:34}
Suppose that \eqref{eq:sg cond} holds, then we have
\begin{equation}
\bsi k_p^2 a_0 + k_s^2 b_0=0. \label{eq:gradient3}
\end{equation} \end{lem}

\begin{lem}\label{lem:g imp}
Let $\mathbf{u}$ be a Lam\'e eigenfunction of \eqref{eq:lame}, where   $\mathbf{u}$ has the expansion \eqref{eq:u} around the origin.  Suppose that  $\Gamma_h^+$ is defined in \eqref{eq:gamma_pm} such that $\Gamma_h^+\in {\mathcal H}_\Omega^{\kappa}$ with the generalized-impedance parameter $\boldsymbol{ \eta}_1 $ satisfying \eqref{eq:eta1 ex}.
%, where
%\begin{equation}
%	\boldsymbol{\eta}=\eta+\sum_{j=1}^\infty \eta_j r^j
%\end{equation} 
Then we have
\begin{equation}\label{eq:1b}
\bigg\{
    \begin{array}{l}
  (1+ \bsi \eta_1)( \bsi k_p a_1 -  k_s b_1) =0,\\
  2 \eta_1 (\lambda+\mu) k_p^2 a_0 + (1 - \bsi \eta_1) ( \bsi k_p^2 a_2 -  k_s^2 b_2) \mu \mathrm{e}^{2 \bsi \varphi_0}=0,
 \end{array}
\end{equation}
and 
\begin{equation}\label{eq:g2}
\bigg\{
    \begin{array}{l}
  2 k_s^2 b_0 + \mathrm{e}^{2 \bsi \varphi_0} (1+\bsi \eta_1) (\bsi k_p^2 a_2 - k_s^2 b_2) =0,\\
  (\bsi \mu - 2 \lambda \eta_1 - \mu \eta_1) k_p^3 a_1 + \mu (\bsi \eta_1 - 1) k_s^3 b_1  - \mu \mathrm{e}^{2 \bsi \varphi_0} ( 1 - \bsi \eta_1) (\bsi k_p^3 a^3 - k_s^3 b_3) =0.
 \end{array}
\end{equation}
%and
%\begin{equation}\label{eq:17}
%\bigg\{
%    \begin{array}{l}
%  2 k_p^2 a_0 - (k_p^2 a_2+ \bsi k_s^2 b_2) \mathrm{e}^{2 \bsi \varphi_0}=0\\
%   -(2 \lambda +\mu) k_p ^3 a_1 + \bsi k_s^3 b_1 \mu -(k_p^3 a_3+ \bsi k_s^3 b_3) \mu  \mathrm{e}^{2 \bsi \varphi_0}=0.
% \end{array}
%\end{equation}
%Furthermore, suppose that
%\begin{equation}\label{eq:aLbL4}
%	a_\ell=b_\ell=0
%\end{equation}
%where $\ell=0,\ldots,m-1$ and $m\in {\mathbb N}$ with $m\geq 3 $, then
%$$
%a_\ell=b_\ell=0, \quad \forall \ell \in \mathbb N \cup \{0\}.
%$$
\end{lem}

\begin{proof}
Since $\Gamma_h^+$ is a generalized-impedance line of $\bmf{u}$, using  \eqref{eq:third2} and \eqref{eq:forth2}, we have
\begin{equation}\label{eq:GI}
 \begin{aligned}
\bmf{0} = & \sum_{m=0}^{\infty} \mathrm{ e}^{\bsi m \varphi_0} \bigg\{ \left[ \frac{\bsi k_p}{2} a_{m} (J_{m-1}(k_p r) + J_{m+1}(k_p r) )\right.\\
   & +  \frac{k_s}{2} b_m (-J_{m-1}(k_s r) + J_{m+1}(k_s r) )\bigg]\hat{\mathbf{e}}_1 + \left[ \frac{\bsi k_p^2}{2} a_{m} (-J_{m-2}(k_p r) + J_{m+2}(k_p r) )\right.\\
   & + \frac{k_s^2}{2} b_m (J_{m-2}(k_s r) + J_{m+2}(k_s r) )\bigg] \mu \hat{\mathbf{e}}_2 \bigg\}\\
   &+\boldsymbol{ \eta }_1 \bigg\{ \left[ \frac{ k_p}{2} a_{m} ( - J_{m-1}(k_p r) + J_{m+1}(k_p r) )\right.
    -  \frac{\bsi k_s}{2} b_m (J_{m-1}(k_s r) + J_{m+1}(k_s r) )\bigg]\hat{\mathbf{e}}_1 \\
   & + \left[- \frac{ k_p^2}{2} a_{m} (J_{m-2}(k_p r)\mu + 2(\lambda+\mu) J_m(k_p r)+ J_{m+2}(k_p r) \mu)\right.\\
   & - \frac{\bsi k_s^2}{2} b_m (J_{m-2}(k_s r) - J_{m+2}(k_s r) )\mu \bigg]  \hat{\mathbf{e}}_2 \bigg\},\quad 0 \leqslant r\leqslant h,
 \end{aligned}
\end{equation}
where $\hat{\mathbf{e}}_1=(0,1)^\top $, $\hat{\mathbf{e}}_2=(1,0)^\top $.   Using  Lemma \ref{lem:co exp}, \eqref{eq:Jmfact} and \eqref{eq:eta1 ex}, and comparing the coefficient of the term $r^0$ in both sides of \eqref{eq:GI}, we see that
\begin{equation}\nonumber
\begin{aligned}
&\left[(\bsi - \eta_1) k_p a_1 - (1+ \bsi \eta_1) k_s b_1\right] {\mathrm e}^{\bsi \varphi_0} \hat{\mathbf{e}}_1\\
&- \left[ 2 \eta (\lambda+\mu) k_p^2 a_0 +\left((\bsi + \eta_1) k_p^2 a_2 - (1 - \bsi \eta_1) k_s^2 b_2\right) \mu {\mathrm e}^{2 \bsi \varphi_0}\right] \hat{\mathbf{e}}_2=\bmf{0},
\end{aligned}
\end{equation}
which readily gives \eqref{eq:1b}. 

Using Lemma \ref{lem:co exp} and \eqref{eq:Jmfact}, and comparing the coefficient of the term $r^1$ in both sides of \eqref{eq:GI}, together with \eqref{eq:eta1 ex}, we can show that
\begin{equation}\nonumber
\begin{aligned}
& \left[2 k_s^2 b_0 + \mathrm{e}^{2 \bsi \varphi_0} (1+\bsi \eta) (\bsi k_p^2 a_2 - k_s^2 b_2)\right] \hat{\mathbf{e}}_1 + \\
& \left[\left((\bsi \mu - 2 \lambda \eta - \mu \eta) k_p^3 a_1 + \mu (\bsi \eta - 1) k_s^3 b_1\right) \mathrm{e}^{ \bsi \varphi_0}  - \mu \mathrm{e}^{3 \bsi \varphi_0} ( 1 - \bsi \eta) (\bsi k_p^3 a^3 - k_s^3 b_3)\right] \hat{\mathbf{e}}_2 \\
& = \bmf{0}.
\end{aligned}
\end{equation}
Since $\hat{\mathbf{e}}_1$ and $\hat{\mathbf{e}}_2$ are linearly independent, we readily obtain \eqref{eq:g2}. 

The proof is complete. 
\end{proof}

\begin{lem}\label{lem:determinant}
Let $\lambda,\mu$ be the Lam\'e constants satisfying the strong convexity condition \eqref{eq:convex}. Suppose that $k_{p}$ and $ k_{{s}} $ are  the compressional and shear wave numbers defined in \eqref{eq:kpks} respectively. If
$$
\eta_1 \neq  \frac{\pm \sqrt{(\lambda + 3\mu)(\lambda + \mu)} - \mu \bsi}{\lambda + 2 \mu},
$$
then for $\forall m \in \mathbb{N} \cup \{0\}$ and $\forall \varphi_0 \in (0,\pi)$
\begin{equation}
\notag
%\label{eq:determinant}
D_m=\left|\begin{array}{ccc}
\bsi k_p^m                                                                & -  k_s^m                                  &            0\\
(\bsi m \mu + m \mu \eta_1 - 2 \lambda \eta_1 - 2 \mu  \eta_1)  k_p^{m+2}        & m \mu ( \bsi \eta_1  - 1) k_s^{m+2}  & \mu( \bsi \eta_1 - 1) \mathrm{ e}^{2 \bsi
\varphi_0}\\
(m \eta_1 + 2 \eta_1 - \bsi m )  k_p^{m+2}                              &   ( m + 2 + \bsi m \eta_1 ) k_s^{m+2}              &   (1 + \bsi \eta_1)\mathrm{ e}^{2 \bsi\varphi_0}\\
\end{array}\right| \neq {0}.
\end{equation}
%where $m \in \mathbb{N} \cup {0}$.
\end{lem}
\begin{proof}
By directly calculations, it can be verified that
%The determinant of \eqref{eq:determinant} is
\begin{equation}\label{eq:determinant1}
D_m= -2 \mathrm{ e}^{2 \bsi\varphi_0} k_p^m k_s^m \left[ \bsi (\lambda + 2 \mu ) k_p^2 \eta_1 ^2 + (\lambda k_p^2 - \mu k_s^2) \eta_1 -\bsi \mu k_s^2\right].
\end{equation}
Substituting  \eqref{eq:kpks} into  \eqref{eq:determinant1}, we obtain that
\begin{equation}\label{eq:determinant2}
D_m=-2   \mathrm{ e}^{2 \bsi\varphi_0} k_p^m k_s^m  \kappa ^2 \Big [\bsi \eta_1 ^2 + \left(\frac{\lambda}{\lambda + 2 \mu} - 1 \right) \eta_1 - \bsi \Big ].
\end{equation}

It can be prove that
$$
\eta_{\sf root}=\frac{\pm \sqrt{(\lambda + 3\mu)(\lambda + \mu)} - \mu \bsi}{\lambda + 2 \mu}
$$
are roots of $\bsi \eta ^2 + \left(\frac{\lambda}{\lambda + 2 \mu} - 1 \right) \eta - \bsi$.

The proof is complete.
%Let $ f(\eta)=\bsi \eta ^2 + (\frac{\lambda}{\lambda + 2 \mu} - 1 ) \eta - \bsi $,  then we just prove $f(\eta)\neq 0$ , that is
%\begin{equation}\label{eq:determinant3}
%\bsi \eta ^2 + (\frac{\lambda}{\lambda + 2 \mu} - 1 ) \eta - \bsi \neq 0,
%\end{equation}
%$(\lambda + 2 \mu)$ $\times$ \eqref{eq:determinant3}, we have
%\begin{equation}\label{eq:determinant4}
%\bsi (\lambda + 2 \mu) \eta ^2 - 2 \mu \eta - \bsi (\lambda + 2 \mu) \neq 0,
%\end{equation}
%that is,
%\begin{equation}\label{eq:determinant5}
% (\lambda + 2 \mu) \eta ^2 + 2 \bsi \mu \eta - (\lambda + 2 \mu) \neq 0,
%\end{equation}
%the discriminant of \eqref{eq:determinant5} is
%\begin{equation}\label{eq:determinant6}
%\Delta=4(\lambda + 3 \mu)(\lambda +\mu)> 0,
%\end{equation}
%then  $ \eta_{1,2} \neq \frac{1}{\lambda + 2 \mu}[\pm \sqrt{(\lambda + 3\mu)(\lambda + \mu)} - \mu \bsi]$, we have \eqref{eq:determinant5}.
%
%The proof is complete.
\end{proof}

\begin{thm}\label{thm:g im}
Let $\mathbf{u}\in L^2(\Omega)^2$ be a solution to \eqref{eq:lame}. If there exits a singular generalized-impedance line $\Gamma_h\subset\Omega$ of $\bmf{u}$ associated with the parameter $\boldsymbol{ \eta}_1$ satisfying
\begin{equation}\label{eq:thm33 cond}
	\eta_1 \neq \pm \bsi \mbox{ and } \frac{\pm \sqrt{(\lambda + 3\mu)(\lambda + \mu)} - \mu \bsi}{\lambda + 2 \mu},
\end{equation}
 where $\eta_1$ is the constant part of  $\boldsymbol{\eta}_1$ defined in \eqref{eq:eta1 ex},  then $\bmf{u}\equiv \bmf{ 0}$.
\end{thm}
\begin{proof}
Suppose that there exits a singular generalized-impedance  line $\Gamma_h=\Gamma_h^+$ of $\mathbf u$ as described at the beginning of this section.  From Lemma \ref{lem:g imp}, we know \eqref{eq:1b} holds.

Since $\Gamma_h^+ \in \mathcal S( \mathcal H_\Omega^\kappa) $, we know that \eqref{eq:def3} and \eqref{eq:sg cond} hold. Therefore \eqref{eq:gradient3} holds according to Lemma \ref{lem:34}. By virtue of \eqref{eq:def3}, substituting  $b_0=0$ into \eqref{eq:gradient3}, we obtain
  \begin{equation}\label{eq:2b}
  a_0=0.
  \end{equation}
Substituting $b_1=b_2=0$ and \eqref{eq:2b} into \eqref{eq:1b}, under the assumption $\eta_1 \neq \pm \bsi $, we obtain
  \begin{equation}\label{eq:3b}\notag
  a_1=
  a_2=0.
  \end{equation}
  By now we have proved that $a_\ell=0$ for $\ell=0,1,2$ under the assumptions $\Gamma_h^+ \in \mathcal S( \mathcal H_\Omega^\kappa) $ and $\eta_1 \neq \pm \bsi$. In the following we prove $a_\ell=b_\ell=0$ for $\forall m \in \mathbb N\backslash\{1,2\}$ by induction.

 % conditions \eqref{eq:def3} and \eqref{eq:sg cond}.

  Suppose that
  \begin{align}\label{eq:352 al}
  	a_{\ell}=b_{\ell}=0,\quad  \ell=0,\ldots, m-1,
  \end{align}
  where $m\geq 3$ is a natural number.  Substituting \eqref{eq:352 al} into \eqref{eq:GI}, one can show that
  \begin{equation}\label{eq:GI1}
 \begin{aligned}
\bmf{0} = & \sum_{m=\ell}^{\infty} \mathrm{ e}^{\bsi m \varphi_0} \bigg\{ \left[ \frac{\bsi k_p}{2} a_{m} (J_{m-1}(k_p r) + J_{m+1}(k_p r) )\right.\\
   & +  \frac{k_s}{2} b_m (-J_{m-1}(k_s r) + J_{m+1}(k_s r) )\bigg]\hat{\mathbf{e}}_1 + \left[ \frac{\bsi k_p^2}{2} a_{m} (-J_{m-2}(k_p r) + J_{m+2}(k_p r) )\right.\\
   & + \frac{k_s^2}{2} b_m (J_{m-2}(k_s r) + J_{m+2}(k_s r) )\bigg] \mu \hat{\mathbf{e}}_2 \bigg\}\\
   &+\boldsymbol{ \eta}_1 \bigg\{ \left[ \frac{ k_p}{2} a_{m} ( - J_{m-1}(k_p r) + J_{m+1}(k_p r) )\right.
    -  \frac{\bsi k_s}{2} b_m (J_{m-1}(k_s r) + J_{m+1}(k_s r) )\bigg]\hat{\mathbf{e}}_1 \\
   & + \left[- \frac{ k_p^2}{2} a_{m} (J_{m-2}(k_p r)\mu + 2(\lambda+\mu) J_m(k_p r)+ J_{m+2}(k_p r) \mu)\right.\\
   & - \frac{\bsi k_s^2}{2} b_m (J_{m-2}(k_s r) - J_{m+2}(k_s r) )\mu \bigg]  \hat{\mathbf{e}}_2 \bigg\},\quad 0 \leqslant r\leqslant h.
 \end{aligned}
\end{equation}
The lowest order of the power with respect to $r$ in the right hand side of \eqref{eq:GI1} is $m-1$. Comparing the coefficients of the terms $r^{m-1}$ in both sides of \eqref{eq:GI1}, together with the use of \eqref{eq:eta1 ex}, one can directly verify that
\begin{equation}\label{eq:5b}
  ( \bsi \eta_1 + 1)( \bsi k_p^m a_m - k_s^m b_m) = 0,\quad
  (\bsi \eta_1 - 1)(\bsi k_p^{m+1} a_{m+1} - k_s^{m+1} b_{m+1})=0.
\end{equation}
Similarly,  comparing the coefficients of the terms $r^m$ in both sides of \eqref{eq:GI1} and in view of \eqref{eq:eta1 ex},  we obtain that
\begin{equation}\label{eq:6b}
 \left\{  \begin{array}{l} (\bsi \eta_1 +1)(\bsi k_p^{m+1} a_{m+1} -  k_s^{m+1} b_{m+1})=0,\\
 \begin{aligned}
 &(\bsi m \mu + m \mu \eta - 2 \lambda \eta_1 - 2 \mu \eta_1)  k_p^{m+2} a_m +  m \mu( \bsi \eta_1  - 1) k_s^{m+2} b_m\\
  &+ \mu ( \bsi \eta_1 -1 ) ( \bsi   k_p^{m+2} a_{m+2} -  k_s^{m+2} b_{m+2}) \mathrm{e}^{2 \bsi \varphi_0}=0.
 \end{aligned}
  \end{array}
  \right.
\end{equation}
Finally comparing the coefficients of the terms $r^{m+1}$ in both sides of \eqref{eq:GI1}, together with the use of \eqref{eq:eta1 ex}, 
 we obtain that
\begin{equation}\label{eq:7b}
 \left\{  \begin{array}{l}
 \begin{aligned}
 &  (m \eta_1 + 2 \eta_1 - \bsi m )  k_p^{m+2} a_m + ( m + 2 + \bsi m \eta_1 ) k_s^{m+2} b_m \\
 & +   ( 1 + \bsi \eta_1 ) \left[\bsi k_p^{m+2} a_{m+2} - k_s^{m+2} b_{m+2}\right] \mathrm{e}^{ 2 \bsi \varphi_0}=0,\\
 & [(m \mu -\mu -2 \lambda) \eta_1 + \bsi (m+1) \mu] k_p^{m+3} a_{m+1} + (m+1) \mu ( \bsi  \eta_1 - 1 ) k_s^{m+3} b_{m+1}\\
 & + \mu ( \bsi \eta_1-1 ) ( \bsi k_p^{m+3} a_{m+3} -  k_s^{m+3} b_{m+3})  \mathrm{e}^{2 \bsi \varphi_0} =0.
 \end{aligned}
  \end{array}
  \right.
\end{equation}
From \eqref{eq:5b}, \eqref{eq:6b}, \eqref{eq:7b}, we have
\begin{align*}
	&\left[\begin{array}{ccc}
\bsi k_p^m                                                                & -  k_s^m                                  &            0\\
(\bsi m \mu + m \mu \eta_1 - 2 \lambda \eta_1 - 2 \mu  \eta_1)  k_p^{m+2}        & m \mu ( \bsi \eta_1  - 1) k_s^{m+2}  & \mu( \bsi \eta_1- 1) \mathrm{ e}^{2 \bsi
\varphi_0}\\
(m \eta_1+ 2 \eta_1 - \bsi m )  k_p^{m+2}                              &   ( m + 2 + \bsi m \eta_1 ) k_s^{m+2}              &   (1 + \bsi \eta_1)\mathrm{ e}^{2 \bsi\varphi_0}
\end{array}\right] \\
&\quad \times
 \left[\begin{array}{c}
 a_m \\ b_m \\ c_{m+2}
 \end{array}\right]=\bmf{0},
\end{align*}
where $c_{m+2}= \bsi k_p^{m+2} a_{m+2} - k_s^{m+2} b_{m+2} $. By Lemma \ref{lem:determinant}, along with the use of \eqref{eq:thm33 cond}, one has $a_m=b_m=0$. Using the above induction procedure, we can prove that $a_m=b_m=0$ for $m \in \mathbb{N} \cup {0}$.

The proof is complete.
\end{proof}

%\subsection { The case with a singular rigid  line, traction-free line, impedance line, respectively}

\section{GHP with the non-degenerate intersection of two homogeneous lines}\label{sec:4}

In this section, microlocal singularities of the Lam\' e  eigenfunction $\mathbf u$ can be established when two homogeneous lines introduced in Definition~\ref{def:1} intersect with each other under the generic non-degenerate case. Microlocal singularities prevent the occurrence of such intersections unless the Lam\'e eigenfunction $\bmf{u}$ is identically vanishing.  As discussed in the beginning of Section~\ref{sect:2}, we assume throughout this section that the aforementioned two homogeneous lines are given by $\Gamma_h^\pm$ in \eqref{eq:gamma_pm} with the intersecting angle $\varphi_0\in (0, \pi)$.

In the following three lemmas, we review some related results from \cite{DLW} for the rigid, traction-free and impedance lines of $\mathbf u$, which shall be needed in our subsequent analysis.

\begin{lem}\cite[Lemma 3.1]{DLW}\label{lem:rigid}
Let $\mathbf{u}$ be a generalized Lam\'e eigenfunction to \eqref{eq:lame} with the radial wave  expansion given in \eqref{eq:u} around the origin. Suppose there exists $\Gamma_h^-\in {\mathcal R}_\Omega^{\kappa}  $. Then one has
\begin{equation}\label{eq:1a}
 k_p a_1 + \bsi k_s b_1=0,\quad
 k_p^3a_1-\bsi k_s^3 b_1=0,
\end{equation}
and
\begin{equation}\label{eq:2a}
  k_p ^2 a_0+ \bsi k_s^2 b_0- k_p^2  a_2- \bsi k_s ^2 b_2=0,\quad
  k_p ^2 a_0- \bsi k_s^2 b_0=0.
\end{equation}
Furthermore, suppose that
\begin{equation}\label{eq:lem cond}
	a_\ell=b_\ell=0,
\end{equation}
where $\ell=0,\ldots,m-1$ and $m\in {\mathbb N}$ with $m\geq 2 $, then
$$
a_\ell=b_\ell=0, \quad \forall \ell \in \mathbb N \cup \{0\}.
$$

\end{lem}

\begin{lem}\cite[Lemma 3.3]{DLW}\label{lem:tranction free}
Let $\mathbf{u}$ be a Lam\'e eigenfunction to \eqref{eq:lame} with the radial wave  expansion \eqref{eq:u} around the origin.  Suppose that  $\Gamma_h^- \in {\mathcal T}_\Omega^{\kappa}  $. Then we have
\begin{equation}\label{eq:3a}
\bsi k_p^2 a_2- k_s^2  b_2=0,\quad
a_0 =0,
\end{equation}
and
\begin{equation}\label{eq:4a}
k_s^3  b_1+\bsi k_p^3  a_3-k_s^3   b_3=0,\quad
 a_1=0.
\end{equation}
Furthermore, suppose that $a_\ell=b_\ell=0$ for $\ell=0,1$, then
\begin{equation}\label{eq:lem33}
	a_\ell=b_\ell=0,\quad \forall \ell \in \mathbb N\cup \{0\}.
\end{equation}
\end{lem}

Assume that $\Gamma_h^-$ is an impedance line of $\mathbf u$ associated with the impedance parameter $\boldsymbol{\eta}_2$ belonging to the class $\mathcal A(\Gamma_h^- )$. By virtue of \eqref{eq:eta2 ex}, the following lemma can be directly obtained by modifying the corresponding proof of \cite[Lemma 3.4]{DLW}. The proof of this lemma is omitted.

%Hence according to Definition \ref{def:class1} $\boldsymbol{\eta}_2$ is given by the following absolutely convergent series at $\mathbf 0\in \Gamma_h^+$:
%\begin{equation}\label{eq:eta2 ex}
%	%\begin{align}
%		\boldsymbol{\eta}_2=\eta_2+\sum_{j=1}^\infty \eta_{2,j} r^j,  
%		%\quad %\label{eq:eta1 ex} \\
%	%	\boldsymbol{\eta}_2=\eta_{2}+\sum_{j=1}^\infty \eta_{2,j}(\theta) r^j \label{eq:eta2 ex}
%	%\end{align}
%\end{equation}
%    where $\eta_{2}\in\mathbb{C}\backslash\{0\}$, $\eta_{2,j}\in \mathbb C$ and $r\in [-h,h]$. 

\begin{lem}%\cite[Lemma 3.4]{DLW}
\label{lem:impedance line}
Let $\mathbf{u}$ be a solution to \eqref{eq:lame} with the radial wave expansion \eqref{eq:u} around the origin.  Suppose that there is an impedance  line $\Gamma_h^-$ of $\bmf{u}$ with an  impedance parameter $ \boldsymbol{ \eta}_2 $ satisfying \eqref{eq:eta2 ex}. Then we have
\begin{equation}\label{eq:5a}
 \eta_2 k_p a_1+\bsi \eta_2 k_s b_1 -\bsi k_p^2 \mu a_2+k_s^2 \mu b_2=0,\quad
 a_0=0,
\end{equation}
and
\begin{equation}\label{eq:6a}
 - k_s^3 \mu b_1 + \eta_2 k_p^2 a_2+\bsi \eta_2 k_s^2 b_2 +\bsi k_p^3 \mu a_3-k_s^3 \mu b_3=0,\quad
 a_1=0.
\end{equation}
Furthermore, if $a_\ell=b_\ell=0$ for $\ell=0,1$, then
\begin{equation}\label{eq:a1bl=0}
	a_\ell=b_\ell=0, \quad \forall \ell \in \mathbb N \cup \{0\}.
\end{equation}
\end{lem}

If there are  two intersecting traction-free lines of $\mathbf u$  under certain generic conditions, it was proved that $\bmf{u}\equiv \bmf{0}$ in  $\Omega$. In fact, we have

\begin{thm}\cite[Theorem 4.2]{DLW} \label{thm:two traction}
Let $\mathbf{u}\in L^2(\Omega)^2$ be a solution to \eqref{eq:lame}. Suppose there exit two intersecting lines $\Gamma_h^\pm$ of $\bmf{u}$ such that $\Gamma_h^\pm $ are two  traction-free lines  with the intersecting angle $\angle(\Gamma_h^+,\Gamma_h^-)=\varphi_0\neq \pi $, if $\bmf{u}(\bmf{0})=\bmf{0}$ and
\begin{equation}\label{eq:Thm41 cond}
	\varphi_0 \in (0,  \varphi_{\sf root}),
\end{equation}
 where $\varphi_{\sf root}$ is the unique root of
\begin{equation}\label{eq:varphi0}
	g(\varphi):=\frac{4}{3}\cdot \frac{\varphi}{\cos^6(\varphi/2 ) }-1, \quad \varphi \in (0,\pi),
\end{equation}
 then $\bmf{u}\equiv \bmf{0}$.
% \eqref{eq:512 cond} is fulfilled.
%Let $\mathbf{u}$ be a lam\'e eigenfunction of \eqref{eq:lame}. Suppose that a traction free line $\Gamma_h^+ \Subset \Omega$ intersects with a rigid line $\Gamma_h^- \Subset \Omega $ of $\bmf{u}$ at $\mathbf{0}$ with the intersecting angle $\angle(\Gamma_h^+,\Gamma_h^-) = \varphi_0 $ satisfying
%$$
%\varphi_0 \neq\pi.
%$$
%Then $$
%\mathrm{Vani}(\bmf{u}; \mathbf{0})=+\infty.
%$$
%Then $a_m=b_m=0$ for $m \in \mathbb{N}\cup \{0\}$.
\end{thm}

In Theorem \ref{thm:Tu thm} in what follows, we shall show that the assumption \eqref{eq:Thm41 cond} in Theorem \ref{thm:two traction}  can be removed. In the proof of Theorem \ref{thm:Tu thm}, we need the complex geometrical optics solution $\mathbf v (\mathbf x)$ introduced in \cite{EBL} to establish the integral equality. In the following we first introduce the geometrical setup.

Let $B_h$ be the central disk of radius $h \in \mathbb{R}_+$. Let $\Gamma^\pm$
signify the infinite extension of $\Gamma_{h}^\pm$ in the half-space $x_2\geq 0$, where $\Gamma_h^\pm$ are defined in \eqref{eq:gamma_pm}.   Consider the open sector
\begin{equation}\label{eq:K}
	\mathcal{K} =\left\{\bmf{x}=(x_1,x_2)  \in \mathbb{R}^{2}~ |~ \bmf{x}\neq \bmf{0},\quad  0<\arg \left(x_{1}+\mathrm{i} x_{2}\right)<\varphi_{0}\right\},\quad \mathrm{i}:=\sqrt{-1},
\end{equation}
which is formed by the two half-lines $\Gamma^-$ and $\Gamma^+$. In the sequel, we set
\begin{equation}\label{eq:sh}
	S_h=\mathcal{K}\cap B_h.
\end{equation}
Next we set
\begin{equation}\label{eq:v}
\bmf{v}(\bmf{x} )=\left(\begin{array}{c}{\exp (-s \sqrt{r} \exp(\bsi \varphi/2))} \\ \bsi \cdot {\operatorname{exp}(-s \sqrt{r} \exp(\bsi \varphi/2) )}\end{array}\right):=
\left(\begin{array}{c}{v_1(\bmf{x})} \\ {v_2(\bmf{x})}\end{array}\right)=v_1(\bmf{x}) \bmf{e}_1,
\end{equation}
where $\bmf{x}=r\cdot (\cos  \varphi, \sin  \varphi ) $,  $s \in \mathbb{R}_{+}$,  $-\pi<\varphi\leqslant \pi$ and $\bmf{e}_1=(1,\bsi)^\top$. $\bmf{v}$ is known as the Complex Geometrical Optics (CGO) solution for the Lam\'e operator and  was first introduced in \cite{EBL}. In the following we set
\begin{equation}\label{eq:deltaK}
	\delta_{{\mathcal K}}=\min_{0<\varphi<\varphi_0} \cos(\varphi/2).
\end{equation}
It can seen that $\delta_{{\mathcal K}}$ is a positive constant depending ${\mathcal K}$.

We first derive the following three critical lemmas.
\begin{lem}\label{lem:r1}
Let $a \in \mathbb{C}$ and $\Re(a)>0$. For any given positive numbers $\alpha$ and  $\varepsilon$ satisfying $\varepsilon < \mathrm e $, if $    \Re(a) \geq \frac{2 \alpha}{\mathrm e}$, one has 
$$
\left|\int_{\varepsilon}^{+ \infty} r^{\alpha} \mathrm{e}^{- a r} \rmd r \right| \leq \frac{2}{\Re(a)}\mathrm{e}^{- \varepsilon \Re(a)/2}.
$$
\end{lem}

\begin{proof}
Let $g(r)=\frac{\ln r}{r}$ where $r\in [\varepsilon, +\infty )$. $g(r)$ increases monotonically in the interval $[\varepsilon,e]$ and decreases monotonically in the interval $[e,+ \infty]$. One has
$$
\max\limits_{r \in (\varepsilon,+\infty)}g(r)= \max\limits_{r \in (\varepsilon,+\infty)} \frac{\ln r}{r}= \frac{1}{e} .
$$
Since $\frac{2 \alpha}{e} \leq \Re(a)$, we have $ 2 \alpha \frac{\ln r}{r}  \leq \Re(a), \, \forall r \in [\varepsilon,+\infty]$. Therefore
$$
r^\alpha \leq {\mathrm e}^{ \Re(a) /2 },
$$
which can then be used to obtain that
\begin{equation}\nonumber
\begin{aligned}
& \left|\int_{\varepsilon}^{+ \infty} r^{\alpha} \mathrm{e}^{- a r} \rmd r \right|
%\leq \int_{\varepsilon}^{+ \infty} r^{\alpha} \mathbf{e}^{- \mathfrak{R}(\mu) r} \rmd r
 \leq   \int_{\varepsilon}^{+ \infty}  \mathrm{e}^{- \frac{1}{2}\Re(a) r }\mathrm d r= \frac{2}{\Re(a)}\mathrm{e}^{- \varepsilon \Re(\mu)/2}.
\end{aligned}
\end{equation}

The proof is complete.
\end{proof}

\begin{lem}\label{lem:r2}
For any given positive numbers $\alpha$ and  $h$ satisfying $h < \mathrm e $, $\varphi \in (0, \varphi_0 )$ with $\varphi_0 \in  (0,\pi)$, if
\begin{equation}\label{eq:s cond}
	   s \geq \frac{2 \alpha}{\mathrm e \, \delta_{\mathcal K} },
\end{equation}
where $\delta_{\mathcal K}$ is defined in \eqref{eq:deltaK},  we have
 \begin{equation}\label{eq:1r1}
  \int_{0}^{h} r^{\alpha} \mathrm{e}^{- s \mathrm{e}^{\bsi \frac{\varphi}{2}} r^{\frac{1}{2}}} \rmd r = \frac{ 2 \Gamma (2 \alpha + 2)}{(s \mathrm{e}^{\bsi \frac{\varphi}{2}}) ^ {2 \alpha + 2}} +  \Oh(\mathrm{e}^{- \frac{s \sqrt{\varepsilon}\delta_{\mathcal K}}{2}})
 \end{equation}
 as $s\rightarrow +\infty. $
 \end{lem}
 \begin{proof}
By using the change of variable $r=t^2$,  we have
\begin{equation}\label{eq:r2}
\begin{aligned}
 &\int_{0}^{h} r^{\alpha} \mathrm{e}^{- s \sqrt{r} \mathrm{e}^{\bsi \frac{\varphi}{2}} } \rmd r = 2\int_{0}^{\sqrt{h}} t^{ 2 \alpha+1} \mathrm{e}^{- s \mathrm{e}^{\bsi \frac{\varphi}{2}} t} \rmd t\\
 &=2\left(\frac{  \Gamma (2 \alpha + 2)}{(s \mathbf{e}^{\bsi \frac{\varphi}{2}}) ^ {2 \alpha + 2}} - \int_{\sqrt{h}}^{+ \infty} t^{ 2 \alpha+1} \mathrm{e}^{- s \mathrm{e}^{\bsi \frac{\varphi}{2}} t} \rmd t\right),
 \end{aligned}
 \end{equation}
 where in the last equality we use the Laplace transform.  %Let $I_s=\int_{\sqrt{\varepsilon}}^{+ \infty} t^{2 \alpha +1} \mathbf{e}^{- s \mathbf{e}^{\bsi \frac{\varphi}{2}}t} \rmd t$,
  Due to \eqref{eq:deltaK}, from \eqref{eq:s cond}, it can be verified that $s\Re({ \mathrm e}^{\bsi \varphi/2} ) > 2\alpha /{\mathrm e}$.  From Lemma \ref{lem:r1},  we have
 \begin{equation}\label{eq:r3}
 %|I_s| =
 \left|\int_{\sqrt{h}}^{+ \infty} t^{2 \alpha+1} \mathrm{e}^{- s \mathrm{e}^{\bsi \frac{\varphi}{2}} t} \rmd t \right| \leq \frac{2}{s \cos \frac{\varphi}{2}} \mathrm{e}^{-\frac{\sqrt{h}}{2} s \cos \frac{\varphi}{2}} \leq \frac{2}{s \delta_{\mathcal K}} \mathrm{e}^{- \frac{s \sqrt{h} \delta_{\mathcal K}}{2}}
 \end{equation}
 by using \eqref{eq:deltaK}. It is easy to see that \eqref{eq:r3} is exponentially decaying as $s\rightarrow +\infty$ . Therefore, we obtain \eqref{eq:1r1}.

The proof is complete.
\end{proof}

\begin{lem}\cite[Lemma 2.12]{DLW}\label{lem:uv}
	Let $\bmf{u}$ and $\bmf{v}$ be respectively given by \eqref{eq:u} and \eqref{eq:v}. Then the following expansion
	\begin{equation}\label{eq:uv old}
 \begin{aligned}
 \mathbf{u} \cdot  \mathbf{v} ={\mathrm e}^{-s \sqrt{ r} \mathrm e^{\bsi \varphi/2}  }  \sum_{m=0}^{\infty} & \mathrm{e}^{\bsi (m+1) \varphi}  \left[- k_p a_{m}
 J_{m+1}\left(k_{p}r\right)
  +\bsi k_s b_{m}
  J_{m+1}\left(k_{s} r\right)
  \right ]
 \end{aligned}
\end{equation}
convergences uniformly  in $S_{2h}:=\mathcal{K}\cap B_{2h}$,
 where $\mathcal K$ is defined in \eqref{eq:K}.
 \end{lem}

\begin{thm}\label{thm:Tu thm}
Let $\mathbf{u}\in L^2(\Omega)^2$ be a solution to \eqref{eq:lame}. Suppose there exit two intersecting lines $\Gamma_h^\pm$ of $\bmf{u}$ such that $\Gamma_h^\pm $ are two  traction-free lines  with the intersecting angle $\angle(\Gamma_h^+,\Gamma_h^-)=\varphi_0\neq \pi $, if $\bmf{u}(\bmf{0})=\bmf{0}$,  then $\bmf{u}\equiv \bmf{0}$.
\end{thm}

\begin{proof}
By virtue of \cite[Lemma 2.17]{DLW}, the following integral identity holds
\begin{equation}\label{eq:CGO2}
	I_3=I_1^+ + I_1^-+I_2,
\end{equation}
where
\begin{subequations}\notag
	\begin{align}
		 I_1^\pm&=\int_{\Gamma_h^\pm } \left[\left({T}_{\nu} \bmf{u}\right) \cdot \bmf{v}-\left({T}_{\nu} \bmf{v}\right) \cdot \bmf{u}\right] \mathrm{d} \sigma, \quad
		  I_2=\int_{\Lambda_h} \left[\left({T}_{\nu} \bmf{u}\right) \cdot \bmf{v}-\left({T}_{\nu} \bmf{v}\right) \cdot \bmf{u}\right] \mathrm{d} \sigma, \label{eq:CGO6a}\\
		 I_3&%=\int_{\mathcal K} \left(\mathcal{L} \bmf{u}\right) \cdot \bmf{v} \rmd x
		 =-\kappa \int_{S_h} \bmf{u} \cdot \bmf{v} \rmd \bmf{x}. \label{eq:I3 int}
		 %\quad  I_4 =-\omega^2\int_{{\mathcal K} \backslash B_h  } \bmf{u} \cdot \bmf{v} \rmd \bmf{x} .  \label{eq:CGO10}
	\end{align}
	\end{subequations}
	Furthermore, it holds that
	\begin{equation}\label{eq:I2I4}
		\begin{split}
			\left|I_2\right|&  \leqslant \mathcal{C}_{{\mathcal K },B_h,\mu,\lambda}\|\bmf{u}\|_{H^2\left( {{\mathcal K}} \cap B_h \right)}\left(1+s\right){\mathrm e}^{-\delta_{{\mathcal K} } s \sqrt{h}},
%			,\\
%			\left|I_4\right|&  \leqslant \frac{6 \varphi_{0}}{\delta_{\mathcal{K}}^{4}} s^{-4} \mathrm{e}^{-\delta_{\mathcal{K} } s \sqrt{h} / 2},
		\end{split}
	\end{equation}
	which is exponentially decays as $s\rightarrow +\infty$. Here  $\delta_{\mathcal{K}} $ is a positive constant defined in \eqref{eq:deltaK}.

By virtue of \cite[Eq. (4.18)]{DLW}, it was derived that
\begin{equation}\label{eq:EFI1}
	I_1^++I_1^-= -90 \mu (1-\mathrm{e}^{-2\bsi \varphi_0})  k_s^4 b_0s^{-6} - r_{I_1^-,2}-r_{I_1^+,2},
\end{equation}
where
\begin{equation}\label{eq:r1+}
	r_{I_1^+,2}= -\bsi s \mu\mathrm{e}^{\bsi \varphi_0/2 } \int_0^h {\mathrm e}^{s \sqrt{ r} \zeta(\varphi_0 ) } R_{2,\Gamma_h^+} \rmd r,\, 	r_{I_1^-,2} = \bsi s \mu\int_0^h {\mathrm e}^{- s \sqrt{ r}  } R_{2,\Gamma_h^-} \rmd r.
	\end{equation}
Here
\begin{equation}
\notag
\label{eq:RTvu+}
\begin{aligned}
&R_{2,\Gamma_h^+} =  r^{\frac{7}{2}}\Bigg\{  a_0 \mathrm e^{ \bsi \varphi_0} \sum_{k=2}^{\infty} \frac{(-1)^{k+1} k_p^{2k+2}}{2^{2k+1}k! (k+1)!}r^{2k-3} \\
&+ \sum_{m=1}^{2} \sum_{k=1}^{\infty} a_m \mathrm e^{ \bsi (m+1)\varphi_0} \frac{(-1)^{k+1} k_p^{2k+m+2}}{2^{2k+m+1}k! (k+m+1)!}r^{2k+m-3} \\
& + \sum_{m=3}^{\infty} \sum_{k=0}^{\infty} a_m \mathrm e^{ \bsi (m+1)\varphi_0} \frac{(-1)^{k+1} k_p^{2k+m+2}}{2^{2k+m+1}k! (k+m+1)!}r^{2k+m-3}  +\bsi  b_0 \mathrm e^{ \bsi \varphi_0} \sum_{k=2}^{\infty} \frac{(-1)^k k_s^{2k+2}}{2^{2k+1}k! (k+1)!}r^{2k-3} \\
& + \bsi \sum_{m=1}^{2} \sum_{k=1}^{\infty} b_m \mathrm e^{ \bsi (m+1)\varphi_0} \frac{(-1)^k k_s^{2k+m+2}}{2^{2k+m+1}k! (k+m+1)!}r^{2k+m-3}\\
& + \bsi \sum_{m=3}^{\infty} \sum_{k=0}^{\infty} b_m \mathrm e^{ \bsi (m+1)\varphi_0} \frac{(-1)^k k_s^{2k+m+2}}{2^{2k+m+1}k! (k+m+1)!}r^{2k+m-3}\Bigg\},
\end{aligned}
\end{equation}
%\begin{equation}\label{eq:RTvu+}
%\begin{aligned}
%R_{2,\Gamma_h^+}& =  r^{\frac{7}{2}}\Bigg\{  a_0 \mathrm e^{ \bsi \varphi_0} \sum_{k=2}^{\infty} \frac{(-1)^{k+1} k_p^{2k+2}}{2^{2k+1}k! (k+1)!}r^{2k-3} \\
%& + \sum_{m=1}^{2} \sum_{k=1}^{\infty} a_m \mathrm e^{ \bsi (m+1)\varphi_0} \frac{(-1)^{k+1} k_p^{2k+m+2}}{2^{2k+m+1}k! (k+m+1)!}r^{2k+m-3} \\
%& + \sum_{m=3}^{\infty} \sum_{k=0}^{\infty} a_m \mathrm e^{ \bsi (m+1)\varphi_0} \frac{(-1)^{k+1} k_p^{2k+m+2}}{2^{2k+m+1}k! (k+m+1)!}r^{2k+m-3} \\
%& +\bsi  b_0 \mathrm e^{ \bsi \varphi_0} \sum_{k=2}^{\infty} \frac{(-1)^k k_s^{2k+2}}{2^{2k+1}k! (k+1)!}r^{2k-3} \\
%& + \bsi \sum_{m=1}^{2} \sum_{k=1}^{\infty} b_m \mathrm e^{ \bsi (m+1)\varphi_0} \frac{(-1)^k k_s^{2k+m+2}}{2^{2k+m+1}k! (k+m+1)!}r^{2k+m-3}\\
%& + \bsi \sum_{m=3}^{\infty} \sum_{k=0}^{\infty} b_m \mathrm e^{ \bsi (m+1)\varphi_0} \frac{(-1)^k k_s^{2k+m+2}}{2^{2k+m+1}k! (k+m+1)!}r^{2k+m-3}\Bigg\},
%\end{aligned}
%\end{equation}
and
\begin{equation}\notag
\label{eq:RTvu-}
\begin{aligned}
& R_{2,\Gamma_h^-} =  r^{\frac{7}{2}}\Bigg\{  a_0  \sum_{k=2}^{\infty} \frac{(-1)^{k+1} k_p^{2k+2}}{2^{2k+1}k! (k+1)!}r^{2k-3}  + \sum_{m=1}^{2} \sum_{k=1}^{\infty} a_m  \frac{(-1)^{k+1} k_p^{2k+m+2}}{2^{2k+m+1}k! (k+m+1)!}r^{2k+m-3} \\
& + \sum_{m=3}^{\infty} \sum_{k=0}^{\infty} a_m  \frac{(-1)^{k+1} k_p^{2k+m+2}}{2^{2k+m+1}k! (k+m+1)!}r^{2k+m-3} +\bsi  b_0  \sum_{k=2}^{\infty} \frac{(-1)^k k_s^{2k+2}}{2^{2k+1}k! (k+1)!}r^{2k-3} \\
& + \bsi \sum_{m=1}^{2} \sum_{k=1}^{\infty} b_m  \frac{(-1)^k k_s^{2k+m+2}}{2^{2k+m+1}k! (k+m+1)!}r^{2k+m-3} + \bsi \sum_{m=3}^{\infty} \sum_{k=0}^{\infty} b_m  \frac{(-1)^k k_s^{2k+m+2}}{2^{2k+m+1}k! (k+m+1)!}r^{2k+m-3}\Bigg\},
\end{aligned}
\end{equation}
 By virtue of \cite[Eq. (4.19)]{DLW}, one can verify that
\begin{equation}\label{eq:54 r1}
|r_{I_1^+,2}| \leq S_3 \cdot\Oh( s^{-8} ),|r_{I_1^-,2}| \leq S_3 \cdot\Oh( s^{-8} ),
\end{equation}
as $s\rightarrow +\infty$, where
\begin{equation}
\notag
%\label{eq:RTvu+S}
\begin{aligned}
S_3 = & |a_0|   \sum_{k=2}^{\infty} \frac{ k_p^{2k+2}}{2^{2k+1}k! (k+1)!}h^{2k-3} +| b_0 |  \sum_{k=2}^{\infty} \frac{ k_s^{2k+2}}{2^{2k+1}k! (k+1)!}h^{2k-3} \\
& + \sum_{m=1}^{2}\sum_{k=1}^{\infty} |a_m|   \frac{ k_p^{2k+m+2}}{2^{2k+m+1}k! (k+m+1)!}h^{2k+m-3} \\
& + \sum_{m=1}^{2}\sum_{k=1}^{\infty} |b_m|   \frac{ k_s^{2k+m+2}}{2^{2k+m+1}k! (k+m+1)!}h^{2k+m-3} \\
& \quad + \sum_{m=3}^{\infty} \bigg| a_m \sum_{k=0}^{\infty}  \frac{(-1)^{k} k_p^{2k+m+2}}{2^{2k+m+1}k! (k+m+1)!}h^{2k+m-3} \\
&\quad  + \bsi b_m  \sum_{k=0}^{\infty} \frac{(-1)^k k_s^{2k+m+2}}{2^{2k+m+1}k! (k+m+1)!}h^{2k+m-3}\bigg |.
\end{aligned}
\end{equation}

Since $\Gamma_h^-$ is a traction-free line of $\mathbf u$, by \eqref{eq:3a}, we have $a_0=0$. Hence substituting \eqref{eq:uv old} into $I_3$, one can derive that
\begin{equation}\label{eq:r6}
	\begin{split}
		I_3 & = -\kappa \int_0^h \int_0^{\varphi_0} \mathrm{e}^{-s\sqrt{r}\mathrm{e}^{\bsi \frac{\varphi}{2}}}\sum_{m=0}^{\infty}  \mathrm{e}^{\bsi (m+1) \varphi}  \left[- k_p a_{m}
 J_{m+1}\left(k_{p}r\right)
  +\bsi k_s b_{m}
  J_{m+1}\left(k_{s} r\right)
  \right ] r \rmd r \rmd {\varphi}\\
 & =-\frac{\bsi}{2}\kappa b_0 k_s^2\int_0^{\varphi_0} \left(\mathrm{e}^{\bsi \varphi}\int_0^h \mathrm{e}^{-s\sqrt{r}\mathrm{e}^{\bsi \frac{\varphi}{2}}} r^2 \rmd r   \right)\rmd {\varphi} - I_{3,1},
	\end{split}
\end{equation}
%I_3=-\frac{\bsi}{2}\kappa b_0 k_s^2\int_0^{\varphi_0} \left(\mathrm{e}^{\bsi \varphi}\int_0^h \mathrm{e}^{-s\sqrt{r}\mathrm{e}^{\bsi \frac{\varphi}{2}}} r^2 \rmd r   \right)\rmd {\varphi} - I_{3,1} .
%$$
where
\begin{equation}\label{eq:r4}
\begin{aligned}
  I_{3,1}=&\bsi b_0 \kappa \sum_{k=1}^{+\infty}\frac{(-1)^k k_s^{2k+2}}{2^{2k+1}k!(k+1)!}\int_0^h  r^{2k+2} \mathrm{e}^{-s\sqrt{r}\mathrm{e}^{\bsi \frac{\varphi}{2}}} \rmd r \int_0^{\varphi_0} \mathrm{e}^{\bsi \varphi} \rmd {\varphi}\\
 &+ \kappa \sum_{m=1}^{+\infty}\sum_{k=0}^{+\infty}\frac{(-1)^k ( - k_p^{2k+m+2} a_m + \bsi k_s^{2k+m+2} b_m)}{2^{2k+m+1}k!(k+m+1)!}\\
 &\times
   \int_0^h  r^{2k+m+2} \mathrm{e}^{-s\sqrt{r}\mathrm{e}^{\bsi \frac{\varphi}{2}}} \rmd r  \int_0^{\varphi_0} \mathrm{e}^{\bsi (m+1) \varphi}  \rmd {\varphi}.
\end{aligned}
\end{equation}

Assume that
\begin{equation}\label{eq:s h cond}
	  s \geq \frac{8}{\mathrm e \, \delta_{\mathcal K} },\quad 0<h <\mathrm e,
\end{equation}
from \eqref{eq:1r1}, we can deduce that
\begin{equation}\label{eq:r5}
\begin{aligned}
 |I_{3,1}| & \leq |b_0| \kappa \sum_{k=1}^{+\infty}\frac{k_s^{2k+2} h^{2k-2}}{2^{2k}k!(k+1)!}  \int_0^{\varphi_0} \frac{  \Gamma (10)}{(s \mathrm{e}^{\bsi \frac{\varphi}{2}}) ^ {10}} \mathrm{e}^{\bsi \varphi} \rmd {\varphi}+ \Oh(\mathrm{e}^{- \frac{s \sqrt{h}\delta_{\mathcal K}}{2}})\\
  + &  \kappa \sum_{m=1}^{+\infty}\sum_{k=0}^{+\infty}\frac{\left  | - k_p^{2k+m+2} a_m + \bsi k_s^{2k+m+2} b_m\right| h^{2k+m-1} }{2^{2k+m}k!(k+m+1)!}  \\
  & \int_0^{\varphi_0} \frac{  \Gamma (8)}{(s \mathrm{e}^{\bsi \frac{\varphi}{2}}) ^ {8}} \mathrm{e}^{\bsi (m+1) \varphi}  \rmd {\varphi} + \Oh(\mathrm{e}^{- \frac{s \sqrt{h}\delta_{\mathcal K}}{2}}).
\end{aligned}
\end{equation}
In view of \eqref{eq:r5}, we can claim that as $s\rightarrow+\infty$
 \begin{equation}\label{eq:1r5}
 |I_{3,1}|=\Oh(s^{-8}). 
 \end{equation}
 Similarly, under \eqref{eq:s h cond}, by virtue of \eqref{eq:1r1}, one has
 \begin{equation}\label{eq:2r5}
 -\frac{\bsi}{2}\kappa b_0 k_s^2\int_0^{\varphi_0} \left(\mathrm{e}^{\bsi \varphi}\int_0^h \mathrm{e}^{-s\sqrt{r}\mathbf{e}^{\bsi \frac{\varphi}{2}}} r^2 \rmd r   \right)\rmd {\varphi}=60 \kappa b_0 k_s^2 s^{-6}(\mathrm{e}^{-2 \bsi \varphi_0}-1)+ \Oh (\mathrm{e}^{- \frac{s \sqrt{h}\delta_{\mathcal K}}{2}}).
 \end{equation}

% Combing the second equation of \eqref{eq:3a}, \eqref{eq:uv old} with \eqref{eq:I3 int} , We have
%\begin{equation}\label{eq:r6}
%\begin{aligned}
%I_3 & = -\kappa \int_0^h \int_0^{\varphi_0} \mathrm{e}^{-s\sqrt{r}\mathrm{e}^{\bsi \frac{\varphi}{2}}}\sum_{m=0}^{\infty}  \mathrm{e}^{\bsi (m+1) \varphi}  \left[- k_p a_{m}
% J_{m+1}\left(k_{p}r\right)
%  +\bsi k_s b_{m}
%  J_{m+1}\left(k_{s} r\right)
%  \right ] r \rmd r \rmd {\varphi}\\
%  & = -\frac{\bsi}{2}\kappa b_0 k_s^2 \int_0^h \mathrm{e}^{-s\sqrt{r}\mathrm{e}^{\bsi \frac{\varphi}{2}}} r^2 \rmd r \int_0^{\varphi_0} \int_0^{\varphi_0}\mathrm{e}^{\bsi \varphi} \rmd {\varphi} - I_{3,1}\\
%\end{aligned}
%\end{equation}
%where
%\begin{equation}\nonumber
%\begin{aligned}
%  I_{3,1}=&\bsi b_0 \kappa \sum_{k=1}^{+\infty}\frac{(-1)^k k_s^{2k+2}}{2^{2k+1}k!(k+1)!}\int_0^h  r^{2k+2} \mathrm{e}^{-s\sqrt{r}\mathrm{e}^{\bsi \frac{\varphi}{2}}} \rmd r \int_0^{\varphi_0} \mathrm{e}^{\bsi \varphi} \rmd {\varphi}\\
% &+ \kappa \sum_{m=1}^{+\infty}\sum_{k=0}^{+\infty}\frac{(-1)^k ( - k_p^{2k+m+2} a_m + \bsi k_s^{2k+m+2} b_m)}{2^{2k+m+1}k!(k+m+1)!}\\
% &\times
%   \int_0^h  r^{2k+m+2} \mathrm{e}^{-s\sqrt{r}\mathrm{e}^{\bsi \frac{\varphi}{2}}} \rmd r  \int_0^{\varphi_0} \mathrm{e}^{\bsi (m+1) \varphi}  \rmd {\varphi}.
%\end{aligned}
%\end{equation}

Substituting \eqref{eq:1r5}, \eqref{eq:2r5} into \eqref{eq:r6}, we can obtain that
\begin{equation}\label{eq:r10}
I_3= 60 \kappa b_0 k_s^2 s^{-6}(\mathrm{e}^{-2 \bsi \varphi_0}-1)+ \Oh (\mathrm{e}^{- \frac{s \sqrt{h}\delta_{\mathcal K}}{2}}) - \Oh(s^{-8}),
\end{equation}
as  $s\rightarrow+\infty$.

Substituting \eqref{eq:EFI1} and  \eqref{eq:r10}  into \eqref{eq:CGO2},  multiplying $s^6$ on both sides of the resulting \eqref{eq:CGO2}, by virtue of \eqref{eq:I2I4} and \eqref{eq:54 r1}, we can deduce that
\begin{equation}\label{eq:433 b0eq}
60 \kappa b_0 k_s^2 (\mathbf{e}^{-2 \bsi \varphi_0}-1)+90 \mu (1-\mathrm{e}^{-2 \bsi \varphi_0})k_s^4 b_0=0
\end{equation}
by letting $s\rightarrow +\infty$. From \eqref{eq:433 b0eq}, it yields that
$$
(\mathrm{e}^{-2 \bsi \varphi_0}-1) k_s^2 b_0 (2 \kappa - 3 \mu k_s^2)=0.
$$
Since $k_s=\sqrt{\frac{\kappa}{\mu}}$ and $\kappa\in \mathbb R_+$, it is readily known that $b_0=0$.

%$s^6 \times$ \eqref{eq:r10} and let $s\rightarrow+\infty$, we have
%\begin{equation}\label{eq:r7}
%s^6 I_3 = 60 \kappa b_0 k_s^2 (\mathbf{e}^{-2 \bsi \varphi_0}-1),
%\end{equation}
%Similarly, $s^6 \times$ \eqref{eq:EFI1} and let $s\rightarrow+\infty$, we deduce that
%\begin{equation}\label{eq:r8}
%s^6 (I_1^+ +I_1^-)=-90 \mu (1-\mathbf{e}^{-2 \bsi \varphi_0})k_s^4 b_0,
%\end{equation}
%Combing \eqref{eq:CGO2}, \eqref{eq:I2I4},\eqref{eq:r7} with \eqref{eq:r8}, we obtain
%
%that is,
%
%where $k_s=\sqrt{\frac{\kappa}{\mu}}$ is defined in \eqref{eq:kpks}.
%Therefore, we can derive that
%$$
% \kappa (\mathbf{e}^{-2 \bsi \varphi_0}-1) k_s^2 b_0 =0.
%$$
%Apparently, when $\kappa\neq0$, we have $b_0=0$. Since $\kappa\in \mathbb{R}^+$, it is obvious that $\kappa \neq 0$ .

The remaining proof is the same as that of \cite[Theorem 4.2]{DLW} and we skip it.

The proof is complete.
\end{proof}

In Theorems \ref{thm:soft thm}-\ref{thm:forth&third exp} in what follows, we shall consider the microlocal singularities of $\mathbf u$ for two intersecting homogeneous lines of $\mathbf u$ introduced in Definition~\ref{def:1}.

\subsection{The case that $\Gamma_h^+$ is a soft-clamped line}
\begin{thm}\label{thm:soft thm}
Let $\mathbf{u}\in L^2(\Omega)^2$ be a solution to \eqref{eq:lame}. If there exist two intersecting lines $\Gamma_h^\pm$ of $\bmf{u}$ such that $\Gamma_h^-$ is a rigid line and $\Gamma_h^+$ is a soft-clamped line with the intersecting angle $\angle(\Gamma_h^+,\Gamma_h^-)=\varphi_0\neq \pi $, then $\bmf{u}\equiv \bmf{0}$.
\end{thm}

\begin{proof}
Since $\Gamma_h^- \in {\mathcal R}_\Omega^{\kappa}$ and $\Gamma_h^+ \in {\mathcal G}_\Omega^{\kappa}$, according to Lemmas \ref{lem:third} and  \ref{lem:rigid}, and combining the first equation of \eqref{eq:1} with the second equation of \eqref{eq:2}, we have
\begin{equation}
\notag
\label{eq:7a}
 \bsi k_p a_1 - k_s b_1=0,\quad
 \bsi k_p^3 a_1- k_s^3 b_1=0,
\end{equation}
which can be used to prove that $a_1=b_1=0$ by noting
%The determinant of the coefficient matrix of \eqref{eq:7a} is
$$
\left| \begin{matrix}
	\bsi k_p & - k_s\\
	\bsi k_p^3 & - k_s^3
\end{matrix} \right|=\bsi k_p k_s (k_p^2 - k_s^2) \neq 0.
$$
%then we know that
%\begin{equation}\label{eq:8a}
%a_1=b_1=0.
%\end{equation}
Similarly, combining the second equation of \eqref{eq:1} with the first equation of \eqref{eq:2}, it yields that
\begin{equation} \notag
\label{eq:9a}
 \bsi k_p^2 a_2 - k_s^2 b_2=0,\quad
 2 k_s^2 b_0+ (\bsi k_p^2 a_2 - k_s^2 b_2)\mathrm{e}^{2 \bsi \varphi_0}=0,
  \end{equation}
which further gives $b_0=0$. Substituting $b_0=0$ into  the second equation of \eqref{eq:2a}, it is easy to see $a_0=0$.

Since we have proved that $a_\ell=b_\ell=0$ for $\ell=0,1$, by virtue of Lemma \ref{lem:rigid}, we have $a_\ell=b_\ell=0$,  $\forall \ell \in \mathbb N \cup \{0\}.$ According to Proposition \ref{prop:1}, one has $\mathbf u \equiv  \mathbf 0$ in $\Omega$.
%\begin{equation}\label{eq:11a}
%a_0=0.
%\end{equation}
%According to \cite[Lemma 3.1]{DLW}, we complete the proof.
\end{proof}

\begin{thm}\label{thm:Tu&third exp}
Let $\mathbf{u}\in L^2(\Omega)^2$ be a solution to \eqref{eq:lame}. If there exist two intersecting lines $\Gamma_h^\pm$ of $\bmf{u}$ such that $\Gamma_h^-$ is a traction-free line and $\Gamma_h^+$ is a soft-clamped line with the intersecting angle $\angle(\Gamma_h^+,\Gamma_h^-)=\varphi_0\neq \pi $, then $\bmf{u}\equiv \bmf{0}$.
\end{thm}

\begin{proof}
Since $\Gamma_h^+ \in {\mathcal G}_\Omega^{\kappa}$, from Lemma \ref{lem:third}, we have $b_0=0$.  Since $\Gamma_h^- \in {\mathcal T}_\Omega^{\kappa}$ and $\Gamma_h^+ \in {\mathcal G}_\Omega^{\kappa}$, from Lemmas \ref{lem:third} and  \ref{lem:tranction free}, one can show that $b_1=0$ by using the first equation of \eqref{eq:1} and the second equation of \eqref{eq:4a}. Therefore, by virtue of the second equations of \eqref{eq:3a} and \eqref{eq:4a}, we have shown that $a_\ell=b_\ell=0$ for $\ell=0,1$.  According to Lemma \ref{lem:tranction free} and Proposition \ref{prop:1}, the conclusion of the theorem can be readily obtained.
\end{proof}

\begin{thm}\label{thm:Tu+&third exp}
Let $\mathbf{u}\in L^2(\Omega)^2$ be a solution to \eqref{eq:lame}. If there exist two intersecting lines $\Gamma_h^\pm$ of $\bmf{u}$ such that $\Gamma_h^-$ is an impedance line and $\Gamma_h^+$ is a soft-clamped line with the intersecting angle $\angle(\Gamma_h^+,\Gamma_h^-)=\varphi_0\neq \pi $, then $\bmf{u}\equiv \bmf{0}$.
\end{thm}

\begin{proof}
Due to $\Gamma_h^- \in {\mathcal I}_\Omega^{\kappa}$ and $\Gamma_h^+ \in {\mathcal G}_\Omega^{\kappa}$, from Lemmas \ref{lem:third} and \ref{lem:impedance line}, we have $b_1=0$ by substituting the second equation of \eqref{eq:6a} into the first equation of \eqref{eq:1}. Similarly, since  $\Gamma_h^+ \in {\mathcal G}_\Omega^{\kappa}$, from Lemma \ref{lem:third}, we have $b_0=0$. By virtue of the second equations of \eqref{eq:5a} and \eqref{eq:6a}, we have $a_\ell=b_\ell$ for $\ell=0,1$. According to Lemma \ref{lem:impedance line} and Proposition \ref{prop:1},  the conclusion of the theorem can be readily obtained.
\end{proof}

%\begin{thm}\label{thm:third&third exp}
%Let $\mathbf{u}\in L^2(\Omega)^2$ be a solution to \eqref{eq:lame}. Recall that $\mathbf{u}$ has the Fourier expansion \eqref{eq:u} around the origin.
%If there exist two intersecting lines $\Gamma_h^\pm$ of $\bmf{u}$  such that   $\Gamma_h^\pm$ are soft-clamped lines and $a_0=0$, where  the intersecting angle $\angle(\Gamma_h^+,\Gamma_h^-)=\varphi_0\neq \pi $, then $\bmf{u}\equiv \bmf{0}$.
%\end{thm}
%
%\begin{proof}
%	Since $\Gamma_h^+$ is a soft-clamped line, from \eqref{eq:a1b1} in Lemma \ref{lem:third}, under the condition $a_0=0$, one has $a_\ell=b_\ell=0$ for $\ell=0,1$. The remaining proof is same as the proof of Theorem \ref{thm:third}. 
%\end{proof}

\subsection{The case that $\Gamma_h^+$ is a simply-supported line.}

\begin{thm}\label{thm:u&forth exp}
Let $\mathbf{u}\in L^2(\Omega)^2$ be a solution to \eqref{eq:lame}. If there exist two intersecting lines $\Gamma_h^\pm$ of $\bmf{u}$ such that $\Gamma_h^-$ is a rigid line and $\Gamma_h^+$ is a simply-supported line with the intersecting angle $\angle(\Gamma_h^+,\Gamma_h^-)=\varphi_0\neq \pi $
 , then $\bmf{u}\equiv \bmf{0}$.
\end{thm}

\begin{proof}
Since $\Gamma_h^+$ is a simply-supported line, from  Lemma \ref{lem:forth}, combining the first equation of \eqref{eq:16} with the first equation of \eqref{eq:17}, we know that
\begin{equation}\label{eq:16a}
a_0=0.
\end{equation}
Since $\Gamma_h^+$ is a simply-supported line, from Lemma \ref{lem:rigid}, substituting   \eqref{eq:16a} into the second equation of \eqref{eq:2a}, we obtain $
%\begin{equation}\label{eq:17a}
 b_0=0.
%\end{equation}
$

Again using the fact that $\Gamma_h^+ \in \mathcal F^\kappa_\Omega$ and $\Gamma_h^+ \in \mathcal R^\kappa_\Omega$, combining the second equation of \eqref{eq:17} in Lemma \ref{lem:rigid} with the first equation of \eqref{eq:25} in Lemma \ref{lem:forth},
 we obtain that
\begin{equation}\label{eq:18a}
 a_1=0.
\end{equation}
Substituting    \eqref{eq:18a} into the second equation of \eqref{eq:16}, we obtain $b_1=0$.

By now we have shown $a_\ell=b_\ell=0$ for $\ell=0,1$. According to Lemma \ref{lem:rigid} and Proposition \ref{prop:1}, the conclusion of the theorem can be readily obtained.
\end{proof}

\begin{thm}\label{thm:Tu&forth exp}
Let $\mathbf{u}\in L^2(\Omega)^2$ be a solution to \eqref{eq:lame}.  If there exist two intersecting lines $\Gamma_h^\pm$ of $\bmf{u}$ such that $\Gamma_h^-$ is a traction-free line and $\Gamma_h^+$ is a simply-supported line with the intersecting angle $\angle(\Gamma_h^+,\Gamma_h^-)=\varphi_0\neq \pi $, then $\bmf{u}\equiv \bmf{0}$.
\end{thm}

\begin{proof}

Since $\Gamma_h^- \in {\mathcal T}_\Omega^{\kappa}$ and $\Gamma_h^+ \in {\mathcal F}_\Omega^{\kappa}$, from Lemma \ref{lem:forth} and \ref{lem:tranction free}, we have $b_1=0$ by substituting the second equation of \eqref{eq:4a} into the second equation of \eqref{eq:16}. Similarly, since  $\Gamma_h^- \in {\mathcal T}_\Omega^{\kappa}$, from Lemma \ref{lem:34} and \ref{lem:tranction free}, we have $b_0=0$ by substituting the second equation of \eqref{eq:3a} into the \eqref{eq:gradient3}. By virtue of the second equations of \eqref{eq:3a} and \eqref{eq:4a}, we have $a_\ell=b_\ell$ for $\ell=0,1$. According to Lemma \ref{lem:tranction free} and Proposition \ref{prop:1}, the conclusion of the theorem can be readily obtained.
\end{proof}

\begin{thm}\label{thm:Tu+&forth exp}
Let $\mathbf{u}\in L^2(\Omega)^2$ be a solution to \eqref{eq:lame}. If there exist two intersecting lines $\Gamma_h^\pm$ of $\bmf{u}$ such that $\Gamma_h^-$ is a impedance line and $\Gamma_h^+$ is a simply-supported line with the intersecting angle $\angle(\Gamma_h^+,\Gamma_h^-)=\varphi_0\neq \pi $, then $\bmf{u}\equiv \bmf{0}$.
\end{thm}

\begin{proof}
Since $\Gamma_h^- \in {\mathcal I}_\Omega^{\kappa}$ and $\Gamma_h^+ \in {\mathcal F}_\Omega^{\kappa}$, from Lemma \ref{lem:forth} and \ref{lem:impedance line}, we have $b_1=0$ by substituting the second equation of \eqref{eq:6a} into the second equation of \eqref{eq:16}. Similarly, since  $\Gamma_h^- \in {\mathcal I}_\Omega^{\kappa}$, from Lemma \ref{lem:34} and \ref{lem:impedance line}, we have $b_0=0$ by substituting the second equation of \eqref{eq:5a} into the \eqref{eq:gradient3}. By virtue of the second equations of \eqref{eq:5a} and \eqref{eq:6a}, we have $a_\ell=b_\ell$ for $\ell=0,1$. According to Lemma \ref{lem:impedance line} and Proposition \ref{prop:1}, the conclusion of the theorem can be readily obtained.
\end{proof}

The following lemma considers the case that $\Gamma_h^-$ is a soft-clamped line of $\mathbf u$, which shall be needed in the proofs of Theorems %\ref{thm:third&third exp}, 
\ref{thm:forth&third exp} and \ref{thm:GI&third exp} in what follows. 

\begin{lem}\label{lem:third1}
Let $\mathbf u$ be a Lam\' e eigenfunction to \eqref{eq:lame}, where $\mathbf u$ has the Fourier expansion \eqref{eq:u} at the origin.  Consider  $\Gamma_h^-$ defined in \eqref{eq:gamma_pm} such that $\Gamma_h^-\in {\mathcal G}_\Omega^{\kappa}$. Then we have the following equations:
\begin{subequations}
	\begin{align}
		&\bsi k_p a_1 -  k_s b_1=0,\quad
 \bsi k_p^2 a_2 - k_s^2 b_2 = 0,\label{eq:15a}\\
&		\label{eq:15b}
  2 k_s ^2 b_0+  (\bsi k_p^2 a_2- k_s^2  b_2)=0,\quad
   \bsi k_p ^3 a_1 - k_s^3 b_1 - (\bsi k_p ^3 a_3 - k_s^3 b_3)=0,\\
   \label{eq:15c}
   &\bsi k_p^3 a_1 + 3 k_s^3 b_1 + ( \bsi k_p^3 a_3 - k_s^3 b_3) =0,\quad
    2 \bsi k_p^4 a_2 - 2 k_s^4 b_2 - ( \bsi k_p^4 a_4 - k_s^4 b_4) =0,
	\end{align}
\end{subequations}
%\begin{equation}\label{eq:15a}
%\bsi k_p a_1 -  k_s b_1=0,\quad
% \bsi k_p^2 a_2 - k_s^2 b_2 = 0,
%\end{equation}
%\begin{equation}\label{eq:15b}
%  2 k_s ^2 b_0+  (\bsi k_p^2 a_2- k_s^2  b_2)=0,\quad
%   \bsi k_p ^3 a_1 - k_s^3 b_1 - (\bsi k_p ^3 a_3 - k_s^3 b_3)=0,
%\end{equation}
%\begin{equation}\label{eq:15c}
%   \bsi k_p^3 a_1 + 3 k_s^3 b_1 + ( \bsi k_p^3 a_3 - k_s^3 b_3) =0,\quad
%    2 \bsi k_p^4 a_2 - 2 k_s^4 b_2 - ( \bsi k_p^4 a_4 - k_s^4 b_4) =0,
%\end{equation}
and
\begin{equation}\label{eq:15d}
 \bigg\{  \begin{array}{l} -2\bsi k_p^4 a_2 + 4 k_s^4 b_2 + (\bsi k_p^4 a_4 - k_s^4 b_4) =0,\\
 3 \bsi k_p^5 a_3 -3 k_s^5 b_3 -  ( \bsi k_p^5 a_5 - k_s^5 b_5)  =0.
  \end{array}
\end{equation}
Moreover, it holds that
\begin{equation} \label{eq:15e}
b_0=a_1=b_1=a_2=b_2=0.
\end{equation}
\end{lem}
\begin{proof}
Since $\Gamma_h^-\in {\mathcal G}_\Omega^{\kappa}$, using \eqref{eq:third3}, we have
 \begin{equation}\label{eq:446 b3}
 \begin{aligned}
\mathbf 0 &=  \sum_{m=0}^{\infty}  \bigg\{ \left[- \frac{\bsi k_p}{2} a_{m} \big (J_{m-1}(k_p r) + J_{m+1}(k_p r)\big )\right.\\
   & +  \frac{k_s}{2} b_m \big (J_{m-1}(k_s r) - J_{m+1}(k_s r) \big )\bigg]\hat{\mathbf{e}}_1 + \left[ \frac{\bsi k_p^2}{2} a_{m} \big (-J_{m-2}(k_p r)\right.\\
    &  + J_{m+2}(k_p r) \big  )
  + \frac{k_s^2}{2} b_m \big (J_{m-2}(k_s r) + J_{m+2}(k_s r) \big )\bigg] \mu \hat{\mathbf{e}}_2 \bigg\},
 \end{aligned}
\end{equation}
By Lemma \ref{lem:co exp},  comparing the coefficients of $r^0$, $r^1$, $r^2$ and $r^3$ on both sides of  \eqref{eq:446 b3}, one can obtain \eqref{eq:15a}, \eqref{eq:15b}, \eqref{eq:15c} and \eqref{eq:15d}, respectively. \eqref{eq:15e} can be derived by using \eqref{eq:15a}, \eqref{eq:15b}, \eqref{eq:15c} and \eqref{eq:15d}.
\end{proof}

By virtue of \eqref{eq:15e} in Lemma \ref{lem:third1} and \eqref{eq:a0 forth} in Lemma \ref{lem:forth}, for two intersecting soft-clamped and simply-supported lines,   one can readily have
\begin{thm}\label{thm:forth&third exp}
Let $\mathbf{u}\in L^2(\Omega)^2$ be a solution to \eqref{eq:lame}. %Recall that $\mathbf{u}$ has the Fourier expansion \eqref{eq:u} around the origin.
If there exist two intersecting lines $\Gamma_h^\pm$ of $\bmf{u}$  such that   $\Gamma_h^-$ is a soft-clamped line and $\Gamma_h^+$ is a simply-supported line, where  the intersecting angle $\angle(\Gamma_h^+,\Gamma_h^-)=\varphi_0\neq \pi $, then $\bmf{u}\equiv \bmf{0}$.
\end{thm}

%{\color{red} add some words related the proof of the following theorem}
%
%Suppose that $\Gamma_h^+$ is a soft-clamped line, from Lemma \ref{lem:forth}, we have $a_0=0$. Assume that $b_0=0$, according to 
%
%\begin{thm}\label{thm:fourth&fourth exp}
%Let $\mathbf{u}\in L^2(\Omega)^2$ be a solution to \eqref{eq:lame}. Recall that $\mathbf{u}$ has the Fourier expansion \eqref{eq:u} around the origin.
%If there exist two intersecting lines $\Gamma_h^\pm$ of $\bmf{u}$  such that   $\Gamma_h^\pm$ are simply-connected lines and $b_0=0$, where  the intersecting angle $\angle(\Gamma_h^+,\Gamma_h^-)=\varphi_0\neq \pi $, then $\bmf{u}\equiv \mathbf 0$. 
%\end{thm} 
%
%
%\begin{proof}
%	Since $\Gamma_h^+$ is a soft-clamped line, from \eqref{eq:a1b1} in Lemma \ref{lem:third}, under the condition $a_0=0$, one has $a_\ell=b_\ell=0$ for $\ell=0,1$. The remaining proof is same as the proof of Theorem \ref{thm:third}. 
%\end{proof}

\subsection{The case that $\Gamma_h^+$ is a generalized-impedance line.}

\begin{thm}\label{thm:u&GI exp}
Let $\mathbf{u}\in L^2(\Omega)^2$ be a solution to \eqref{eq:lame} and have the Fourier expansion \eqref{eq:u}. Suppose that  there exist two intersecting lines $\Gamma_h^\pm$ of $\bmf{u}$ such that $\Gamma_h^-$ is a rigid line and $\Gamma_h^+$ is a generalized-impedance line with the intersecting angle $\angle(\Gamma_h^+,\Gamma_h^-)=\varphi_0\neq \pi $, where the associated impedance parameter $\boldsymbol{ \eta}_1$ to $\Gamma_h^+$ satisfies  \eqref{eq:eta1 ex}.  If either 
\begin{equation}\label{eq:cond eta1}
	\eta_1\neq -\frac{\bsi \mu \mathrm{e}^{2 \bsi \varphi_0}}{\lambda+\mu(1+\mathrm{e}^{2 \bsi \varphi_0})}
\end{equation}
where $\eta_1$ is the constant part of $\boldsymbol{\eta}_1$ in \eqref{eq:eta1 ex}, or
\begin{equation}\label{eq:cond a0 eta1}
	a_0= 0, \quad \eta_1\neq \mathrm i, 
\end{equation} 
then $\bmf{u}\equiv \bmf{0}$.
\end{thm}
\begin{proof}
Since  $\Gamma_h^- \in {\mathcal R}_\Omega^{\kappa}$, from \eqref{eq:1a} of Lemma \ref{lem:rigid}, we can prove that $a_1=b_1=0$ by noting
$$
\bigg|\begin{array}{cc}
k_p   & \bsi k_s\\
k_p^3 &  -\bsi k_s^3
\end{array}
\bigg|=-\bsi k_p k_s (k_p^2+k_s^2)\neq 0.
$$

Due to that $\Gamma_h^- \in {\mathcal R}_\Omega^{\kappa}$ and $\Gamma_h^+ \in {\mathcal H}_\Omega^{\kappa}$, from Lemmas \ref{lem:condition} and \ref{lem:rigid}, combining the second equation of \eqref{eq:1b}  with the first equation of \eqref{eq:2a}, we have
\begin{equation}\label{eq:25a}
 \bigg\{  \begin{array}{l} k_p ^2 a_0+ \bsi k_s^2 b_0- k_p^2  a_2- \bsi k_s ^2 b_2=0,\\
  2 \eta_1 (\lambda+\mu) k_p^2 a_0 + (1 - \bsi \eta_1) ( \bsi k_p^2 a_2 -  k_s^2 b_2) \mu \mathrm{e}^{2 \bsi \varphi_0}=0.
% \\
% 3 k_p^2 a_0 + \bsi k_s^2  b_0 =0,
  \end{array}
\end{equation}
Multiplying $\mathrm i (1 - \bsi \eta_1) \mu \mathrm{e}^{2 \bsi \varphi_0} $ on both sides of the first equation of \eqref{eq:25a}, then adding it to the second one of \eqref{eq:25a}, one has
\begin{align} \label{eq:26a}
\left(  2 \eta_1 (\lambda+\mu)+\mathrm i  (1 - \bsi \eta_1) \mu \mathrm{e}^{2 \bsi \varphi_0}\right)	k_p ^2 a_0- (1 - \bsi \eta_1) \mu \mathrm{e}^{2 \bsi \varphi_0} k_s^2 b_0=0.
\end{align}
If \eqref{eq:cond eta1} is fulfilled, combining \eqref{eq:26a} with the second equation of \eqref{eq:2a}, one can prove that $a_0=b_0=0$ by using the fact that
\begin{equation}\nonumber
\begin{aligned}
&\left| \begin{matrix}
\left(  2 \eta_1 (\lambda+\mu)+\mathrm i  (1 - \bsi \eta_1) \mu \mathrm{e}^{2 \bsi \varphi_0}\right) k_p^2 & - (1 - \bsi \eta_1) \mu \mathrm{e}^{2 \bsi \varphi_0}k_s^2 \\
  k_p^2& -\bsi k_s^2
\end{matrix} \right|\\
&=-2\bsi k_p^2 k_s^2 \left[(\lambda+\mu(1+\mathrm{e}^{2 \bsi \varphi_0}))\eta_1 + \bsi \mu \mathrm{e}^{2 \bsi \varphi_0}\right]\neq0.
\end{aligned}
\end{equation}
%since  $\eta\neq -\frac{\bsi \mu \mathrm{e}^{2 \bsi \varphi_0}}{\lambda+\mu(1+\mathrm{e}^{2 \bsi \varphi_0})}$.

Considering the case \eqref{eq:cond a0 eta1}, from \eqref{eq:25a}, it can be directly shown that $b_0=0$. 

For two separate cases \eqref{eq:cond eta1} and \eqref{eq:cond a0 eta1},  by now we have shown $a_{\ell}=b_{\ell}=0$ for $\ell =0,1$. According to Lemma \ref{lem:rigid} and Proposition \ref{prop:1}, the conclusion of the theorem can be readily obtained.
\end{proof}

\begin{thm}\label{thm:Tu&GI exp}
Let $\mathbf{u}\in L^2(\Omega)^2$ be a solution to \eqref{eq:lame}. If there exist two intersecting lines $\Gamma_h^\pm$ of $\bmf{u}$  such that $\Gamma_h^-$ is a traction-free line and $\Gamma_h^+$ is a generalized-impedance line with the intersecting angle $\angle(\Gamma_h^+,\Gamma_h^-)=\varphi_0\neq \pi $, where the associated impedance parameter $\boldsymbol{ \eta}_1$ to $\Gamma_h^+$ has the expansion \eqref{eq:eta1 ex} and satisfies  $\eta_1\neq \bsi$, then $\bmf{u}\equiv \bmf{0}$.
\end{thm}
\begin{proof}
Since $\Gamma_h^- \in {\mathcal T}_\Omega^{\kappa}$ and $\Gamma_h^+ \in {\mathcal H}_\Omega^{\kappa}$, from Lemma \ref{lem:34} and \ref{lem:tranction free}, we have $b_0=0$ by substituting the second equation of \eqref{eq:3a} into  \eqref{eq:gradient3}. Similarly, since  $\Gamma_h^- \in {\mathcal T}_\Omega^{\kappa}$ and $\Gamma_h^+ \in {\mathcal H}_\Omega^{\kappa}$, from  Lemma \ref{lem:tranction free}, we have $b_1=0$ by substituting the second equation of \eqref{eq:4a} into the first equation of \eqref{eq:1b}. By virtue of the second equations of \eqref{eq:3a} and \eqref{eq:4a}, we have $a_\ell=b_\ell$ for $\ell=0,1$. The proof can be readily concluded by using Lemma \ref{lem:tranction free} and Proposition \ref{prop:1}.
\end{proof}

\begin{thm}\label{thm:Tu+&GI exp}
{ Let $\mathbf{u}\in L^2(\Omega)^2$ be a solution to \eqref{eq:lame}.  If there exist two intersecting lines $\Gamma_h^\pm$ of $\bmf{u}$  such that $\Gamma_h^-$ is an impedance line associated with the paramter $\boldsymbol{\eta}_1$ satisfying \eqref{eq:eta1 ex}  and $\Gamma_h^+$ is a generalized-impedance line associated with the parameter $\boldsymbol{\eta}_2$ satisfying \eqref{eq:eta2 ex}, where the intersecting angle $\angle(\Gamma_h^+,\Gamma_h^-)=\varphi_0\neq \pi $, $\eta_2\neq \bsi$, then $\bmf{u}\equiv \bmf{0}$.}
\end{thm}

\begin{proof}
Since $\Gamma_h^- \in {\mathcal I}_\Omega^{\kappa}$, using Lemma \ref{lem:impedance line},  one has $a_0=a_1=0$. Due to  $\Gamma_h^+ \in {\mathcal H}_\Omega^{\kappa}$, from Lemma \ref{lem:g imp}, under the assumption $\eta_2 \neq \mathrm i$, we have $b_1=0$ by substituting $a_1=0$ into  \eqref{eq:1b}. Similarly, substituting $a_0=0$ into  \eqref{eq:1b} in Lemma \ref{lem:g imp}, under the assumption $\eta_2 \neq \mathrm i$, we have
\begin{equation}\label{eq:i g intect}
	\mathrm i k_p^2 a_2-k_s^2 b_2=0.
\end{equation}
Substituting \eqref{eq:i g intect} and $a_1=0$ into \eqref{eq:5a} in Lemma \ref{lem:impedance line}, by using the assumption $\eta_1\neq 0$, we have $b_1=0$. Therefore we obtain that $a_\ell=b_\ell$ for $\ell=0,1$. According to Lemma \ref{lem:impedance line},  the proof can be readily concluded.
\end{proof}

By virtue of \eqref{eq:15e} in Lemma \ref{lem:third1} and the second equation of \eqref{eq:1b} in Lemma \ref{lem:g imp}, for two intersecting soft-clamped and generalized-impedance lines,   one can readily to have
\begin{thm}\label{thm:GI&third exp}
Let $\mathbf{u}\in L^2(\Omega)^2$ be a solution to \eqref{eq:lame}. %Recall that $\mathbf{u}$ has the Fourier expansion \eqref{eq:u} around the origin.
If there exist two intersecting lines $\Gamma_h^\pm$ of $\bmf{u}$  such that   $\Gamma_h^-$ is a soft-clamped line and $\Gamma_h^+$ is a generalized-impedance line, where  the intersecting angle $\angle(\Gamma_h^+,\Gamma_h^-)=\varphi_0\neq \pi $, then $\bmf{u}\equiv \bmf{0}$.
\end{thm}

In Theorem \ref{thm:GI&GI} we shall consider two intersecting generalized-impedance lines. Before that, we need the following two lemmas. 

\begin{lem}\label{lem:48}
Suppose that $\Gamma_h^-$ is a generalized-impedance line of $\mathbf u$ associated with the impedance parameter $\boldsymbol{\eta}_2$ satisfying \eqref{eq:eta2 ex}. Then we have
\begin{equation}\label{eq:1b1}
\bigg\{
    \begin{array}{l}
  (1+ \bsi \eta_2)( \bsi k_p a_1 -  k_s b_1) =0,\\
  2 \eta_2 (\lambda+\mu) k_p^2 a_0 + (1 - \bsi \eta_2) ( \bsi k_p^2 a_2 -  k_s^2 b_2) \mu =0,
 \end{array}
 \end{equation}
 and
 \begin{equation}\label{eq:g3}
\bigg\{
    \begin{array}{l}
  2 k_s^2 b_0 +  (1+\bsi \eta_2) (\bsi k_p^2 a_2 - k_s^2 b_2) =0,\\
  (\bsi \mu - 2 \lambda \eta_2 - \mu \eta_2) k_p^3 a_1 + \mu (\bsi \eta_2 - 1) k_s^3 b_1  - \mu ( 1 - \bsi \eta_2) (\bsi k_p^3 a^3 - k_s^3 b_3) =0.
 \end{array}
\end{equation}
\end{lem}
\begin{proof}
Since $\Gamma_h^-$ is a generalized-impedance line of $\bmf{u}$, using \eqref{eq:third3} and \eqref{eq:forth3}, we have
\begin{equation}\label{eq:GI2}
 \begin{aligned}
\bmf{0} = & \sum_{m=0}^{\infty}  \bigg\{ \left[ - \frac{\bsi k_p}{2} a_{m} (J_{m-1}(k_p r) + J_{m+1}(k_p r) )\right.\\
   & +  \frac{k_s}{2} b_m (J_{m-1}(k_s r) - J_{m+1}(k_s r) )\bigg]\hat{\mathbf{e}}_1 + \left[ \frac{\bsi k_p^2}{2} a_{m} (-J_{m-2}(k_p r) + J_{m+2}(k_p r) )\right.\\
   & + \frac{k_s^2}{2} b_m (J_{m-2}(k_s r) + J_{m+2}(k_s r) )\bigg] \mu \hat{\mathbf{e}}_2 \bigg\}\\
   &+\boldsymbol{ \eta}_2  \bigg\{ \left[ \frac{ k_p}{2} a_{m} (  J_{m-1}(k_p r) - J_{m+1}(k_p r) )\right.
    +  \frac{\bsi k_s}{2} b_m (J_{m-1}(k_s r) + J_{m+1}(k_s r) )\bigg]\hat{\mathbf{e}}_1 \\
   & + \left[- \frac{ k_p^2}{2} a_{m} (J_{m-2}(k_p r)\mu + 2(\lambda+\mu) J_m(k_p r)+ J_{m+2}(k_p r) \mu)\right.\\
   & - \frac{\bsi k_s^2}{2} b_m (J_{m-2}(k_s r) - J_{m+2}(k_s r) )\mu \bigg]  \hat{\mathbf{e}}_2 \bigg\},\quad 0 \leqslant r\leqslant h,
 \end{aligned}
\end{equation}
where $\hat{\mathbf{e}}_1=(0,1)^\top $, $\hat{\mathbf{e}}_2=(1,0)^\top $. Using  Lemma \ref{lem:co exp} and \eqref{eq:Jmfact}, by comparing the coefficients of the terms $r^0$ in both sides of \eqref{eq:GI2} and by using \eqref{eq:eta2 ex}, 
 it holds that
\begin{equation}\nonumber
\begin{aligned}
&\left[(\bsi - \eta_2) k_p a_1 - (1+ \bsi \eta_2) k_s b_1\right]  \hat{\mathbf{e}}_1\\
&- \left[ 2 \eta_2 (\lambda+\mu) k_p^2 a_0 +\left((\bsi + \eta_2) k_p^2 a_2 - (1 - \bsi \eta_2) k_s^2 b_2\right) \mu\right] \hat{\mathbf{e}}_2=\bmf{0}.
\end{aligned}
\end{equation}
Since $\hat{\mathbf{e}}_1$ and $\hat{\mathbf{e}}_2$ are linearly independent, we can obtain \eqref{eq:1b1}.

Similarly, by comparing the coefficients of the terms $r^1$ in both sides of \eqref{eq:GI2} and by using \eqref{eq:eta2 ex},  we can deduce \eqref{eq:g3}. 
\end{proof}

\begin{lem}\label{lem:49}
Let $k_p$ and $k_s$ be defined in \eqref{eq:kpks}, where $\kappa \in \mathbb R_+$ is the eigenvalue to \eqref{eq:lame}.  For any $m \in \mathbb N$ and Lam\'e constants $\lambda$, $\mu$ satisfying \eqref{eq:convex}, if $\eta\in \mathbb C$ is a constant such that 
\begin{equation}\label{eq:cond eta m}
	\eta\neq % \pm \bsi \mbox{ and } 
-\frac{\bsi m}{m+2}, 
\end{equation}
% $\Im(\eta) >0$, 
then 
	\begin{equation}\nonumber
\begin{aligned}
&D_m=\bigg|\begin{array}{cc}
 \bsi k_p^m   & - k_s^m\\
 \left(\bsi m \mu - 2 \lambda \eta + (m-2)\mu \eta \right) k_p^{m+2} &   m \mu (\bsi \eta - 1) k_s^{m+2}
\end{array}
\bigg|\neq 0. 
% \\
%&= k_p^m k_s^m [(- m \mu k_s^2 - 2 \lambda k_p^2 + (m-2)\mu k_p^2)\eta +\bsi m \mu (k_p^2 - k_s^2)]\neq 0,
\end{aligned}
\end{equation}
\end{lem}
\begin{proof}
	By direct calculations, it yields that
	\begin{equation}\notag
	\begin{split}
		D_m &=k_p^m k_s^m [(- m \mu k_s^2 - 2 \lambda k_p^2 + (m-2)\mu k_p^2)\eta +\bsi m \mu (k_p^2 - k_s^2)]\\
			&=	- \frac{k_p^m k_s^m(\lambda+\mu ) \kappa }{\lambda+2\mu }\left[(m+2) \eta +\mathrm i m \right],
	\end{split}
	\end{equation}
	where we use \eqref{eq:kpks}. Since $\eta$ satisfies \eqref{eq:cond eta m}, it is easy to see that $D_m\neq 0.$
	\end{proof}

\begin{thm}\label{thm:GI&GI}
Let $\mathbf{u}\in L^2(\Omega)^2$ be a solution to \eqref{eq:lame}. Suppose that  there exist two intersecting lines $\Gamma_h^\pm$ of $\bmf{u}$  such that $\Gamma_h^-$ is a generalized-impedance line and $\Gamma_h^+$ is a generalized-impedance line with the intersecting angle 
\begin{equation}\label{eq:cond Gi1}
	\angle(\Gamma_h^+,\Gamma_h^-)=\varphi_0\neq \pi .
\end{equation}
The associated generalized-impedance parameter $\boldsymbol{ \eta}_1$ to $\Gamma_h^+ $ and $\boldsymbol{ \eta}_2$ to $\Gamma_h^- $ have the expansions \eqref{eq:eta1 ex} and \eqref{eq:eta2 ex} respectively, where $\eta_\ell$ ($\ell=1,2$) are the constant parts of $\boldsymbol{\eta}_\ell$ $\eta_1=\eta_2=\eta$ satisfying 
\begin{equation}\label{eq:cond eta}
\eta_1=\eta_2=\eta \mbox{ and } \eta \neq \pm \mathrm i, \, -\frac{\mathrm i m}{m+2}	. 
\end{equation}
%and 
 %where the associated impedance parameter $\eta$ to $\Gamma_h^\pm$ satisfies  $\eta\neq \pm \bsi,-\frac{\bsi m}{m+2}$, 
 Then $\bmf{u}\equiv \bmf{0}$.
\end{thm}

\begin{proof}
Since $\Gamma_h^\pm$ are generalized-impedance lines and \eqref{eq:cond eta},   multiplying $\mathrm{e}^{2 \bsi \varphi_0} $ on both sides  of  the second equation of \eqref{eq:1b1} in Lemma \ref{lem:48}, then  subtracting  the second equation of \eqref{eq:1b} in Lemma \ref{lem:g imp}, we have
\begin{equation}\label{eq:g1}
2 \eta (\lambda+\mu)(\mathrm{e}^{2 \bsi \varphi_0}-1)k_p^2 a_0=0,
\end{equation}
which can be used to obtain that $a_0=0$ by using \eqref{eq:convex}, \eqref{eq:cond Gi1} and \eqref{eq:cond eta}.

 Multiplying $\mathrm{e}^{2 \bsi \varphi_0} $ on both sides  of  the first equation of \eqref{eq:g3} in Lemma \ref{lem:48}, then  subtracting the the first equation of \eqref{eq:g2} in Lemma \ref{lem:g imp}, we have
\begin{equation}\nonumber
2 (\mathrm{e}^{2 \bsi \varphi_0}-1)k_s^2 b_0=0,
\end{equation}
which can be used to obtain that $b_0=0$ under the condition \eqref{eq:cond Gi1}. Substituting $b_0=0$ into  \eqref{eq:g3}, using \eqref{eq:cond eta}, one has
\begin{equation}\label{eq:g9}
 \bsi k_p^2 a_2 - k_s^2 b_2 =0.
\end{equation}

Similarly, multiplying  $\mathrm{e}^{2 \bsi \varphi_0} $ on both sides of the second equation of \eqref{eq:g3}, and then subtracting the second equation of \eqref{eq:g2}, one can obtain that
\begin{equation}\label{eq:g5}
[(\bsi \mu - 2 \lambda \eta - \mu \eta) k_p^3 a_1 + \mu (\bsi \eta - 1) k_s^3 b_1](\mathrm{e}^{2 \bsi \varphi_0}-1)=0,
\end{equation}
Combining the first equation of \eqref{eq:1b1} with \eqref{eq:g5}, we have 
\begin{equation}\nonumber
\bigg\{
\begin{array}{l}
\bsi k_p a_1 - k_s b_1=0,\\
(\bsi \mu - 2 \lambda \eta - \mu \eta) k_p^3 a_1 + \mu (\bsi \eta - 1) k_s^3 b_1=0,
\end{array}
\end{equation}
which can be used to prove that $a_1=b_1=0$ by noting
\begin{equation}\nonumber
\begin{aligned}
&D_1=\bigg|\begin{array}{cc}
 \bsi k_p   & - k_s\\
 (\bsi \mu - 2 \lambda \eta - \mu \eta) k_p^3 &  \mu (\bsi \eta - 1) k_s^3
\end{array}
\bigg| \neq 0, %\\
%&= k_p k_s [(- \mu k_s^2 - 2 \lambda k_p^2-\mu k_p^2)\eta +\bsi \mu (k_p^2 - k_s^2)]\neq 0,
\end{aligned}
\end{equation}
where we use Lemma \ref{lem:49} for the case $m=1$ under the condition \eqref{eq:cond eta}.  %$\eta \neq -\frac{\bsi}{3}$.

Using Lemma \ref{lem:co exp} and \eqref{eq:Jmfact}, comparing the coefficient of the term $r^2$ in both sides of \eqref{eq:GI}, and by virtue of \eqref{eq:eta1 ex}, one can show that
\begin{equation}\nonumber
\begin{aligned}
& \left[ \mathrm{e}^{2 \bsi \varphi_0} (1+\bsi \eta) (\bsi k_p^3 a_3 - k_s^3 b_3)\right] \hat{\mathbf{e}}_1 + \\
& \left[\left(2(\bsi \mu -  \lambda \eta ) k_p^4 a_2 + 2 \mu (\bsi \eta - 1) k_s^4 b_2\right) \mathrm{e}^{ \bsi \varphi_0}  - \mu \mathrm{e}^{3 \bsi \varphi_0} ( 1 - \bsi \eta) (\bsi k_p^4 a^4 - k_s^4 b_4)\right] \hat{\mathbf{e}}_2 \\
& = \bmf{0}. 
\end{aligned}
\end{equation}
Since $\hat{\mathbf{e}}_1$ and $\hat{\mathbf{e}}_2$ are linearly independent, we obtain that
\begin{equation}\label{eq:g6}
\bigg\{
    \begin{array}{l}
   \mathrm{e}^{2 \bsi \varphi_0} (1+\bsi \eta) (\bsi k_p^3 a_3 - k_s^3 b_3) =0,\\
 2(\bsi \mu -  \lambda \eta ) k_p^4 a_2 + 2 \mu (\bsi \eta - 1) k_s^4 b_2  - \mu \mathrm{e}^{2 \bsi \varphi_0} ( 1 - \bsi \eta) (\bsi k_p^4 a^4 - k_s^4 b_4) =0.
 \end{array}
\end{equation}
Similarly, comparing the coefficients of the terms $r^2$ in both sides of \eqref{eq:GI2} and by virtue of \eqref{eq:eta2 ex},   we can deduce that
\begin{equation}\label{eq:g7}
\bigg\{
    \begin{array}{l}
   (1+\bsi \eta) (\bsi k_p^3 a_3 - k_s^3 b_3) =0,\\
 2(\bsi \mu -  \lambda \eta ) k_p^4 a_2 + 2 \mu (\bsi \eta - 1) k_s^4 b_2  - \mu  ( 1 - \bsi \eta) (\bsi k_p^4 a^4 - k_s^4 b_4) =0.
 \end{array}
\end{equation}
 Multiplying $\mathrm{e}^{2 \bsi \varphi_0} $ on both sides  of  the second equation of \eqref{eq:g7}, then subtracting  the the second equation of \eqref{eq:g6}, we have
\begin{equation}\label{eq:g8}
[2(\bsi \mu -  \lambda \eta ) k_p^4 a_2 + 2 \mu (\bsi \eta - 1) k_s^4 b_2  ](\mathrm{e}^{2 \bsi \varphi_0}-1)=0. 
\end{equation}
Combining \eqref{eq:g9} with \eqref{eq:g8}, we have
\begin{equation}\nonumber
\bsi k_p^2 a_2 - k_s^2 b_2=0,\quad
2(\bsi \mu -  \lambda \eta ) k_p^4 a_2 + 2 \mu (\bsi \eta - 1) k_s^4 b_2=0,
\end{equation}
which can be used to prove that $a_2=b_2=0$ by noting
\begin{equation}\nonumber
\begin{aligned}
&D_2=\bigg|\begin{array}{cc}
 \bsi k_p^2   & - k_s^2\\
 (\bsi \mu - \lambda \eta ) k_p^4 &   \mu (\bsi \eta - 1) k_s^4
\end{array}
\bigg|\neq 0, % \\
%&= k_p^2 k_s^2 [(- \mu k_s^2 -  \lambda k_p^2)\eta +\bsi \mu (k_p^2 - k_s^2)]\neq 0,
\end{aligned}
\end{equation}
where we use Lemma \ref{lem:49} for the case $m=2$ under the condition \eqref{eq:cond eta}.    % $\eta \neq -\frac{\bsi}{2}$.

By mathematical induction, we first assume that $a_{\ell}=b_{\ell}=0$, where $\ell=0,\ldots,m-1$ and $m \in \mathbb{ N}$ with $m\geq 3$.  Using  Lemma \ref{lem:co exp} and \eqref{eq:Jmfact}, comparing the coefficients of the terms $r^{m-1}$ in both sides of \eqref{eq:GI} and by virtue of \eqref{eq:eta1 ex},  one can deduce that
\begin{equation}\label{eq:g13}
\bigg\{
\begin{array}{l}
(1+ \bsi \eta) (\bsi k_p^m a_m - k_s^m b_m)=0,\\
\mu (1- \bsi \eta) (\bsi k_p^{m+1} a_{m+1} - k_s^{m+1} b_{m+1})=0.
\end{array}
\end{equation}

Using  Lemma \ref{lem:co exp} and \eqref{eq:Jmfact}, comparing the coefficients of the terms $r^m$ in both sides of \eqref{eq:GI} and by virtue of \eqref{eq:eta1 ex}, one can show that
\begin{equation}\nonumber
\begin{aligned}
& \big[ \mathrm{e}^{2 \bsi \varphi_0} (1+\bsi \eta) (\bsi k_p^{m+1} a_{m+1} - k_s^{m+1} b_{m+1})\big] \hat{\mathbf{e}}_1 + \\
& \big[\left((\bsi m \mu - 2 \lambda \eta + (m-2)\mu \eta) k_p^{m+2} a_m + m \mu (\bsi \eta - 1) k_s^{m+2} b_m\right) \mathrm{e}^{ \bsi \varphi_0}\\
& - \mu \mathrm{e}^{3 \bsi \varphi_0} ( 1 - \bsi \eta) (\bsi k_p^{m+2} a^{m+2} - k_s^{m+2} b_{m+2})\big] \hat{\mathbf{e}}_2 \\
& = \bmf{0}.
\end{aligned}
\end{equation}
Since $\hat{\mathbf{e}}_1$ and $\hat{\mathbf{e}}_2$ are linearly independent, we obtain that
\begin{equation}\label{eq:g10}
\left\{
    \begin{array}{l}
    \begin{aligned}
 & \mathrm{e}^{2 \bsi \varphi_0} (1+\bsi \eta) (\bsi k_p^{m+1} a_{m+1} - k_s^{m+1} b_{m+1}) =0,\\
 & (\bsi m \mu - 2 \lambda \eta + (m-2)\mu \eta) k_p^{m+2} a_m + m \mu (\bsi \eta - 1) k_s^{m+2} b_m\\
 & - \mu \mathrm{e}^{2 \bsi \varphi_0} ( 1 - \bsi \eta) (\bsi k_p^{m+2} a^{m+2} - k_s^{m+2} b_{m+2}) =0.
 \end{aligned}
 \end{array}\right. 
\end{equation}
Similarly, by comparing the coefficients of the terms $r^m$ in both sides of \eqref{eq:GI2} together with the use of \eqref{eq:eta2 ex}, we can deduce that
\begin{equation}\label{eq:g11}
\left\{
     \begin{array}{l}
    \begin{aligned}
 & (1+\bsi \eta) (\bsi k_p^{m+1} a_{m+1} - k_s^{m+1} b_{m+1}) =0,\\
 & (\bsi m \mu - 2 \lambda \eta + (m-2)\mu \eta) k_p^{m+2} a_m + m \mu (\bsi \eta - 1) k_s^{m+2} b_m\\
 & - \mu  ( 1 - \bsi \eta) (\bsi k_p^{m+2} a^{m+2} - k_s^{m+2} b_{m+2}) =0.
 \end{aligned}
 \end{array}\right. 
\end{equation}
Multiplying $\mathrm{e}^{2 \bsi \varphi_0} $ on both sides  of  the second equation of \eqref{eq:g11}, then subtracting  the the second equation of \eqref{eq:g10}, we have
\begin{equation}\label{eq:g12}
[(\bsi m \mu - 2 \lambda \eta + (m-2)\mu \eta) k_p^{m+2} a_m + m \mu (\bsi \eta - 1) k_s^{m+2} b_m  ](\mathrm{e}^{2 \bsi \varphi_0}-1)=0. 
\end{equation}

Combining the first equation of  \eqref{eq:g13} with \eqref{eq:g12}, we have
\begin{equation}\nonumber
\bigg\{
\begin{array}{l}
\bsi k_p^m a_m - k_s^m b_m=0,\\
(\bsi m \mu - 2 \lambda \eta + (m-2)\mu \eta) k_p^{m+2} a_m + m \mu (\bsi \eta - 1) k_s^{m+2} b_m=0,
\end{array}
\end{equation}
which can be used to prove that $a_m=b_m=0$ by noting
\begin{equation}\nonumber
\begin{aligned}
&D_m=\bigg|\begin{array}{cc}
 \bsi k_p^m   & - k_s^m\\
 \left(\bsi m \mu - 2 \lambda \eta + (m-2)\mu \eta \right) k_p^{m+2} &   m \mu (\bsi \eta - 1) k_s^{m+2}
\end{array}
\bigg| \neq 0, %\\
%&= k_p^m k_s^m [(- m \mu k_s^2 - 2 \lambda k_p^2 + (m-2)\mu k_p^2)\eta +\bsi m \mu (k_p^2 - k_s^2)]\neq 0,
\end{aligned}
\end{equation}
where we use Lemma \ref{lem:49}  under the condition \eqref{eq:cond eta}. %where $\eta \neq -\frac{\bsi m}{m+2}$.

According  to Proposition \ref{prop:1}, the proof is complete.
\end{proof}

\begin{rem}
	The assumption \eqref{eq:cond eta} in Theorem \ref{thm:GI&GI} can be replaced  by
	\begin{equation}\label{eq:cond eta new}
\eta_1=\eta_2=\eta \mbox{ and } \Im(\eta) \in \mathbb R_+ \backslash\{1\},
\end{equation}
which shall be adopted  in studying the inverse elastic obstacle and diffractive  grating problems. 
	
	\end{rem}
	
	The following theorem deals with the case that a generalized-impedance line intersects a simply-supported line.  

\begin{thm}\label{thm:GI&forth exp}
Let $\mathbf{u}\in L^2(\Omega)^2$ be a solution to \eqref{eq:lame}. %Recall that $\mathbf{u}$ has the Fourier expansion \eqref{eq:u} around the origin.
If there exist two intersecting lines $\Gamma_h^\pm$ of $\bmf{u}$  such that   $\Gamma_h^-$ is a generalized-impedance line and $\Gamma_h^+$ is a  simply-supported line, where  the intersecting angle $\angle(\Gamma_h^+,\Gamma_h^-)=\varphi_0\neq \pi $, then $\bmf{u}\equiv \bmf{0}$.
\end{thm}
\begin{proof}
Since $\Gamma_h^+$ is a  simply-supported line, from Lemma \ref{lem:forth}, we have $a_0=0$. Substituting $a_0=0$ into the first equation of \eqref{eq:16}, one has
\begin{equation}\label{eq:a2 b2 sim}
	k_p^2 a_2+\mathrm i k_s^2 b_2=0.
\end{equation}
Since $\Gamma_h^-$ is a generalized-impedance line, substituting \eqref{eq:a2 b2 sim} into the first equation of \eqref{eq:g3} in Lemma \ref{lem:48}, it is easy to see that $b_0=0$. Adopting the same argument in the proof of Theorem \ref{thm:fourth}, we can show that $a_\ell=b_\ell=0$ for $\forall \ell \in \mathbb N$.  Finally, the proof can be concluded by using Proposition \ref{prop:1}.
\end{proof}

\section{Unique identifiability for inverse elastic obstacle problems}\label{sect:5}

In this section, the theoretical findings in the previous sections can be used to study the unique identifiability for the inverse elastic  obstacle and diffractive grating  problems. The inverse problems are concerned with recovering the geometrical shapes/profiles of certain unknown objects by using the elastic wave probing data. The inverse elastic obstacle and diffractive grating problems arise from industrial applications of practical importance. 

%We next introduce the mathematical setup of the inverse obstacle problem that expands the abstract formation \eqref{eq:ipa1}.

\subsection{Unique recovery for the inverse elastic  obstacle problem}\label{subsec:obs}

Consider a  bounded Lipschitz domain $\Omega\subset\mathbb{R}^2$  such that $\mathbb{R}^2\backslash\bar{\Omega}$ is connected. In what follows, we let $\mathbf{d}\in\mathbb{S}^1$ denote the incident direction, $\bmf{d}^{\perp}\in \mathbb S^1 $ be orthogonal to $\bmf{d}$, $k_p$ and $k_s$ be compressional and shear wave numbers  defined in \eqref{eq:kpks}. Consider an incident elastic wave field $\bmf{u}^i$, which is a time-harmonic elastic plane wave of the form
\begin{equation}\label{eq:ui}
\bmf{u}^i:=\bmf{u}^i(\mathbf{ x};k_p,k_s, \mathbf{d})=\alpha_{p} \bmf{d} \mathrm{e}^{\mathrm{i} k_{p} \bmf{x} \cdot \bmf{d} }+\alpha_{s} \bmf{d}^{\perp} \mathrm{e}^{\mathrm{i} k_s  \bmf{x} \cdot \bmf{d} }, \quad \alpha_{p}, \alpha_{s} \in \mathbb{C}, \quad\left|\alpha_{p}\right|+\left|\alpha_{s}\right| \neq 0.
\end{equation}
Physically speaking, $\bmf{u}^i$ is the detecting wave field and $\Omega$ denotes an impenetrable obstacle which interrupts the propagation of the incident wave and generates the corresponding scattered wave field $\bmf{u}^{\mathrm{sc}}$. Using the Helmholtz decomposition, we can decompose the scattered field $\bmf{u}^{\mathrm{sc} }$ in ${\mathbb R}^2 \backslash \Omega$  into
the sum of the compressional  part $\bmf{u}^{\mathrm{sc}}_p$ and the shear part $\bmf{u}^{\mathrm{sc}}_s$  as follows
\begin{equation}
	\bmf{u}^{\mathrm{s c} }=\bmf{u}_{p}^{\mathrm {s c} }+\bmf{u}_{s}^{\mathrm{s c} }, \quad \bmf{u}_{p}^{\mathrm {s c} }=-\frac{1}{k_{p}^{2}} \nabla\left( \nabla \cdot \bmf{ u}^{\mathrm {s c} }\right ), \quad\bmf{ u}_{s}^{\mathrm{s c}}=\frac{1}{k_{s}^{2}} \bf{curl} \operatorname{curl} u^{\mathrm {s c} },
\end{equation}
where
\[
 {\rm curl}\bmf{ u}=\partial_1 u_2-\partial_2 u_1, \quad {\bf
curl}{u}=(\partial_2 u, -\partial_1 u)^\top.
\]
Let $\omega=\sqrt{ \kappa}$ be the angular frequency, where $\kappa$ is the Lam\'e eigenvalue of \eqref{eq:lame}.  Define  $\bmf{u}:=\bmf{u}^i+\bmf{u}^{\mathrm{ sc} }$ to be the total wave field, then the forward scattering problem of this process can be described by the following system,
\begin{equation}\label{forward}
\begin{cases}
& {\mathcal L} \bmf{ u} + \omega^2  \bmf{u} = 0\qquad\quad \mbox{in }\ \ \mathbb{R}^2\backslash\overline{\Omega},\medskip\\
& \bmf{u} =\bmf{u}^i+\bmf{u}^{\mathrm{sc} }\hspace*{1.56cm}\mbox{in }\ \ \mathbb{R}^2,\medskip\\
& \mathscr{B}(\bmf{u})=\bmf{0}\hspace*{1.95cm}\mbox{on}\ \ \partial\Omega,\medskip\\
&\displaystyle{ \lim_{r\rightarrow\infty}r^{\frac{1}{2}}\left(\frac{\partial \bmf{u}_\beta^{\mathrm{sc} }}{\partial r}-\mathrm{i}k_\beta \bmf{u}_\beta^{\mathrm{sc} }\right) =\,0,} \quad \beta=p,s,
\end{cases}
\end{equation}
where the last equation is the Kupradze radiation condition that holds uniformly in $\hat{\mathbf{ x}}:=\mathbf{ x}/|\mathbf{ x}|\in\mathbb{S}^1$. The boundary condition $\mathscr{B}(\mathbf u)$ on $\partial \Omega$ could be either of the following six conditions:
\begin{enumerate}
	\item the first kind (Dirichlet) condition ($\Omega$ is a rigid obstacle): \begin{equation}\label{eq:B1u}
		\mathscr{B}(\bmf{u}):=\mathscr{B}_2(\bmf{u})=\bmf{u};
	\end{equation}
	\item the second kind (Neumann) condition ($\Omega$ is a traction-free obstacle):
	\begin{equation}\label{eq:B2u}
	\mathscr{B}(\bmf{u}):=\mathscr{B}_1(\bmf{u})=T_\nu \bmf{u};
	\end{equation}

	\item the impedance condition ($\Omega$ is an impedance obstacle): $\mathscr{B}(\bmf{u})=T_\nu \bmf{u}+\boldsymbol{ \eta} \bmf{u},\ \boldsymbol \eta \in \mathcal{A},  \ \Re(\boldsymbol \eta)\geq 0 \mbox{ and } \Im(\boldsymbol \eta)>0$;
%	\begin{equation}\notag \label{eq:bound imp}
%		 \mathscr{B}(u):=T_\nu \bmf{u}+\eta \bmf{u},\quad \Im(\eta)<0
%	\end{equation}
\item the third kind boundary condition ($\Omega$ is an soft-clamped  obstacle):
\begin{equation}\label{eq:B3u}
	\mathscr{B}(\bmf{u}):=\mathscr{B}_3(\bmf{u})=\left(
\begin{array}{c}
{\nu} \cdot  \mathbf{u} \\
\boldsymbol{\tau} \cdot T_{\mathbf{\nu}} \bmf{u} \\
\end{array}
\right);
\end{equation}
\item the forth kind boundary condition ($\Omega$ is a simply-supported obstacle):
\begin{equation}\label{eq:B4u}
\mathscr{B}(\bmf{u}):=\mathscr{B}_4(\bmf{u})=\left(
\begin{array}{c}
\boldsymbol{\tau} \cdot \bmf{u} \\
{\nu} \cdot T_{\mathbf{\nu}} \mathbf{u} \\
\end{array}
\right);
\end{equation}
\item the  generalized-impedance  kind boundary condition ($\Omega$ is a generalized  impedance obstacle):
\begin{equation}\label{eq:B5u}
\mathscr{B}(\bmf{u})%:=\mathscr{B}_{5,\eta}(\bmf{u})
=\left(
\begin{array}{c}
{\nu} \cdot  \mathbf{u} \\
\boldsymbol{\tau} \cdot T_{\mathbf{\nu}} \bmf{u} \\
\end{array}
\right)+ \boldsymbol \eta \left(
\begin{array}{c}
\boldsymbol{\tau} \cdot \bmf{u} \\
{\nu} \cdot T_{\mathbf{\nu}} \mathbf{u} \\
\end{array}
\right),\  \boldsymbol \eta \in \mathcal{A}, \, \Re( \boldsymbol\eta)\geq 0 \mbox{ and } \Im( \boldsymbol\eta)>0,
\end{equation}
	\end{enumerate}
where $\nu$ denotes the exterior unit normal vector to $\partial\Omega$,  $\boldsymbol{\tau}= \nu^\perp$ and the boundary  traction operator $T_\nu$ is defined in \eqref{eq:Tu}.

%Moreover, in the impedance condition given above, $ \boldsymbol \eta\in L^\infty(\partial\Omega)$, and this is different from our study in the previous sections, where the impedance $\eta$ is always required to be a constant. We would also like to point out that the conditions $\Re(\boldsymbol \eta)\geq 0 \mbox{ and } \Im(\boldsymbol \eta)>0$ are the physical requirement.

%To unify the notations, we represent the three types (1), (2) and (5) of boundary conditions with
%\begin{equation}\label{bound}
%\mathscr{B}(\bmf{u}):=T_\nu \bmf{u}+\eta\bmf{ u}=\bmf{0}  \quad\mbox{on } \partial\Omega,
%\end{equation}
%where $\eta=\infty$ and $\eta=0$ respectively stand for the Dirichlet and Neumann boundary conditions.

In this paper, we always assume the well posedness of the elastic system \eqref{forward}  associated with either of the six kinds
of boundary conditions is fullfilled, which means that there exits a unique
solution $ \bmf{u} \in H^1_{\mathrm{loc} } ({\mathbb R}^2 \backslash \Omega )$ to \eqref{forward}. We refer to \cite{ElschnerYama2010,Lai,Kupradze} for the related results when the boundary condition in \eqref{forward} is the third or forth kind boundary, while the corresponding  well posedness result for the first kind boundary can be found in \cite{HannerHsiao}. It is known that the compressional and shear parts $\bmf{u}_\beta^{\mathrm {sc} }$ ($\beta=p,s$)
of a radiating solution $\bmf{u}^{\mathrm{sc} }$ to the elastic system  \eqref{forward} possess the following asymptotic expansions
\begin{equation}
	\begin{aligned}
	\bmf{u}_{p}^{\mathrm{sc}}(\bmf{x}; k_p,k_s, \mathbf{d}) &=\frac{\mathrm{e}^{\mathrm{i} k_{p} r}}{\sqrt{r}}\left\{u_{p}^{\infty}(\hat{\bmf{x}}; \bmf{d} ) \hat{\bmf{x}}+\Oh\left(\frac{1}{r}\right)\right\} \\
	\bmf{u}_{s}^{\mathrm{sc}}(\bmf{x};k_p,k_s, \mathbf{d}) &=\frac{\mathrm{e}^{\mathrm{i} k_{s} r}}{\sqrt{r}}\left\{u_{s}^{\infty}(\hat{\bmf{x}}; \bmf{d}) \hat{\bmf{x}}^{\perp}+\Oh\left(\frac{1}{r}\right)\right\}
	\end{aligned}
\end{equation}
as $r =|\bmf{x} | \rightarrow \infty$, where $u_{p}^{\infty}$ and $u_{s}^{\infty}$ are both scalar functions defined on $\mathbb S^1$. Hence, a Kupradze radiating solution has the asymptotic behavior
$$
\bmf{u}^{\mathrm{sc}}(\bmf{x}; k_p,k_s, \mathbf{d})=\frac{\mathrm{e}^{\mathrm{i} k_{p} r}}{\sqrt{r}} u_{p}^{\infty}(\hat{\bmf{x} }; \bmf{d} ) \hat{\bmf{x}}+\frac{\mathrm{e}^{\mathrm{i} k_{s} r}}{\sqrt{r}} u_{s}^{\infty}(\hat{\bmf{x} }; \bmf{d} ) \hat{\bmf{x} }^{\perp}+\Oh\left(\frac{1}{r^{3 / 2}}\right) \quad \text { as } \quad r \rightarrow \infty
$$
The far-field pattern $\bmf{u}^\infty$ of $\bmf{u}^{\mathrm{sc}} $ is defined as
$$
\bmf{u}_t^{\infty}(\hat{\bmf{x}}; \bmf{d} ) :=u_{p}^{\infty}(\hat{\bmf{x}}; \bmf{d} ) \hat{\bmf{x}}+u_{s}^{\infty}(\hat{\bmf{x}}; \bmf{d} ) \hat{\bmf{x}}^{\perp}.
$$
Obviously, the compressional and shear parts of the far-field are uniquely determined by $\bmf{u}^{\infty}$
as follows:
$$
\bmf{u}_{p}^{\infty}(\hat{\mathbf{x}}; \bmf{d} )=\bmf{u}^{\infty}(\hat{\bmf{x}};\bmf{d} ) \cdot \hat{\bmf{x}};  \quad \bmf{u}_{s}^{\infty}(\hat{\bmf{x}},\bmf{d} )=\bmf{u}^{\infty}(\hat{\bmf{x}}; \bmf{d} ) \cdot \hat{\bmf{x}}^{\perp}.
$$
% If $\mathscr{B}(u):=u$, the boundary condition is of Dirichlet type and $\Omega$ is said to be a sound-soft obstacle. If $\mathscr{B}(u):=\partial_\nu u$, the boundary condition is of Neumann type and $\Omega$ is said to be a sound-hard obstacle. If $\mathscr{B}(u):=\partial_\nu u+\eta u$, $\Omega$ becomes an impedance obstacle with Robin type boundary condition where $\nu$ denotes the exterior unit normal vector to $\partial\Omega$ and $\eta\in L^\infty(\partial\Omega)$ signifies the corresponding impedance boundary parameter. For unification of the notation, we represent all these three types of boundary conditions with
%\begin{equation}\label{bound}
%\mathscr{B}(u):=\partial_\nu u+\eta u=0 \quad\mbox{on } \partial\Omega,
%\end{equation}
%where $\eta=\infty$ and $\eta=0$ respectively stands for the Dirichlet and Neumann boundary condition.

The inverse elastic scattering problem corresponding to \eqref{forward} concerns the determination of the scatterer $\Omega$ (and $\boldsymbol \eta$ as well in the impedance case) by knowledge of the far-field pattern $\bmf{u}_\beta^\infty(\hat{\mathbf{ x}},\mathbf{d},k)$, where $\beta=t,p$ or $s$.  We introduce the operator $\mathcal{F}$ which sends the obstacle to the corresponding far-field pattern and is defined by the forward scattering system \eqref{forward}, the aforementioned inverse problem can be formulated as
\begin{equation}\label{inverse}
\mathcal{F}(\Omega, \eta)=\bmf{u}_\beta^\infty(\hat{\mathbf{ x}}; \mathbf{d}),\quad \beta=t,p, \mbox{ or } s.
\end{equation}

Next, we show that by using the generalized Holmgren's uniqueness principle, we can establish two novel unique identifiability results for \eqref{inverse} in determining an obstacle without knowing its a priori physical property as well as its possible surface impedance by at most six  far-field patterns, namely $\bmf{u}_\beta^\infty(\hat{\mathbf{x}})$ corresponding to six  distinct $\mathbf{d}$'s.

\begin{defn}\label{def:61}
Let $Q\subset\mathbb{R}^2$ be a polygon in $\mathbb{R}^2$ such that
\begin{equation}\label{eq:edge}
\partial Q=\cup_{j=1}^\ell \Gamma_j,
\end{equation}
where each $\Gamma_j$ is an edge of $\partial Q$. $Q$ is said to be a generalized elastic obstacle associated with \eqref{forward} if there exists a Lipschitz dissection of $\Gamma_j$, $1\leq j\leq \ell$,
\[
\Gamma_j=\cup_{i=1}^6 \Gamma_i^j %\cup\Gamma_2^j\cup\Gamma_3^j
\]
such that
\begin{equation}\label{eq:66}
\begin{split}
\mathscr{B}_i(\bmf{u})&=\mathbf{0} \quad \text { on } \Gamma_{i}^j, \quad i=1,2,3,4,\\
\mathscr{B}_1(\bmf{u})+\boldsymbol\eta \mathscr{B}_2(\bmf{u})&=\mathbf{0} \quad \text { on } \Gamma_{5}^j,\quad \mathscr{B}_3(\bmf{u})+\boldsymbol\eta \mathscr{B}_4(\bmf{u})=\mathbf{0} \quad \text { on } \Gamma_{6}^j,
%\mathscr{B}_1(\bmf{u})+\eta_{1,2} \mathscr{B}_2(\bmf{u})&=\mathbf{0} \quad \text { on } \Gamma_{5}^j,\quad \mathscr{B}_3(\bmf{u})+\eta_{3,4} \mathscr{B}_4(\bmf{u})=\mathbf{0} \quad \text { on } \Gamma_{6}^j,
%\\
%\mathscr{B}_1(\mathbf{ u} )&=\mathbf{0} \quad \text { on } \Gamma_{1}^j, \quad T_\nu \mathbf{u}=\mathbf{0}  \quad \text { on } \Gamma_{2}^j, \quad T_\nu \mathbf{u}+\eta_{1,j}   \mathbf{u}=\mathbf{0} \quad \text { on } \Gamma_{3}^j,\\
%\mathscr{B}_3(\mathbf{ u}) &=\mathbf{0} \quad \text { on } \Gamma_{4}^j, \quad \mathscr{B}_4(\mathbf{u})=\mathbf{0}  \quad \text { on } \Gamma_{5}^j, \quad \mathscr{B}_{5,\eta_{2,j}}(   \mathbf{u})=\mathbf{0} \quad \text { on } \Gamma_{6}^j,
\end{split}
\end{equation}
where $\boldsymbol \eta \in \mathcal{A}$ with $\Im( \boldsymbol \eta) \geq 0$, $\mathscr{B}_1(\mathbf{ u})$, $\mathscr{B}_2(\mathbf{ u})$, $\mathscr{B}_3(\mathbf{ u})$ and $\mathscr{B}_4(\mathbf{ u})$ are defined in  \eqref{eq:B1u},  \eqref{eq:B2u}, \eqref{eq:B3u} and \eqref{eq:B4u}, respectively.  %and $\mathscr{B}_{5,\eta_{2,j}}(   \mathbf{u})$  and \eqref{eq:B5u}
\end{defn}

It is emphasized that in \eqref{eq:66}, either $\Gamma_1^j, \Gamma_2^j$, $\Gamma_3^j, \Gamma_4^j$, $\Gamma_5^j$ or $\Gamma_6^j$ could be an empty set, and hence a generalized-impedance obstacle could be purely a rigid obstacle, a traction-free obstacle, an impedance obstacle, a soft-clamped obstacle, a simply-supported obstacle, a  generalized elastic obstacle  or a mixed type. Moreover, one each edge of the polygonal obstacle, the impedance parameter can take different (complex) values. In order to simply notations, for $i\in \{1,3\}$ we formally write $\mathscr{B}_{i}(\bmf{u})+\boldsymbol\eta \mathscr{B}_{i+1}(\bmf{u})$ with $\boldsymbol \eta \equiv \infty$ to signify $\mathscr{B}_{i+1}(\bmf{u})=\bmf{0}$. In doing so, \eqref{eq:66} can be unified as $\mathscr{B}_{i}(\bmf{u})+\boldsymbol\eta \mathscr{B}_{i+1}(\bmf{u})=\bmf{0}$ on $\partial\Omega$ with
\begin{equation}\label{eq:eta}
\begin{split}
\boldsymbol\eta=&0\cdot\chi_{\cup_{j=1}^\ell \Gamma_1^j}+0\cdot\chi_{\cup_{j=1}^\ell \Gamma_3^j}+\infty\cdot\chi_{\cup_{j=1}^\ell \Gamma_2^j}+\infty\cdot\chi_{\cup_{j=1}^\ell \Gamma_4^j}+\sum_{j=1}^\ell \boldsymbol \eta_j\cdot\chi_{\Gamma_5^j}\\
&+\sum_{j=1}^\ell \boldsymbol \eta_j\cdot\chi_{\Gamma_6^j}\quad (\boldsymbol \eta_j \in \mathcal A)
\end{split}
\end{equation}
and $i \in \{1,3\}$.
We write $(Q,\boldsymbol\eta)$ to denote a generalized polygonal impedance obstacle as describe above with $\boldsymbol\eta\in L^\infty(\partial Q)\cup\{\infty\}$.
In what follows, $(\Omega , \boldsymbol\eta)$ is said to be an admissible complex obstacle if
\begin{equation}\label{eq:p1}
(\Omega, \boldsymbol\eta)=\cup_{j=1}^p (\Omega_j, \boldsymbol\eta_j),
\end{equation}
where each $(\Omega_j, \boldsymbol\eta_j)$ is a generalized polygonal impedance obstacle such that $\Omega_j, j=1, 2,\ldots, p$ are pairwise disjoint and
\begin{equation}\label{eq:r2b}
\boldsymbol\eta=\sum_{j=1}^p \boldsymbol\eta_j\chi_{\partial\Omega_j},\quad \boldsymbol\eta_j\in L^\infty(\partial\Omega_j)\cup\{\infty\},
\end{equation}
the constant part $\eta_j$ of the variable function $\boldsymbol\eta_j \in \mathcal A$ is not equal to 
$$
\pm \bsi \mbox{ and } \frac{\pm \sqrt{(\lambda + 3\mu)(\lambda + \mu)} - \mu \bsi}{\lambda + 2 \mu}. 
$$

%\begin{proof}
%	Assume that there exit some complex constants $c_\ell $ such that
%	\begin{equation}
%			\sum_{\ell=1}^n c_\ell \bmf{u}(\bmf{x}; k_p,k_s, \mathbf{d}_\ell)=\bmf{0}.
%	\end{equation}
%
%\end{proof}

The following auxiliary lemma states the linear independence of the set of total wave fields $\{\bmf{u}(\bmf{x}; k_p,k_s, \mathbf{d}_\ell), \, \ell=1,\ldots, n\}$  associated with the incident wave \eqref{eq:ui}, where $\mathbf{d}_\ell$ are pairwise distinct.
\begin{lem}\cite[Lemma 5.1]{DLW}\label{lem:51}
	Let $\mathbf {\mathbf d}_{\ell}\in\mathbb{S}^1$, $\ell=1,\ldots, n$, be $n$ vectors
	which are distinct from each other. Suppose that $\Omega$ is a bounded Lipschitz domain and $\mathbb R^2\backslash \overline{\Omega } $ is connected. Let the incident elastic wave filed $\bmf{u}^i(\mathbf{ x};k_p,k_s, \mathbf{d}_\ell)$   be defined in \eqref{eq:ui}. Furthermore, suppose that the total elastic wave filed $\bmf{u}(\bmf{x}; k_p,k_s, \mathbf{d}_\ell)$ associated with $\bmf{u}^i(\mathbf{ x};k_p,k_s, \mathbf{d}_\ell)$
 satisfies \eqref{forward}. Then {the following set of functions is linearly independent:}
	$$
	\{\bmf{u}(\bmf{x}; k_p,k_s, \mathbf{d}_\ell);~\mathbf x \in D , \ \ \ell=1,2,\ldots, n \},
	$$
	where $D \subset \mathbb R^2 \backslash \overline \Omega $ is an open set.
\end{lem}

\begin{thm}\label{thm:uniqueness1}
Let $(\Omega, \boldsymbol \eta)$ and $(\widetilde\Omega, \widetilde{\boldsymbol\eta})$ be two admissible complex obstacles. Let $\omega\in\mathbb{R}_+$ be fixed and $\mathbf{d}_\ell$, $\ell=1, \ldots, 8$ be eight distinct incident directions from $\mathbb{S}^1$. Let $\bmf{u}_\beta^\infty$ and $\widetilde{\bmf{u}}^\infty_\beta$ be, respectively, the far-field patterns associated with $(\Omega, \boldsymbol\eta)$ and $(\widetilde\Omega, \widetilde{\boldsymbol\eta})$, where $\beta=t,p, \mbox{ or } s$. If
\begin{equation}\label{eq:cond1}
\bmf{u}_\beta^\infty(\hat{\mathbf{ x}}; \mathbf{d}_\ell )=\widetilde{\bmf{u}}_\beta^\infty(\hat{\mathbf{ x}}; \mathbf{d}_\ell), \ \ \hat{\mathbf x}\in\mathbb{S}^1, \ell=1, \ldots, 8,
\end{equation}
then one has that
\begin{equation}\label{eq:u1n}
\Omega =\widetilde{\Omega}\mbox{ and } \boldsymbol\eta=\widetilde {\boldsymbol\eta}.
\end{equation}
%\begin{equation}\label{eq:nn3}
%\Omega \Delta   \widetilde  \Omega:=\big(\Omega \backslash \widetilde \Omega \big)\cup \big(\widetilde \Omega \backslash \Omega \big)
%\end{equation}
\end{thm}

\begin{proof}%[Proof of Theorem~\ref{thm:uniqueness1}]
By an absurdity argument, we first prove that if \eqref{eq:cond1} holds, one must have that $\Omega=\widetilde \Omega$.  Suppose that $\Omega $ and $\widetilde{\Omega}$ are two different admissible complex obstacles such that $\Omega\neq \widetilde\Omega$ and \eqref{eq:cond1} holds. Let $\mathbf{G}$ denote the unbounded connected component of $\mathbb{R}^2\backslash\overline{(\Omega\cup\widetilde\Omega)}$. Then by a similar topological argument to that in \cite{Liu-Zou}, one can show that there exists a line segment $\Gamma_h\subset\partial\mathbf{G}\backslash\partial\Omega$ or $\Gamma_h\subset\partial\mathbf{G}\backslash\partial\widetilde\Omega$. Without loss of generality, we assume the former case.

Let $\bmf{u}$ and $\widetilde{\bmf{u}}$ respectively denote the total wave fields to \eqref{forward} associated with $(\Omega, \eta)$ and $(\widetilde\Omega, \widetilde\eta)$. By \eqref{eq:cond1} and the Rellich theorem (cf. \cite{CK}), we know that
\begin{equation}\label{eq:aa3}
\bmf{u} (\mathbf x; k_p,k_s, \mathbf{d}_\ell)=\widetilde{\bmf{u} }(\mathbf x; k_p,k_s,\mathbf{d}_\ell),\quad \mathbf{x}\in\mathbf{G},\ \ell=1, \ldots, 8.
\end{equation}
 By using \eqref{eq:aa3} as well as the generalized-impedance boundary condition on $\partial\widetilde\Omega$, we readily have
\begin{equation}\label{eq:aa4}
% \bmf{u}= \widetilde {\bmf{u}} =\bmf{0}\quad\mbox{on}\ \ \Gamma_{h,D} \quad T_\nu \bmf{u}=T_\nu \widetilde {\bmf{u}} =\bmf{0}\quad\mbox{on}\ \ \Gamma_{h,N},\quad
 \mathscr{B}_{i}(\bmf{u})+{\widetilde {\boldsymbol\eta}} \mathscr{B}_{i+1}(\bmf{u})=
 \mathscr{B}_{i}( \widetilde {\bmf{u}} )+\widetilde{\boldsymbol \eta} \mathscr{B}_{i+1}(\widetilde{ \bmf{u}} )=\bmf{0}\quad\mbox{on}\ \ \Gamma_{h} \mbox{ for } i \in \{1,3\}.
\end{equation}
%where $\Gamma_{h}=\Gamma_{h,D}\cup \Gamma_{h,N} \cup \Gamma_{h,I} $,  $\Gamma_{h,D}$, $\Gamma_{h,N}$ and $\Gamma_{\ell,I}$ are disjoint.
Consider a fixed point $\mathbf x_0 \in \Gamma_h $. There exits a sufficient small  positive number  $\varepsilon \in {\mathbb R}_+$ such that $B_{2\varepsilon}  (\mathbf x_0) \Subset \mathbf{G} $, where $B_{2\varepsilon }(\mathbf x_0)$ is a disk centered at $\mathbf x_0$ with the radius $2\varepsilon$. Without loss of generality, we may assume that $\mathbf x_0 $ is the origin.   Let $\Gamma_\varepsilon =  B_\varepsilon (\mathbf 0) \cap  \Gamma_h$, where $B_{\varepsilon }(\mathbf 0)$ is a disk centered at $\mathbf{x_0}$ with the radius $\varepsilon$. It is also noted that
$$
-{\mathcal L} \bmf{u}= \omega^2 \bmf{u} \mbox{ in } B_{2\varepsilon } (\mathbf 0).
$$
Therefore due to  \eqref{eq:u}  one knows that $\bmf{u}(\mathbf x; k_p,k_s,  \mathbf{d}_\ell)$ has the Fourier expansion  around the origin as follows
\begin{equation}\label{eq:u new}
 \begin{aligned}
 \mathbf{u}(\mathbf{x}; k_p,k_s,  \mathbf{d}_\ell)=  &\sum_{m=0}^{\infty}  \left\{ \frac{k_p}{2} a_{\ell,m} \mathrm{e}^{\bsi m \varphi} \left\{J_{m-1}\left(k_{p} r\right)\mathrm{e}^{-\bsi \varphi}\mathbf{e}_1
 -J_{m+1}\left(k_{p}r\right)\mathrm{e}^{\bsi \varphi}\mathbf{e}_2 \right\}\right.\\
 & \, + \frac{\bsi k_s}{2} b_{\ell,m} \mathrm{e}^{\bsi m \varphi} \left\{J_{m-1}\left(k_{s} r\right)\mathrm{e}^{-\bsi \varphi}\mathbf{e}_1
  +J_{m+1}\left(k_{s} r\right)\mathrm{e}^{\bsi \varphi}
  \mathbf{e}_2
   \right\} \bigg\} .
 \end{aligned}
\end{equation}

 %The Fourier coefficients $a_m$ and $b_m$ are given in \eqref{eq:u}.
Recall that the unit normal vector $\nu $ and the tangential vector $\boldsymbol{\tau}$ to $\Gamma_h$ are defined in \eqref{eq:nutau}, respectively.  Due to the linear dependence of eight $\mathbb{C}^7$-vectors, it is easy to see that there exist eight complex constants $a_\ell$  such that
%the following conditions
\begin{equation}\notag
 \sum_{\ell=1}^8 \alpha_{\ell } \begin{bmatrix}
	\bmf{u} (\mathbf 0; k_p,k_s,  \mathbf{d}_\ell)\cr   %\boldsymbol{\tau} \cdot \partial_{\nu} \bmf{u} |_{\bmf{x}={\bmf{0}} } \\
	\boldsymbol{\nu}^\top \nabla \bmf{u} \boldsymbol{\nu}|_{ \mathbf x=\bmf{0} }\\
	\boldsymbol{\tau}^\top \nabla \bmf{u} \boldsymbol{\nu} |_{ \mathbf x=\bmf{0} }\\
	b_{\ell,0}\\
	b_{\ell,1}\\
	b_{\ell,2}
\end{bmatrix}=\bmf{0},
%	 \bmf{u}(\bmf{x}_0)=\bmf{0} \mbox{ and } \boldsymbol{\tau} \cdot \partial_{\nu} \bmf{u} |_{\bmf{x}={\bmf{x}}_0 }=0
\end{equation}
where $b_{\ell,m}$ is given in  \eqref{eq:u new}, $m=0,1,2$. 
Moreover, there exits at least one $\alpha_\ell $ is not zero.
Let
\begin{equation}\label{eq:u513}
\begin{split}
 \bmf{u}(\bmf{x};k_p,k_s)&=\sum_{\ell=1}^8 \alpha_{\ell } \bmf{u} (\mathbf x; k_p,k_s, \mathbf{d}_\ell)\\
 &=\sum_{m=0}^{\infty}  \left\{ \frac{k_p}{2} a_{m} \mathrm{e}^{\bsi m \varphi} \left\{J_{m-1}\left(k_{p} r\right)\mathrm{e}^{-\bsi \varphi}\mathbf{e}_1
 -J_{m+1}\left(k_{p}r\right)\mathrm{e}^{\bsi \varphi}\mathbf{e}_2 \right\}\right.\\
 & \, + \frac{\bsi k_s}{2} b_{m} \mathrm{e}^{\bsi m \varphi} \left\{J_{m-1}\left(k_{s} r\right)\mathrm{e}^{-\bsi \varphi}\mathbf{e}_1
  +J_{m+1}\left(k_{s} r\right)\mathrm{e}^{\bsi \varphi}
  \mathbf{e}_2
   \right\} \bigg\},
 \end{split}
\end{equation}
where
$$
a_{m}=\sum_{\ell=1}^8 \alpha_\ell a_{\ell,m},\quad b_{m}=\sum_{\ell=1}^8 \alpha_\ell b_{\ell,m},\quad \forall m \in \mathbb N\cup\{0\}. 
$$
Then we have
\begin{equation}\label{eq:611}
%\sum_{\ell=1}^4 a_{\ell } \begin{bmatrix}
%	\bmf{u} (\mathbf x_0;  \mathbf{d}_\ell)\cr   \boldsymbol{\tau} \cdot \partial_{\nu} \bmf{u} |_{\bmf{x}={\bmf{x}}_0 }.
%\end{bmatrix}=\bmf{0}
 \bmf{u}(\bmf{0};k_p,k_s,\mathbf d_\ell )=\bmf{0} \mbox{ and } %\boldsymbol{\tau} \cdot \partial_{\nu} \bmf{u} |_{\bmf{x}={\bmf{0}} }=
 \boldsymbol{\nu}^\top \nabla \bmf{u} \boldsymbol{\nu}|_{ \mathbf x=\bmf{0} }=\boldsymbol{\tau}^\top \nabla \bmf{u} \boldsymbol{\nu} |_{ \mathbf x=\bmf{0} }=b_m =0,\, m=0,1,2.
\end{equation}
Next we distinguish two separate cases. The first case is that $ \bmf{u}(\bmf{x};k_p,k_s)\equiv \bmf{0},\,  \forall \bmf{x} \in \mathbf{G}$. In view of \eqref{eq:u513}, since $\alpha_\ell $ are not all zero and $\bmf{d}_\ell$ are distinct, we readily have a contradiction by Lemma \ref{lem:51}. For the second case,  we suppose that $ \bmf{u}(\bmf{x};k_p,k_s)\equiv\hspace*{-3mm}\backslash\  \bmf{0}$.
  In view of \eqref{eq:aa4} and \eqref{eq:611}, recalling Definitions \ref{def:2} and \ref{def:3}, we know that $\Gamma_\varepsilon $ is a singular  line of $\bmf{u}$, which implies that $\Gamma_\varepsilon  $ could be a singular rigid, or singular traction-free, or singular impedance, or singular soft-clamped, or singular simply-supported or singular generalized-impedance line of $\bmf{u}$ in Definition \ref{def:2} and \ref{def:3}. Therefore, by the generalized Holmgren's principle (cf. \cite[Theorems 3.1 and 3.2]{DLW},  Theorems \ref{thm:impedance line},  \ref{thm:third}, \ref{thm:fourth} and \ref{thm:g im}), we obtain that
\begin{equation}\label{eq:613 con}
	\bmf{u}\equiv \bmf{0} \mbox{ in } B_{2\varepsilon }(\bmf{0}),
\end{equation}
which is obviously a contradiction.

Next, we prove that by knowing $\Omega =\widetilde{\Omega}$, one must have that $\boldsymbol \eta \equiv  \widetilde {\boldsymbol \eta}$.  Assume contrarily that $\boldsymbol \eta \neq  \widetilde{\boldsymbol \eta}$. We consider two separated cases. The first case is that the involved boundary condition of $\mathbf u$ given by
$$
\mathscr{B}_1(\bmf{u})+\boldsymbol \eta \mathscr{B}_2(\bmf{u})=\mathscr{B}_1(\bmf{u})+\widetilde {\boldsymbol \eta} \mathscr{B}_2(\bmf{u})=\mathbf{0} \mbox{ on } \Sigma \Subset \partial \Omega=\partial \widetilde \Omega,
$$
where  $\mathscr{B}_1(\bmf{u})$ and $\mathscr{B}_2(\bmf{u})$ are defined in \eqref{eq:B1u} and \eqref{eq:B2u}, respectively.  One can easily show that there exists an open subset $\Sigma$ of $\partial \Omega=\partial \widetilde \Omega$ such that
\begin{equation}\label{eq:523}
	\mathbf{u}=T_\nu \mathbf{u}=\mathbf{0} \mbox { on } \Sigma.
\end{equation}
The second case is that the involved boundary condition of $\mathbf u$ given by
$$
\mathscr{B}_3(\bmf{u})+\boldsymbol \eta \mathscr{B}_4(\bmf{u})=\mathscr{B}_3(\bmf{u})+\widetilde {\boldsymbol \eta} \mathscr{B}_4(\bmf{u})=\mathbf{0} \mbox{ on } \Sigma \Subset \partial \Omega=\partial \widetilde \Omega,
$$
where  $\mathscr{B}_3(\bmf{u})$ and $\mathscr{B}_4(\bmf{u})$ are defined in \eqref{eq:B3u} and \eqref{eq:B4u}, respectively. Since $\boldsymbol \eta \neq \widetilde {\boldsymbol \eta}$,  one can readily have
$$
\mathscr{B}_3(\bmf{u})=\mathscr{B}_4(\bmf{u})=\mathbf 0.
$$
By \eqref{eq:B3u} and \eqref{eq:B4u}, it is obvious that
$$
\nu \cdot \mathbf u=\boldsymbol{\tau} \cdot \mathbf u=\nu \cdot T_\nu \mathbf u=\boldsymbol{\tau} \cdot T_\nu\mathbf u=0 \mbox{ on } \Sigma.
$$
Since $\nu \perp \boldsymbol{\tau}$, $\nu \in \mathbb R^2$, $\boldsymbol{\tau} \in \mathbb R^2$, by distinguishing the real and imaginary parts of $\mathbf u$ and $T_\nu \mathbf u$, one can know that \eqref{eq:523} holds.

Therefore for these two cases by the classical Holmgren's principle, we know that $\mathbf{u}\equiv \mathbf{0}$ in ${\mathbb R}^2 \backslash \Omega $, which readily yields a contradiction.

The proof is complete.
\end{proof}

\begin{rem}
	The shape of an admissible complex obstacle and its physical properties (in the case that $\boldsymbol \eta=0$ or $\boldsymbol \eta=\infty$) can be determined by at most eight far-field patterns. Furthermore, if it is of generalized-impedance type or impedance type, one can determine the surface generalized-impedance parameter as well. In Table \ref{table:results}, we summarize the unique identifiability result for a pure  polygonal rigid obstacle, a pure polygonal traction-free obstacle,  a pure polygonal impedance obstacle, a pure polygonal soft-clamped obstacle, a 	pure polygonal simply-supported obstacle and a 	pure polygonal  generalized-impedance obstacle by analyzing the assumptions' quantities in Definitions \ref{def:2} and \ref{def:3}.  
\end{rem}

\begin{table} [!htbp]
	\centering
	{	
		\begin{tabular}{|c|c|}
			\hline
		Types of the elastic obstacle	&\tabincell{c}{ The number of distinct incident directions \\ are needed for the unique identifiability}\\
			\hline
	\tabincell{c} {a	pure polygonal \\rigid obstacle} &2  \\
			\hline
			\tabincell{c}{a	pure polygonal\\ traction-free obstacle} &4  \\
			\hline
			\tabincell{c}{a	pure polygonal\\ impedance obstacle} &4  \\
			\hline
			\tabincell{c}{a	pure polygonal\\ soft-clamped obstacle} &2  \\
			\hline
			\tabincell{c}{a	pure polygonal \\
			simply-supported obstacle} &2  \\
			\hline
	\tabincell{c}{a	pure polygonal \\ generalized-impedance obstacle }&6  \\
			\hline
		\end{tabular}
	}
	\medskip
	\caption{}
	\label{table:results}
\end{table}

Finally, we show that if fewer far-field patterns are used, one can establish a local uniqueness result in determining a generic class of admissible complex obstacles. In \cite[Theorem 5.2]{DLW}, the corresponding local uniqueness result are considered by imposing a certain condition on the degree of an admissible complex obstacle; please refer to the detailed discussions in \cite[Theorem 5.2]{DLW}.

To that end, we first introduce an admissible complex obstacle with the class $\mathcal C$. Let $\Omega$ be defined in \eqref{eq:p1} that consists of finitely many pairwise disjoint polygons. Let $\Gamma, \Gamma'\subset\partial\Omega$ be two adjacent edges of $\partial\Omega$. 
Moreover, we let $\boldsymbol  \zeta$ and ${\boldsymbol \zeta}'$ respectively signify the values of $\boldsymbol \eta$ on $\Gamma$ and $\Gamma'$ around the vertex formed by those two edges. It is noted that $\boldsymbol \zeta$ and $\boldsymbol \zeta'$ may be $0, \infty$ or a variable function belonging to the class $\mathcal A$ introduced in Definition \ref{def:3}. If $\boldsymbol \zeta \in \mathcal A$, according to  Definition \ref{def:3}, $\boldsymbol \zeta$ is given by the following absolutely convergent series at the intersecting point $\Gamma \cap \Gamma'$:
\begin{equation}\label{eq:zeta1 ex}
	%\begin{align}
		\boldsymbol{\zeta}=\zeta+\sum_{j=1}^\infty \zeta_{1,j} r^j,  
		%\quad %\label{eq:eta1 ex} \\
	%	\boldsymbol{\eta}_2=\eta_{2}+\sum_{j=1}^\infty \eta_{2,j}(\theta) r^j \label{eq:eta2 ex}
	%\end{align}
\end{equation}
    where $\zeta\in\mathbb{C}\backslash\{0\}$, $\zeta_{1,j}\in \mathbb C$ and $r\in [0,h]$. Similarly, if $\boldsymbol \zeta' \in \mathcal A$, $\boldsymbol \zeta'$ is given by the following absolutely convergent series at the intersecting point $\Gamma \cap \Gamma'$:
\begin{equation}\label{eq:zeta2 ex}
	%\begin{align}
		\boldsymbol{\zeta}'=\zeta'+\sum_{j=1}^\infty \zeta_{2,j} r^j,  
		%\quad %\label{eq:eta1 ex} \\
	%	\boldsymbol{\eta}_2=\eta_{2}+\sum_{j=1}^\infty \eta_{2,j}(\theta) r^j \label{eq:eta2 ex}
	%\end{align}
\end{equation}
    where $\zeta'\in\mathbb{C}\backslash\{0\}$, $\zeta_{2,j}\in \mathbb C$ and $r\in [0,h]$.

An admissible complex obstacle $(\Omega, \eta)$ is said to belong to the class $\mathcal{C}$ if
\begin{equation}\label{eq:cond1n}
\Im(\zeta) \in \mathbb R_+\backslash\{1\}, \quad \Im(\zeta') \in \mathbb R_+\backslash\{1\},
\end{equation}
where $\zeta$ and $\zeta'$ are defined in \eqref{eq:zeta1 ex} and \eqref{eq:zeta2 ex} respectively,
and
\begin{equation}\label{eq:cond2n}
\zeta=\zeta'\quad \mbox{if both $\boldsymbol \zeta$ and $\boldsymbol \zeta'$ are variable functions given by \eqref{eq:zeta1 ex} and \eqref{eq:zeta2 ex} },
\end{equation}
for all two adjacent edges $\Gamma, \Gamma'$ of $\partial\Omega$.

%Extending $\Gamma$ and $\Gamma'$ into straight lines in the plane $\mathbb{R}^2$, we denote them by $\widehat{\Gamma}$ and $\widehat{\Gamma'}$. Clearly, the intersection of $\widehat{\Gamma}$ and $\widehat{\Gamma'}$ forms two angles, with one belonging to $(0,\pi/2]$ and the other one belonging to $[\pi/2, \pi)$. We write $\angle_{\rm acute}(\Gamma,\Gamma')$ to signify the former one. Define
%\begin{equation}\label{de:dg1}
%\mathrm{deg}(\Omega):=\max_{\Gamma, \Gamma'\in \partial\Omega}\{\angle_{\rm acute}(\Gamma, \Gamma')|\ \Gamma, \Gamma'\ \mbox{are two adjacent edges of}\ \partial\Omega\}.
%\end{equation}
%

\begin{thm}\label{thm:uniqueness2}
Let $(\Omega, \boldsymbol \eta)$ and $(\widetilde\Omega, \widetilde{\boldsymbol \eta})$ be two admissible complex obstacles from the class $\mathcal{C}$ as described above.  Let $\omega \in\mathbb{R}_+$ be fixed and $\mathbf{d}_\ell$, $\ell=1, \ldots, 5$ be five distinct incident directions from $\mathbb{S}^1$.  Let $\mathbf{G}$ denote the unbounded connected component of $\mathbb{R}^2\backslash\overline{(\Omega\cup\widetilde\Omega)}$. Let $\bmf{u}_\beta^\infty$ and $\widetilde{\bmf{u}}^\infty_\beta$ be, respectively, the far-field patterns associated with $(\Omega, \eta)$ and $(\widetilde\Omega, \widetilde\eta)$, where $\beta=t,p, \mbox{ or } s$. If
\begin{equation}\label{eq:cond1 corner}
\bmf{u}_\beta^\infty(\hat{\mathbf{ x}},\mathbf{d}_\ell )=\widetilde{\bmf{u}}_\beta^\infty(\hat{\mathbf{ x}},\mathbf{d}_\ell), \ \ \hat{\mathbf x}\in\mathbb{S}^1, \ell=1, \ldots, 5,
\end{equation}
then one has that
$$
\left(\partial \Omega \backslash \partial \overline{ \widetilde{\Omega }} \right  )\cup \left(\partial \widetilde{\Omega } \backslash \partial \overline{ \Omega } \right)
$$
%\begin{equation}\label{eq:nn3}
%\Omega \Delta   \widetilde  \Omega:=\big(\Omega \backslash \widetilde \Omega \big)\cup \big(\widetilde \Omega \backslash \Omega \big)
%\end{equation}
cannot have a corner on  $\partial \mathbf{G}$.
\end{thm}

\begin{proof}
	We prove the theorem by contradiction. Assume \eqref{eq:cond1 corner}   holds but $
\left(\partial \Omega \backslash \partial \overline{ \widetilde{\Omega }} \right  )\cup \left(\partial \widetilde{\Omega } \backslash \partial \overline{ \Omega } \right)
$ has a corner $\mathbf x_c$ on $\partial \mathbf{G}$.  Clearly, $\mathbf x_c$ is either a vertex of $\Omega$ or a vertex of $\widetilde\Omega$. Without loss of generality, we assume the latter case. Let $h\in\mathbb{R}_+$ be sufficiently small such that $B_h(\bmf{x}_c)\Subset\mathbb{R}^2\backslash\overline \Omega $. Moreover, since $\bmf{x}_c$ is a vertex of $\widetilde\Omega$, we can assume that
\begin{equation}\label{eq:aa2}
B_h(\mathbf x_c)\cap \partial\widetilde\Omega=\Gamma_h^\pm,
\end{equation}
where $\Gamma_h^\pm$ are the two line segments lying on the two edges of $\widetilde\Omega$ that intersect at $\mathbf x_c$.  Furthermore, on $\Gamma_h^\pm$ the boundary conditions are given by \eqref{eq:66}.

By \eqref{eq:cond1 corner} and the Rellich theorem (cf. \cite{CK}), we know that
\begin{equation}\label{eq:aa5}
\bmf{u} (\mathbf x; k_p,k_s, \mathbf{d}_\ell)=\widetilde{\bmf{u} }(\mathbf x; k_p,k_s, \mathbf{d}_\ell),\quad x\in\mathbf{G},\ \ell=1,\ldots, 5. 
\end{equation}
It is clear that $\Gamma_h^\pm\subset\partial\mathbf{G}$. Hence, by using \eqref{eq:aa3} as well as the generalized boundary condition \eqref{eq:66} on $\partial\widetilde\Omega$, we readily have
\begin{equation}\label{eq:aa4new}
\mathscr{B}_{i}(\bmf{u})+{\widetilde {\boldsymbol\eta}} \mathscr{B}_{i+1}(\bmf{u})=
 \mathscr{B}_{i}( \widetilde {\bmf{u}} )+\widetilde{\boldsymbol \eta} \mathscr{B}_{i+1}(\widetilde{ \bmf{u}} )=\bmf{0}\quad\mbox{on}\ \ \Gamma_{h}^\pm  \mbox{ for } i \in \{1,3\}.
\end{equation}
It is also noted that in $B_h(\mathbf x_c)$, $-{\mathcal L} \bmf{u}=\omega^2 \bmf{u}$.

Due to the linear dependence of five $\mathbb{C}^2$-vectors, we see that there exits five complex constants $\alpha_\ell $ such that
$$
\sum_{\ell=1}^5 \alpha_\ell \bmf{u}(\bmf{x}_c;k_p,k_s,\bmf{d}_\ell)=\bmf{0},
$$
where there exits at least one $\alpha_\ell $ is not zero. Set $\bmf{u}(\bmf{x};k_p,k_s)=\sum_{\ell=1}^5  \alpha_\ell \bmf{u}(\bmf{x};k_p,k_s,\bmf{d}_\ell)$. Then we know that
\begin{equation}\label{eq:611new}
	 \bmf{u}(\bmf{x}_c;k_p,k_s)=\bmf{0}.
\end{equation}
Similar to the proof of Theorem \ref{thm:uniqueness2}, we consider the following two cases. The first one is $ \bmf{u}(\bmf{x};k_p,k_s)\equiv \bmf{0},\, \forall \bmf{x}\in \bmf{G}$. Since there exits at one $\alpha_\ell$ such that $\alpha_\ell \neq 0$ and $\mathbf{d}_\ell$ are distinct, from Lemma \ref{lem:51}, we can arrive at a contradiction. The other case is that $ \bmf{u}(\bmf{x};k_p,k_s)\equiv\hspace*{-3mm}\backslash\  \bmf{0}$. By \eqref{eq:cond1n} and \eqref{eq:cond2n}, as well as the generalized Holmgren's principle in \cite[Theorems 4.1, 4.3]{DLW},  Theorems \ref{thm:54}--\ref{eq:impedance line exp}, Theorems \ref{thm:two traction}--\ref{thm:GI&forth exp},  one can show that
\[
\mathbf{u} \equiv \mathbf{0} \mbox{ in } \mathbf{G},
\]
which yields a contradiction again.

The proof is complete.
\end{proof}

%\begin{rem}
%Following a similar argument, one can derive more unique identifiability results similar to Theorem~\ref{thm:uniqueness2}. For example, if one excludes the presence of $T_\nu \mathbf{u}=\mathbf{0}$ on any boundary portion in \eqref{eq:66} of Definition \ref{def:61}, then the assumption \eqref{eq:cond1n} in Theorem \ref{thm:uniqueness2} can be removed. We choose not to discuss the details about those extensions in this article.
%\end{rem}	

\subsection{Unique recovery for the inverse elastic  diffraction grating problem}

In this subsection, we consider the unique recovery for the inverse elastic  diffraction grating problem. First we give a brief review of the basic mathematical model for this inverse problem.   
Let the profile of a diffraction grating be described by the curve
\begin{equation}\label{grating}
\Lambda_f=\{(x_1,x_2) \in \R^2 ;~x_2 =f(x_1)\},
\end{equation}
where $f$ is a periodic Lipschitz function with period $2\pi$. Let
$$
\Omega_f=\{ \mathbf  x\in \R^2; x_2 > f(x_1), x_1 \in \R\}
$$
be filled with an elastic material whose mass density is equal to one. Suppose further that the incident wave is given by the  pressure wave
\begin{equation} \label{eq:p int}
\mathbf u^i(\mathbf x;k_p, \mathbf d)=\mathbf d {\mathrm e}^{\bsi k_p \mathbf d \cdot \mathbf x},\quad \mathbf d=(\sin \theta, -\cos \theta)^\top ,\quad \theta \in \left( - \frac{\pi}{2} , \frac{\pi}{2} \right),
\end{equation}
or the shear wave
\begin{equation}\label{eq:s int}
	\mathbf u^i(\mathbf x;k_s, \mathbf d)=\mathbf d^\perp {\mathrm e}^{\bsi k_s \mathbf d \cdot \mathbf x},\quad \mathbf d^\perp=( \cos \theta,\sin \theta)^\top ,\quad \theta \in \left( - \frac{\pi}{2} , \frac{\pi}{2} \right),
\end{equation}
propagating to $\Lambda_f$ from the top, where  $k_p$ and $k_s$ are defined in \eqref{eq:kpks}.  Then the total wave field satisfies the following Navier system:
\begin{equation}\label{eq:76}
\mathcal L  \mathbf u+\omega ^2 \mathbf  u=\mathbf 0\hspace*{0,5cm} \mbox{in}\ \ \Omega_f; \quad  {\mathscr B}(\mathbf u)\big|_{\Lambda _f}=\mathbf 0 \hspace*{0,5cm}  \mbox{on} \ \ \Lambda_f, 
\end{equation}
with the generalized-impedance boundary condition
\begin{equation}\label{eq:gbc}
\mathscr{B}(\mathbf u)=\mathscr{B}_{i}(\bmf{u})+\boldsymbol\eta \mathscr{B}_{i+1}(\bmf{u})=\bmf{0}\quad\mbox{on}\ \ \Lambda_f \mbox{ for } i \in \{1,3\}
\end{equation}
where $\boldsymbol  \eta$ can be $\infty$, $0$, or $\boldsymbol \eta \in \mathcal{A}$ with $\Im( \boldsymbol \eta) \geq 0$, and $\mathscr{B}_1(\mathbf{ u})$, $\mathscr{B}_2(\mathbf{ u})$, $\mathscr{B}_3(\mathbf{ u})$ and $\mathscr{B}_4(\mathbf{ u})$ are defined in  \eqref{eq:B1u},  \eqref{eq:B2u}, \eqref{eq:B3u} and \eqref{eq:B4u}, respectively.

In what follows, we shall mainly consider the incident pressure wave \eqref{eq:p int}. To achieve the uniqueness of \eqref{eq:76}, the total wave field $\mathbf u$ should be $\alpha$-quasiperiodic in the  $x_1$-direction, with $\alpha=k_p\sin \theta$ for the incident pressure wave \eqref{eq:p int}, 
%and $\alpha=k_s\sin \theta$ for the incident shear wave \eqref{eq:s int},   
which means that 
$$
\mathbf u(x_1+2\pi, x_2)=e^{2\bsi \alpha \pi}\cdot \mathbf u(x_1,x_2),
$$
and the scattered field $u^s$ satisfies  the Rayleigh expansion  (cf.\cite{Arens,CGK}):
\begin{align}\label{radiation}
\mathbf u^s(\mathbf x;\omega, \mathbf d )&=\sum_{n\in \mathbb Z} u_{p,n}{\mathrm e}^{\bsi { \xi}_{p,n} (\theta ) \cdot \mathbf  x } { \xi}_{p,n}(\theta )
%\begin{bmatrix}
%	\alpha_n \cr \beta_n
%\end{bmatrix} 
+\sum_{n\in \mathbb Z} u_{s,n}{\mathrm e}^{\bsi { \xi}_{s,n} (\theta ) \cdot \mathbf  x }\begin{bmatrix}
	0& 1\cr-1&0 
\end{bmatrix} { \xi}_{s,n} (\theta )
 \end{align}
for $ x_2 > \max_{x_1\in [0, 2\pi]} f(x_1)$, where $u_{p,n},\, u_{s,n}\in \C(n\in \mathbb Z)$ are called the Rayleigh coefficient of $\mathbf  u^s$, and 
\begin{equation}\label{eq:839n}
\begin{split}
\xi_{p,n}(\theta )&=\left(\alpha_{n}(\theta), \beta_{p,n}(\theta)\right)^\top, \quad \alpha_{n}(\theta )=n+ \alpha, \\
\beta_{p,n}(\theta )&=\left\{\begin{array}
{c}
\sqrt{k_p^2- \alpha_n^2 (\theta) }, \quad\mbox{ if } |\alpha_{n} (\theta )| \leq k_p\\
\\[1pt]
\bsi \sqrt{\alpha_{n}^2 (\theta)-k_p^2 }, \quad\mbox{ if } |\alpha_{n} (\theta)| > k_p
\end{array}, \right. \\
\xi_{s,n}(\theta )&=\left(\alpha_{n}(\theta),\beta_{s,n}(\theta)\right)^\top, \\ %\quad \alpha_{s,n}(\theta )=n+k_s\sin \theta, \\
 \beta_{s,n}(\theta )&=\left\{\begin{array}
{c}
\sqrt{k_s^2- \alpha_{n}^2 (\theta) }, \quad\mbox{ if } |\alpha_{n} (\theta )| \leq k_s\\
\\[1pt]
\bsi \sqrt{\alpha_{n}^2 (\theta)-k_s^2 }, \quad\mbox{ if } |\alpha_{n} (\theta)| > k_s
\end{array}. \right. 
\end{split}
\end{equation}

%\begin{equation}\label{eq:839n}
%\begin{split}
%\xi_{p,n}(\theta )&=\left(\alpha_{p,n}(\theta), \beta_{p,n}(\theta)\right)^\top, \quad \alpha_{p,n}(\theta )=n+k_p\sin \theta, \\
%\beta_{p,n}(\theta )&=\left\{\begin{array}
%{c}
%\sqrt{k_p^2- \alpha_n^2 (\theta) }, \quad\mbox{ if } |\alpha_{p,n} (\theta )| \leq k_p\\
%\\[1pt]
%\bsi \sqrt{\alpha_{p,n}^2 (\theta)-k_p^2 }, \quad\mbox{ if } |\alpha_{p,n} (\theta)| > k_p
%\end{array}, \right. \\
%\xi_{s,n}(\theta )&=\left(\alpha_{s,n}(\theta),\beta_{s,n}(\theta)\right)^\top, \quad \alpha_{s,n}(\theta )=n+k_s\sin \theta, \\
% \beta_{s,n}(\theta )&=\left\{\begin{array}
%{c}
%\sqrt{k_s^2- \alpha_{s,n}^2 (\theta) }, \quad\mbox{ if } |\alpha_{s,n} (\theta )| \leq k_s\\
%\\[1pt]
%\bsi \sqrt{\alpha_{s,n}^2 (\theta)-k_s^2 }, \quad\mbox{ if } |\alpha_{s,n} (\theta)| > k_s
%\end{array}. \right. 
%\end{split}
%\end{equation}
%

The existence and uniqueness of the $\alpha$-quasiperiodic solution to \eqref{eq:76} for the  the first kind (Dirichlet) condition boundary condition can be found in \cite{Arens}.  %It should be pointed out that the uniqueness of the direct scattering problem associated with the sound-hard condition is not always true (see \cite{kamotski}). 
In our subsequent study, we assume the well-posedness of the forward scattering problem and focus on the study of the inverse grating problem.

Introduce a measurement  boundary as
$$
\Gamma_b:=\{ (x_1,b)\in \R^2 ;~ 0 \leq x_1 \leq 2 \pi, \, b> \max_{x_1\in [0, 2\pi]}|f(x_1)|  \}.
$$
The inverse diffraction grating problem is to determine $(\Lambda_f,\eta)$ from the knowledge of $\mathbf u(\mathbf x|_{\Gamma_b};\omega ,\mathbf d)$, and 
can be formulated as the operator equation:
\begin{equation}\label{eq:iopd} %\notag
%\mathcal{F}: \, u_b(\mathbf x;k,\mathbf d) \rightarrow(\Lambda_b,\eta) , \quad  \mathbf x\in \Gamma_b. 
\mathcal{F}(\Lambda_f,\boldsymbol \eta )=\mathbf u(\mathbf x;k_p,\mathbf d), \quad  \mathbf x\in \Gamma_b,
\end{equation}
where $\mathcal{F}$ is defined by the forward diffraction scattering system, and is nonlinear.

The unique recovery result on the inverse elastic  diffraction  grating problem with the the first kind  boundary condition under a priori information about the height of the grating profile by  a finite number of incident plane waves  can be found in \cite{CGK}. 
	But the unique identifiability is still open for the impedance or generalized-impedance cases, which we shall resolve in the rest of the paper. 
	In doing so, 
	%Similar to Definition \ref{def:r1}, 
	we introduce the following admissible polygonal gratings associated with 
	the inverse elastic  diffraction grating problem.
%which is shown to be uniquely determined by at most two incident waves with one measurement on $\Gamma_b$. 
	\begin{defn}\label{def:dd1}
		Let $(\Lambda_f, \boldsymbol \eta)$ be a periodic grating as described in \eqref{grating} with the generalized-impedance boundary condition  \eqref{eq:gbc}. Suppose that $\Lambda_f$ restricted to the interval $[0,2\pi]$ consists of finitely many line segments $
		\Gamma_j$ ($j=1,\ldots,\ell $), and $\Lambda_f$ is not a straight line parallel to the $x_1$-axis. $(\Lambda_f, \boldsymbol \eta)$  is said to be an admissible polygonal grating associated with \eqref{eq:iopd}   if there exists a Lipschitz dissection of $\Gamma_j$, $1\leq j\leq \ell$,
\[
\Gamma_j=\cup_{i=1}^6 \Gamma_i^j ,%\cup\Gamma_2^j\cup\Gamma_3^j
\]
such that \eqref{eq:66} is fulfilled, where the variable function $\boldsymbol \eta$ in \eqref{eq:66}  satisfies
% one of the following three conditions:
	%	\begin{itemize}
	%		\item $\boldsymbol{\eta}\equiv 0$;
	%		\item $\boldsymbol{\eta} \equiv \infty$;
	%		\item 
	$ \boldsymbol{\eta}\in  \mathcal{A}$ with $\Im( \boldsymbol \eta) \geq 0$, and the constant part of the variable function $\boldsymbol\eta \in \mathcal A$ is not equal to 
$$
\pm \bsi \mbox{ and } \frac{\pm \sqrt{(\lambda + 3\mu)(\lambda + \mu)} - \mu \bsi}{\lambda + 2 \mu}. 
$$
%		\end{itemize}
		
%		 either belongs to the class $  \mathcal{A}$ (possibly zero) or is $\infty$, 
%	
%		
%		Suppose there is a partition, $[0,2\pi]=\cup_{i=1}^{\ell}[a_i, a_{i+1}]$ with $a_i<a_{i+1}$, $a_1=0$ and $a_{\ell+1}=2\pi$. 
%		If on each piece $[a_i, a_{i+1}]$, $1\leq i\leq\ell$, $f$ is a linear polynomial and $\eta$ is either a constant (possibly zero) or $\infty$, 

%		then $(\Lambda_f, \eta)$ is said to be an admissible polygonal grating. 
\end{defn}  

We should emphasize that in \eqref{eq:66}, either $\Gamma_1^j, \Gamma_2^j$, $\Gamma_3^j, \Gamma_4^j$, $\Gamma_5^j$ or $\Gamma_6^j$ could be an empty set.

%\begin{defn}\label{def:r1g}
%	Let $(\Lambda_f, \eta)$ be an admissible polygonal grating. Let $\Gamma^+$ and $\Gamma^-$ be two adjacent pieces of $\Lambda_f$. The intersecting point of  $\Gamma^+$ and $\Gamma^-$ is called a corner point of $\Lambda_f$, and $\angle(\Gamma^+,\Gamma^-)$ is called a corner angle. If all the corner angles of $\Lambda_f$ are irrational, then it is said to be an \emph{irrational polygonal grating}. If a corner angle 
%	of $\Lambda_f$ is rational, it is called a \emph{rational polygonal grating}. The smallest degree of the rational corner angles of $\Lambda_f$ is referred to as the \emph{rational degree} of $\Lambda_f$. 
%\end{defn}

%Clearly for a rational polygonal grating $\Lambda_f$ in Definition \ref{def:r1g}, the rational degree of $\Lambda_f$ is at least $2$. 

Next, we establish our uniqueness result in determining an admissible polygonal grating by at most eight  incident waves. We first present a useful lemma.  %whose proof follows from a completely similar argument to that of \cite[Theorem 5.1]{CK}.

%Before that we state the following lemma regarding the linear independence of $e^{\bsi k  \mathbf d_{\ell} \cdot \mathbf x} $, which can directly be obtained from the proof of \cite[Theorem 5.1]{CK}.

\begin{lem}\cite[Lemma 8.1]{CDLZ} \label{lem:83}
	Let $\mathbf \xi_{\ell}\in\mathbb{R}^2$, $\ell=1,\ldots, n$, be $n$ vectors 
	which are distinct from each other, $D$ be an open set in $\mathbb{R}^2$. 
	Then all the functions in the following set are linearly independent:
	$$
	\{\mathrm e^{\bsi \mathbf \xi_{\ell} \cdot \mathbf x} ;~\mathbf x \in D, \ \ \ell=1,2,\ldots, n \}
	$$
\end{lem}

\begin{thm}\label{thm:uniqueness1g}
	Let $(\Lambda_f, \boldsymbol \eta)$ and $(\Lambda_{\widetilde f}, \widetilde{\boldsymbol \eta})$ be two admissible polygonal gratings.  
	Let $\omega \in\mathbb{R}_+$ be fixed and $\mathbf{d}_\ell$, $\ell=1, \ldots, 8$ be eight  distinct incident directions from $\mathbb{S}^1$, with 
	\begin{equation}\label{eq:8incident} \notag
	\mathbf{d}_\ell=(\sin \theta_\ell, -\cos \theta_\ell)^\top ,\quad \theta_\ell  \in \left( - \frac{\pi}{2} , \frac{\pi}{2} \right). 
	\end{equation}
	Let $\mathbf u(\mathbf  x;\omega ,\mathbf d_\ell )$ and $\widetilde {\mathbf u}(\mathbf x;\omega ,\mathbf  d_\ell )$ denote the total fields associated with $(\Lambda_f, \boldsymbol \eta)$ and $(\Lambda_{\widetilde f}, \widetilde{\boldsymbol \eta})$ respectively and let $\Gamma_b$ be a measurement boundary given by
		$$
		\Gamma_b:=\left\{ (x_1,b)\in \R^2 ;~ 0 \leq x_1 \leq 2 \pi, \, b>\max\left\{ \max_{x_1\in [0, 2\pi]}|f(x_1)|,\, \max_{x_1\in [0, 2\pi]}|\widetilde {f}(x_1)| \right\} \right \},
		$$ 
	If it holds that
	\begin{equation}\label{eq:836}
	\mathbf u(\mathbf  x;\omega ,\mathbf d_\ell )=\widetilde {\mathbf u} (\mathbf x; \omega ,\mathbf  d_\ell ),\quad \ell=1,\ldots,8, \quad  \mathbf x=(x_1, b) \in \Gamma_b, 
	\end{equation}
	then
$$
\Lambda_{f}=\Lambda_{\widetilde f} \mbox{ and } \boldsymbol \eta= \widetilde{\boldsymbol \eta}. 
$$	
%	 it cannot be true that there exists a corner point of $\Lambda_{f}$ lying on $\partial\mathbf{G}\backslash\partial\Lambda_{\widetilde f}$, or a corner point of $\Lambda_{\widetilde f}$ lying on $\partial\mathbf{G}\backslash \partial \Lambda_{f}$. 

	%
	%%\begin{equation}\label{eq:cond1}
	%%u_\infty(\hat x, \mathbf{d}_\ell )=\widetilde u_\infty(\hat x, \mathbf{d}_\ell), \ \ \hat x\in\mathbb{S}^1, \ell=1, 2,
	%%\end{equation}
	%then one has
	%\begin{equation}\label{eq:cond21}
	%\Lambda_f =\Lambda_{\widetilde f} :=\Sigma,
	%\end{equation}
	%and
	%\begin{equation}\label{eq:cond31}
	%\eta=\widetilde\eta\ \ \mbox{on}\ \  \partial\Sigma. 
	%\end{equation}
\end{thm}

\begin{proof}
	The proof follows from a similar argument to that for Theorem~\ref{thm:uniqueness1}, and 
	we only sketch the necessary modifications in this new setup. 
	By contradiction and without loss of generality, we assume that there exists a line segment   $\Gamma_h$ of $\Lambda_{\widetilde f}$ which lies on $\partial\mathbf{G}\backslash\Lambda_{ f}$, where $\mathbf{G}$ be the unbounded connected component of $\Omega_f\cap \Omega_{\widetilde f}$. 
	
	{ First, by the well-posedness of the diffraction grating problem \eqref{eq:76}-\eqref{radiation} as well as the unique continuation, we show that $\mathbf u(\mathbf  x;\omega ,\mathbf d_\ell )=\widetilde {\mathbf u}(\mathbf x;\omega ,\mathbf  d_\ell )$ for $\mathbf{x}\in\mathbf{G}$. In fact, by introducing $\mathbf w(\mathbf{ x}; \omega , \mathbf d_\ell):=\mathbf u(\mathbf{ x}; \omega , \mathbf d_\ell)-\widetilde{\mathbf  u}(\mathbf x;\omega ,\mathbf  d_\ell)$, $\ell=1,\ldots, 8$,  
		{we see from \eqref{eq:836}} that $\mathbf w$ fulfils
		\[
		{\mathcal L}\mathbf  w+\omega ^2 \mathbf w=0\mbox{ in }\mathbf{U};  \ \mathbf w=\mathbf 0 \mbox{ on } \Gamma_b\ \ \mbox{and}\ \ \mathbf w\ \mbox{satisfies the Rayleigh expansion \eqref{radiation}, }
		\]
		where $ {\mathbf U}:=\mathbf{G} \cap  \{ \mathbf  x\in \R^2; x_2 > b, x_1 \in \R\}$ with $\partial \mathbf{U}=\Gamma_b$. Hence, by the uniqueness of the solution to the diffraction grating problem, we readily know $\mathbf w= \mathbf 0$ in $\mathbf{U}$. On the other hand, since $\mathbf u(\mathbf  x;\omega ,\mathbf d_\ell )$ and $\widetilde {\mathbf u}(\mathbf x;\omega ,\mathbf  d_\ell )$  are analytic in $\mathbf G$,
		{we know $\mathbf w(\mathbf x;\omega ,\mathbf  d_\ell )=0$ in $\mathbf{G}$ by means of the analytic continuation, that is,} 
		$\mathbf u(\mathbf  x;\omega ,\mathbf d_\ell )=\widetilde {\mathbf u}(\mathbf x;\omega ,\mathbf  d_\ell )$ for $\mathbf{x}\in\mathbf{G}$.}
	
	Next, using a similar argument to the proof of Theorem \ref{thm:uniqueness1}, we can prove that
	\begin{equation}\label{eq:542}
	\sum_{\ell=1}^8 \alpha_\ell  \mathbf u(\mathbf x; \omega ,\mathbf d_\ell) = \mathbf 0  \quad \mbox{for} \quad x_2 > \max_{x_1\in [0, 2\pi]} f(x_1), %\quad \forall \mathbf x \in  \Omega_f,
	\end{equation}
	where $\alpha_\ell$ are nonzero complex constant, $\ell=1,\ldots, 8$. 
	Next, when $x_2  > \max_{x_1\in [0, 2\pi]}|f(x_1)| $, $\mathbf u(\mathbf x;\omega , \mathbf d_\ell  )$ has the Rayleigh expansion (cf.\cite{Arens}): 
	\begin{align}\label{eq:839}
	\mathbf u(\mathbf x;\omega ,\mathbf d_\ell )&={\mathbf d}_\ell {\mathrm e}^{\bsi k_p \mathbf d_\ell  \cdot \mathbf x}+ \sum_{n\in \mathbb Z} u_{p,n}(\theta_\ell  ){\mathrm e}^{\bsi { \xi}_{p,n} (\theta_\ell  ) \cdot \mathbf  x } { \xi}_{p,n}(\theta_\ell )
%\begin{bmatrix}
%	\alpha_n \cr \beta_n
%\end{bmatrix} 
\\
&\quad +\sum_{n\in \mathbb Z} u_{s,n}(\theta_\ell  ){\mathrm e}^{\bsi { \xi}_{s,n} (\theta_\ell  ) \cdot \mathbf  x }\begin{bmatrix}
	0& 1\cr-1&0 
\end{bmatrix} { \xi}_{s,n} (\theta_\ell ) \notag
	%\quad \xi_n(\theta_\ell  )=\left(\alpha_n(\theta_\ell), \beta_n(\theta_\ell)\right)^\top, \\
	%		\alpha_n(\theta_\ell)&=n+k\sin \theta_\ell,\quad  \beta_n(\theta_\ell )=\left\{\begin{array}
	%			{c}
	%			\sqrt{k^2- \alpha_n^2 (\theta_\ell) }, \quad\mbox{ if } |\alpha_n (\theta_\ell)| \leq k\\
	%			\\[1pt]
	%			\bsi \sqrt{\alpha_n^2 (\theta_\ell)-k^2 }, \quad\mbox{ if } |\alpha_n (\theta_\ell)| > k
	%		\end{array}. \right. \notag
	\end{align}
for $x_2 > \max_{x_1\in [0, 2\pi]} f(x_1)$, 	where $\xi_n(\theta_\ell  )$, $\alpha_n(\theta_\ell)$, $\beta_{p,n}(\theta_\ell )$ and $\beta_{s,n}(\theta_\ell )$ are defined in \eqref{eq:839n}. Using the definition of $\alpha_0(\theta_\ell)$ and  $\beta_0(\theta_\ell) $ in  \eqref{eq:839n}, we can readily show that 
	\begin{equation}\label{eq:840}
	k_p\mathbf d_{\ell}=(\alpha_0(\theta_\ell), -\beta_0(\theta_\ell) )^\top . 
	\end{equation}

%	\medskip
%	
%	\noindent {\bf Case 1.}	Suppose that either $u(\mathbf x_c; k, \mathbf{d}_1)$ or $u(\mathbf x_c; k, \mathbf{d}_2)$ is zero. Without loss of generality, we assume the former case. Then
%	$$
%	u(\mathbf x;k, \mathbf d_1)=0\quad \mbox{for} \quad x_2 > \max_{x_1\in [0, 2\pi]} f(x_1). %\quad \forall \mathbf x\in \Omega_f
%	$$
%	Clearly any two vectors of $\{\xi_n (\theta_1 )~|~n\in \mathbb Z\}$ are distinct from each other.  Moreover,  in view of \eqref{eq:840}, $k \mathbf d_1 \notin \{\xi_n(\theta_1 ) ~|~n\in \mathbb Z\}$	since $|\theta_1|<\pi/2$. In view of \eqref{eq:839}, from Lemma \ref{lem:83} we can arrive at a contradiction.
%	
%	\medskip 
	
%	\noindent {\bf Case 2.}~
	
	Substituting \eqref{eq:839} into \eqref{eq:542}, it holds that
%	\begin{equation}\label{eq:841}
%	\alpha_1 u(\mathbf x; k, \mathbf{d}_1)+\alpha_2 u(\mathbf x; k, \mathbf{d}_2)=0  \quad \mbox{for} \quad  x_2 > \max_{x_1\in [0, 2\pi]} f(x_1), %\quad\mbox{in}\ \ {\Omega}_f,
%	\end{equation}
%	where $\alpha_\ell  \neq 0$, $\ell=1,2$, are defined in \eqref{eq:bb1}. Substituting \eqref{eq:839} into \eqref{eq:841}, we derive that
	\begin{align}\label{eq:842}
\mathbf 0&=	\sum_{\ell=1}^8 \alpha_\ell \mathbf d_\ell  {\mathrm e}^{\bsi k \mathbf d_\ell  \cdot \mathbf x}+ \sum_{n\in \mathbb Z}  \sum_{\ell=1}^8 \alpha_\ell u_{p,n} (\theta_\ell )  {\mathrm e}^{\bsi { \xi}_{p,n} (\theta_\ell ) \cdot \mathbf  x }{ \xi}_{p,n} (\theta_\ell )\\
&+\quad  \quad  \sum_{n\in \mathbb Z} \sum_{\ell=1}^8 \alpha_\ell  u_{s,n}(\theta_\ell  ){\mathrm e}^{\bsi { \xi}_{s,n} (\theta_\ell  ) \cdot \mathbf  x }\begin{bmatrix}
	0& 1\cr-1&0 
\end{bmatrix} { \xi}_{s,n} (\theta_\ell ) \notag\quad \mbox{for} \quad  x_2 > \max_{x_1\in [0, 2\pi]} f(x_1), %\quad \mbox{ in}\ \ \Omega_f, 
	\end{align}
	where $u_{p,n}(\theta_\ell), \, u_{s,n}(\theta_\ell) \in \C(n\in \mathbb Z)$ are the  Rayleigh coefficients of $\mathbf u^s(\mathbf x; \omega , \mathbf d_\ell)$ associated with the incident wave $\mathbf d_\ell  {\mathrm e}^{\bsi k_p \mathbf d_\ell \cdot \mathbf x}$. In \eqref{eq:842}, taking inner product with $\hat{\mathbf e}_2$,  then letting $x_2=l x_1$ ($l>0$), one has
	\begin{equation}\notag
	\begin{split}
		S_1(x_1,lx_1  )=S_2(x_1,lx_1  ), \quad \forall x_1 \in \mathbb R_+,
%		& \sum_{|\alpha_n (\theta_\ell)| >k_p}  \sum_{\ell=1}^8 \alpha_\ell u_{p,n} (\theta_\ell )  {\mathrm e}^{\bsi { \xi}_{p,n} (\theta_\ell ) \cdot (1,l) x_1 }{ \beta }_{p,n} (\theta_\ell )\\
%		&\quad -\sum_{ |\alpha_n(\theta_\ell )| >k_s } \sum_{\ell=1}^8 \alpha_\ell  u_{s,n}(\theta_\ell  ){\mathrm e}^{\bsi { \xi}_{s,n} (\theta_\ell  ) \cdot (1,l)  x_1 } { \alpha }_{n} (\theta_\ell )
	\end{split}
	\end{equation}
	where
	\begin{equation}\notag
		\begin{split}
		S_1(x_1,lx_1  )=		&\sum_{\ell=1}^8 \alpha_\ell \cos(\theta_\ell)  {\mathrm e}^{\bsi k \mathbf d_\ell  \cdot (1,l) x_1}- \sum_{|\alpha_n (\theta_\ell)| \leq k_p}  \sum_{\ell=1}^8 \alpha_\ell u_{p,n} (\theta_\ell )  {\mathrm e}^{\bsi { \xi}_{p,n} (\theta_\ell ) \cdot (1,l) x_1 }{ \beta }_{p,n} (\theta_\ell )\\
		&\quad +\sum_{ |\alpha_n(\theta_\ell )| \leq k_s } \sum_{\ell=1}^8 \alpha_\ell  u_{s,n}(\theta_\ell  ){\mathrm e}^{\bsi { \xi}_{s,n} (\theta_\ell  ) \cdot (1,l)  x_1 } { \alpha }_{n} (\theta_\ell ),\\
		S_2(x_1,lx_1  )&= \sum_{|\alpha_n (\theta_\ell)| >k_p}  \sum_{\ell=1}^8 \alpha_\ell u_{p,n} (\theta_\ell )  {\mathrm e}^{\bsi { \xi}_{p,n} (\theta_\ell ) \cdot (1,l) x_1 }{ \beta }_{p,n} (\theta_\ell )\\
		&\quad-\sum_{ |\alpha_n(\theta_\ell )| >k_s } \sum_{\ell=1}^8 \alpha_\ell  u_{s,n}(\theta_\ell  ){\mathrm e}^{\bsi { \xi}_{s,n} (\theta_\ell  ) \cdot (1,l)  x_1 } { \alpha }_{n} (\theta_\ell ). 
		\end{split}
	\end{equation}
Noting that $S_1(x_1,lx_1  )$ is an almost periodic function on $\mathbb R_+$ and $S_2(x_1,lx_1  )$	is exponentially decaying functions as $x_1 \rightarrow +\infty$, from \cite[(d) in Page 493]{Bottcher} we can prove that 
\begin{equation}\notag
	\begin{split}
	\max_{ x_1 \in \mathbb R_+ }	 |S_1(x_1,lx_1  )|=\limsup_{x_1\rightarrow +\infty } |S_1(x_1,lx_1  )|=\limsup_{x_1\rightarrow +\infty } |S_2(x_1,lx_1  )| \leq
	 \varepsilon
	\end{split}
\end{equation}
for any $\varepsilon \in \mathbb R_+$. Therefore it yields that 
\begin{equation}\label{eq:548}
S_1(x_1,lx_1  ) \equiv 0,\quad \forall x_1 \in \mathbb R_+. 
\end{equation}

	Clearly, any two vectors of the set 
	$$
	 \bigcup_{\ell=1}^8  \{ k_p \mathbf d_\ell \}\bigcup_{\ell=1}^8 \bigcup   \{\xi_{p,n} (\theta_\ell )~|~n\in \mathbb Z\}\bigcup_{\ell=1}^8 \bigcup \{\xi_{s,n} (\theta_\ell )~|~n\in \mathbb Z\} 
	$$
	are distinct since $|\theta_\ell| < \pi/2$ and \eqref{eq:840}. 
	
		Using Lemma \ref{lem:83} and 
		\eqref{eq:548}, we can see $\alpha_\ell=0$ for $\ell=1,\ldots, 8$ by noting $\cos \theta_\ell \in \mathbb R_+$, which is a contradiction to the fact there exits an index $\ell_0\in \{ 1,\ldots, 8\}$ such that $\alpha_{\ell_0}  \neq 0$. 
\end{proof}

Similar to Theorem \ref{thm:uniqueness2},  we can establish a local uniqueness result in determining a generic class of admissible polygonal grating by using fewer scattering measurements, which we choose not to discuss the details in this paper.

\vspace*{-.3cm}

\section*{Acknowledgement}

The work  of H Diao was supported in part by the National Natural Science Foundation of China under grant 11001045 and the Fundamental Research Funds for the Central Universities under the grant 2412017FZ007.  The work of H Liu was supported by the startup fund from City University of Hong Kong and the Hong Kong RGC General Research Fund (projects 12301420, 12302919 and 12301218).

\section*{Appendix}

\begin{proof}[Proof of Lemma~\ref{lem:condition}]
We first prove \eqref{eq:gradient1}. Using \eqref{eq:u1} and \eqref{eq:u2}, it is directly shown that
\begin{eqnarray}
%\begin{equation}\label{eq:u1 par}
%\begin{aligned}
\frac{\partial u_1}{\partial r}
&=& \sum_{m=0} ^\infty \left\{ \frac{k_p^2}{4}a_m \left\{\mathrm{e}^{\bsi \left(m-1\right) \varphi} J_{m-2}\left(k_p r\right) - \left[\mathrm{e}^{\bsi \left(m-1\right) \varphi}+\mathrm{e}^{\bsi \left(m+1\right) \varphi}\right] J_m \left(k_p r\right) \right\}\right.\notag \\
&& \hspace*{-3mm} + \frac{\bsi k_s^2}{4}b_m \left\{\mathrm{e}^{\bsi \left(m-1\right) \varphi} J_{m-2}\left(k_s r\right) - \left[\mathrm{e}^{\bsi \left(m-1\right) \varphi}-\mathrm{e}^{\bsi \left(m+1\right) \varphi}\right] J_m \left(k_s r\right) \right\}\notag \\
&& \hspace*{-3mm} + \frac{k_p^2}{4}a_m \mathrm{e}^{\bsi \left(m+1\right) \varphi} J_{m+2} \left(k_p r\right)-\frac{\bsi k_s^2}{4}b_m \mathrm{e}^{\bsi \left(m+1\right) \varphi} J_{m+2} \left(k_s r\right)
 \bigg\},\notag \\
%\end{aligned}
%\end{equation}
%and
%\begin{equation}
%\begin{aligned}
\frac{\partial u_1}{\partial \varphi}
 & =& \sum_{m=0} ^\infty  \left\{\frac{\bsi\left(m-1\right)}{2} k_p \mathrm{e}^{\bsi \left(m-1\right) \varphi} J_{m-1} \left(k_p r\right) a_m - \frac{\bsi\left(m+1\right)}{2} k_p \mathrm{e}^{\bsi \left(m+1\right) \varphi} J_{m+1} \left(k_p r\right) a_m \right. \notag \\
 && \hspace*{-3mm} - \frac{\left(m-1\right)}{2} k_s \mathrm{e}^{\bsi \left(m-1\right) \varphi} J_{m-1} \left(k_s r\right) b_m - \frac{\left(m+1\right)}{2} k_s \mathrm{e}^{\bsi \left(m+1\right) \varphi} J_{m+1} \left(k_s r\right) b_m\bigg\}, \label{eq:u1 par}
%\end{aligned}
%\end{equation}
\end{eqnarray}
and
\begin{eqnarray}
%\begin{equation}\label{eq:u2 par}
%\begin{aligned}
\frac{\partial u_2}{\partial r}
&=& \sum_{m=0} ^\infty \left\{\frac{\bsi k_p^2}{4}a_m \left\{\mathrm{e}^{\bsi \left(m-1\right) \varphi} J_{m-2}\left(k_p r\right) - \left[\mathrm{e}^{\bsi \left(m-1\right) \varphi}-\mathrm{e}^{\bsi \left(m+1\right) \varphi}\right] J_m \left(k_p r\right)\right\}\right.\notag \\
&&\hspace*{-3mm}  +  \frac{k_s^2}{4}b_m \left\{-\mathrm{e}^{\bsi \left(m-1\right) \varphi} J_{m-2}\left(k_s r\right) + \left[\mathrm{e}^{\bsi \left(m-1\right) \varphi}+\mathrm{e}^{\bsi \left(m+1\right) \varphi}\right] J_m \left(k_s r\right) \right\}\notag \\
&&\hspace*{-3mm} -  \frac{\bsi k_p^2}{4}a_m \mathrm{e}^{\bsi \left(m+1\right) \varphi} J_{m+2} \left(k_p r\right) -\frac{k_s^2}{4} b_m \mathrm{e}^{\bsi \left(m+1\right) \varphi} J_{m+2} \left(k_s r\right)
\bigg\},\notag \\
%\end{aligned}
%\end{equation}
%and
%\begin{equation}
%\begin{aligned}
 \frac{\partial u_2}{\partial \varphi}
&=& \sum_{m=0} ^\infty  \left\{-\frac{\left(m-1\right)}{2} k_p \mathrm{e}^{\bsi \left(m-1\right) \varphi} J_{m-1} \left(k_p r\right) a_m - \frac{\left(m+1\right)}{2} k_p \mathrm{e}^{\bsi \left(m+1\right) \varphi} J_{m+1} \left(k_p r\right) a_m \right.\notag \\
 && \hspace*{-3mm}- \frac{\bsi \left(m-1\right)}{2} k_s \mathrm{e}^{\bsi \left(m-1\right) \varphi} J_{m-1} \left(k_s r\right) b_m + \frac{\bsi \left(m+1\right)}{2} k_s \mathrm{e}^{\bsi \left(m+1\right) \varphi} J_{m+1} \left(k_s r\right) b_m\bigg\}. \label{eq:u2 par}
%\end{aligned}
%\end{equation}
\end{eqnarray}
For $i=1,2$, one has
\begin{equation}\label{eq:u3 par}
	\begin{split}
		\frac{\partial u_i}{\partial x_1}&=\cos\varphi \cdot \frac{\partial u_i}{\partial r}- \frac{\sin \varphi}{r} \cdot \frac{\partial u_i}{\partial \varphi},\quad
		\frac{\partial u_i}{\partial x_2}=\sin \varphi \cdot \frac{\partial u_i}{\partial r}+ \frac{\cos \varphi}{r} \cdot \frac{\partial u_i}{\partial \varphi}.
	\end{split}
\end{equation}

Combining  \eqref{eq:u1 par}, \eqref{eq:u2 par} with \eqref{eq:u3 par}, after tedious but straightforward calculations, one can obtain that
\begin{eqnarray}
%\begin{equation}\label{eq:u4a par}
%\begin{aligned}
& &\partial_1 u_1 \cdot \left(-\sin \varphi_0\right)+\partial_1 u_2  \cdot \left(\cos \varphi_0\right)
\notag \\
&&  =  \sum_{m=0} ^\infty  \Bigg \{ \cos{\varphi}\Big[\frac{\bsi k_p^2}{4} a_m \mathrm{e}^{\bsi \left(m-1\right) \varphi} J_{m-2}(k_p r) \mathrm{e}^{\bsi \varphi_0}  - \frac{\bsi k_p^2}{4} a_m J_m(k_p r)\mathrm{e}^{\bsi (m-1)\varphi}(\mathrm{e}^{\bsi \varphi_0} -\mathrm{e}^{2 \bsi \varphi}  \mathrm{e}^{-\bsi \varphi_0} )\notag \\
&& \quad  - \frac{\bsi k_p^2}{4} a_m \mathrm{e}^{\bsi \left(m+1\right) \varphi} J_{m+2}(k_p r) \mathrm{e}^{-\bsi \varphi_0}  - \frac{ k_s^2}{4} b_m \mathrm{e}^{\bsi \left(m-1\right) \varphi} J_{m-2}(k_s r) \mathrm{e}^{\bsi \varphi_0} \notag \\
&& \quad  + \frac{ k_s^2}{4} b_m J_m(k_s r) \mathrm{e}^{\bsi (m-1) \varphi} (\mathrm{e}^{\bsi \varphi_0} +\mathrm{e}^{2 \bsi \varphi}  \mathrm{e}^{-\bsi \varphi_0} ) - \frac{ k_s^2}{4} b_m J_{m+2}(k_s r) \mathrm{e}^{\bsi (m+1) \varphi} \mathrm{e}^{-\bsi \varphi_0} \Big ] \notag \\
&&\quad  + \frac{ \sin{\varphi} }{r} \Big [\frac{(m-1)k_p}{2} a_m \mathrm{e}^{\bsi (m-1)\varphi} J_{m-1}(k_p r) \mathrm{e}^{\bsi \varphi_0}\notag \\
&& \quad + \frac{ (m+1)k_p}{2} a_m \mathrm{e}^{\bsi (m+1)\varphi} J_{m+1}(k_p r) \mathrm{e}^{-\bsi \varphi_0} + \frac{\bsi (m-1)k_s}{2} b_m \mathrm{e}^{\bsi (m-1)\varphi} J_{m-1}(k_s r) \mathrm{e}^{\bsi \varphi_0}\notag \\
 &&\quad - \frac{\bsi (m+1)k_s}{2} b_m \mathrm{e}^{\bsi (m+1)\varphi} J_{m+1}(k_s r) \mathrm{e}^{-\bsi \varphi_0}\Big] \Bigg  \}, \notag \\
&&  \partial_2 u_1 \cdot \left(-\sin \varphi_0\right)+\partial_2 u_2 \cdot \left(\cos \varphi_0\right)\notag \\
&& =  \sum_{m=0} ^\infty \Bigg \{ \sin{\varphi } \Big [ \frac{\bsi k_p^2}{4} a_m \mathrm{e}^{\bsi \left(m-1\right) \varphi} J_{m-2}(k_p r) \mathrm{e}^{\bsi \varphi_0}  - \frac{\bsi k_p^2}{4} a_m J_m(k_p r) \mathrm{e}^{\bsi (m-1)\varphi} (\mathrm{e}^{\bsi \varphi_0} -\mathrm{e}^{2 \bsi \varphi}  \mathrm{e}^{-\bsi \varphi_0} )  \notag \\
&&\quad  - \frac{\bsi k_p^2}{4} a_m \mathrm{e}^{\bsi \left(m+1\right) \varphi} J_{m+2}(k_p r) \mathrm{e}^{-\bsi \varphi_0} - \frac{ k_s^2}{4} b_m \mathrm{e}^{\bsi \left(m-1\right) \varphi} J_{m-2}(k_s r)\mathrm{e}^{\bsi \varphi_0}\notag \\
&&\quad + \frac{ k_s^2}{4} b_m J_m(k_s r) \mathrm{e}^{\bsi (m-1)\varphi}   (\mathrm{e}^{\bsi \varphi_0} +\mathrm{e}^{2 \bsi \varphi}  \mathrm{e}^{-\bsi \varphi_0} )  - \frac{ k_s^2}{4} b_m \mathrm{e}^{\bsi \left(m+1\right) \varphi} J_{m+2}(k_s r) \mathrm{e}^{-\bsi \varphi_0}  \Big] \notag \\
&& \quad + \frac{\cos{\varphi} }{r} \Big[ \frac{-(m-1)k_p}{2} a_m \mathrm{e}^{\bsi (m-1)\varphi} J_{m-1}(k_p r) \mathrm{e}^{\bsi \varphi_0}\notag \\
&& \quad - \frac{(m+1)k_p}{2} a_m \mathrm{e}^{\bsi (m+1)\varphi} J_{m+1}(k_p r) \mathrm{e}^{-\bsi \varphi_0}  - \frac{\bsi(m-1)k_s}{2} b_m \mathrm{e}^{\bsi (m-1)\varphi} J_{m-1}(k_s r) \mathrm{e}^{\bsi \varphi_0} \notag \\
 &&\quad + \frac{\bsi(m+1)k_s}{2} b_m \mathrm{e}^{\bsi (m+1)\varphi} J_{m+1}(k_s r) \mathrm{e}^{-\bsi \varphi_0}  \Big] \Bigg \}, \label{eq:u4a par}
%\end{aligned}
%\end{equation}
\end{eqnarray}
and
\begin{eqnarray}
%\begin{equation}\label{eq:u4b par}
%\begin{aligned}
&& \partial_1 u_1 \cdot \left(-\cos \varphi_0\right)+\partial_1 u_2  \cdot \left(-\sin \varphi_0\right) \notag \\
&& \quad =  \sum_{m=0} ^\infty  \Bigg \{ \cos{\varphi}\Big[\frac{- k_p^2}{4} a_m \mathrm{e}^{\bsi \left(m-1\right) \varphi} J_{m-2}(k_p r) \mathrm{e}^{\bsi \varphi_0}  + \frac{ k_p^2}{4} a_m J_m(k_p r)\mathrm{e}^{\bsi (m-1)\varphi}(\mathrm{e}^{\bsi \varphi_0} +\mathrm{e}^{2 \bsi \varphi}  \mathrm{e}^{-\bsi \varphi_0} )  \notag \\
&&\quad  - \frac{ k_p^2}{4} a_m \mathrm{e}^{\bsi \left(m+1\right) \varphi} J_{m+2}(k_p r) \mathrm{e}^{-\bsi \varphi_0}  - \frac{\bsi k_s^2}{4} b_m \mathrm{e}^{\bsi \left(m-1\right) \varphi} J_{m-2}(k_s r) \mathrm{e}^{\bsi \varphi_0}  \notag  \\
&&\quad  + \frac{\bsi k_s^2}{4} b_m J_m(k_s r) \mathrm{e}^{\bsi (m-1) \varphi} (\mathrm{e}^{\bsi \varphi_0} -\mathrm{e}^{2 \bsi \varphi}  \mathrm{e}^{-\bsi \varphi_0} ) + \frac{\bsi k_s^2}{4} b_m J_{m+2}(k_s r) \mathrm{e}^{\bsi (m+1) \varphi}  \mathrm{e}^{-\bsi \varphi_0} \Big ]  \notag  \\
&&\quad  + \frac{ \sin{\varphi} }{r} \Big [\frac{(m-1)k_p}{2} a_m \mathrm{e}^{\bsi (m-1)\varphi} J_{m-1}(k_p r) \mathrm{e}^{\bsi \varphi_0}  \notag  \\
&& \quad - \frac{\bsi (m+1)k_p}{2} a_m \mathrm{e}^{\bsi (m+1)\varphi} J_{m+1}(k_p r) \mathrm{e}^{-\bsi \varphi_0} - \frac{(m-1)k_s}{2} b_m \mathrm{e}^{\bsi (m-1)\varphi} J_{m-1}(k_s r) \mathrm{e}^{\bsi \varphi_0} \notag  \\
 &&\quad - \frac{ (m+1)k_s}{2} b_m \mathrm{e}^{\bsi (m+1)\varphi} J_{m+1}(k_s r) \mathrm{e}^{-\bsi \varphi_0}\Big] \Bigg  \},  \notag \\
&& \partial_2 u_1 \cdot \left(-\cos \varphi_0\right)+\partial_2 u_2 \cdot \left(-\sin \varphi_0\right)  \notag
 \\
&&\quad =  \sum_{m=0} ^\infty \Bigg \{ \sin{\varphi } \Big [ \frac{- k_p^2}{4} a_m \mathrm{e}^{\bsi \left(m-1\right) \varphi} J_{m-2}(k_p r) \mathrm{e}^{\bsi \varphi_0}  + \frac{ k_p^2}{4} a_m J_m(k_p r) \mathrm{e}^{\bsi (m-1)\varphi} (\mathrm{e}^{\bsi \varphi_0} + \mathrm{e}^{2 \bsi \varphi}  \mathrm{e}^{-\bsi \varphi_0} )   \notag  \\
&&\quad  - \frac{ k_p^2}{4} a_m \mathrm{e}^{\bsi \left(m+1\right) \varphi} J_{m+2}(k_p r) \mathrm{e}^{-\bsi \varphi_0} - \frac{\bsi k_s^2}{4} b_m \mathrm{e}^{\bsi \left(m-1\right) \varphi} J_{m-2}(k_s r)\mathrm{e}^{\bsi \varphi_0} \notag   \\
&&\quad + \frac{\bsi k_s^2}{4} b_m J_m(k_s r) \mathrm{e}^{\bsi (m-1)\varphi}   (\mathrm{e}^{\bsi \varphi_0} - \mathrm{e}^{2 \bsi \varphi}  \mathrm{e}^{-\bsi \varphi_0} )  + \frac{ \bsi k_s^2}{4} b_m \mathrm{e}^{\bsi \left(m+1\right) \varphi} J_{m+2}(k_s r) \mathrm{e}^{-\bsi \varphi_0}  \Big]   \notag\\
& &\quad + \frac{\cos{\varphi} }{r} \Big[ \frac{- \bsi (m-1)k_p}{2} a_m \mathrm{e}^{\bsi (m-1)\varphi} J_{m-1}(k_p r) \mathrm{e}^{\bsi \varphi_0}  \notag \\
&& \quad + \frac{\bsi (m+1)k_p}{2} a_m \mathrm{e}^{\bsi (m+1)\varphi} J_{m+1}(k_p r) \mathrm{e}^{-\bsi \varphi_0}  + \frac{ (m-1)k_s}{2} b_m \mathrm{e}^{\bsi (m-1)\varphi} J_{m-1}(k_s r) \mathrm{e}^{\bsi \varphi_0} \notag \\
& &\quad + \frac{ (m+1)k_s}{2} b_m \mathrm{e}^{\bsi (m+1)\varphi} J_{m+1}(k_s r) \mathrm{e}^{-\bsi \varphi_0}  \Big] \Bigg \}.\label{eq:u4b par}
%\end{aligned}
%\end{equation}
\end{eqnarray}

Furthermore, we have
\begin{eqnarray}
%\begin{equation}\label{eq:u4c par}
%\begin{aligned}
&& \partial_1 u_1 \cdot \left(-\sin \varphi_0\right)+\partial_2 u_1  \cdot \left(\cos \varphi_0\right) \notag
\\
&& =  \sum_{m=0} ^\infty \bigg\{\frac{\bsi k_p^2}{4} a_m \mathrm{e}^{\bsi (m-2)\varphi} J_{m-2}(k_p r) \mathrm{e}^{\bsi \varphi_0} + \frac{ k_p^2}{2} a_m \mathrm{e}^{\bsi m \varphi} J_{m}(k_p r) \sin \varphi_0 \notag\\
& &- \frac{\bsi k_p^2}{4} a_m \mathrm{e}^{\bsi (m+2)\varphi} J_{m+2}(k_p r) \mathrm{e}^{ - \bsi \varphi_0} - \frac{ k_s^2}{4} b_m \mathrm{e}^{\bsi (m-2)\varphi} J_{m-2}(k_s r) \mathrm{e}^{\bsi \varphi_0}\notag \\
& &-  \frac{ k_s^2}{2} b_m \mathrm{e}^{\bsi m \varphi} J_{m}(k_s r) \cos \varphi_0 - \frac{ k_s^2}{4} b_m \mathrm{e}^{\bsi (m+2) \varphi} J_{m+2}(k_s r) \mathrm{e}^{ - \bsi \varphi_0}  \bigg \}.\notag\\
&& \partial_1 u_2 \cdot \left(-\sin \varphi_0\right)+\partial_2 u_2  \cdot \left(\cos \varphi_0\right)\notag\\
& &=  \sum_{m=0} ^\infty \bigg\{- \frac{ k_p^2}{4} a_m \mathrm{e}^{\bsi (m-2)\varphi} J_{m-2}(k_p r) \mathrm{e}^{\bsi \varphi_0} - \frac{ k_p^2}{2} a_m \mathrm{e}^{\bsi m \varphi} J_{m}(k_p r) \cos \varphi_0 \notag\\
&& - \frac{ k_p^2}{4} a_m \mathrm{e}^{\bsi (m+2)\varphi} J_{m+2}(k_p r) \mathrm{e}^{ - \bsi \varphi_0} - \frac{ \bsi k_s^2}{4} b_m \mathrm{e}^{\bsi (m-2)\varphi} J_{m-2}(k_s r) \mathrm{e}^{\bsi \varphi_0}\notag \\
& &-  \frac{ k_s^2}{2} b_m \mathrm{e}^{\bsi m \varphi} J_{m}(k_s r) \sin \varphi_0 + \frac{ \bsi k_s^2}{4} b_m \mathrm{e}^{\bsi (m+2) \varphi} J_{m+2}(k_s r) \mathrm{e}^{ - \bsi \varphi_0}  \bigg \}.\label{eq:u4c par}
%\end{aligned}
%\end{equation}
\end{eqnarray}

According to \eqref{eq:u4a par}, we can obtain that
\begin{eqnarray}
%\begin{equation}\label{eq:u5 par}
%\begin{aligned}
& &-\sin \varphi_0 (-\sin \varphi_0 \cdot \partial_1 u_1 +\cos \varphi_0 \cdot \partial_1 u_2 ) + \cos \varphi_0 ( -\sin \varphi_0 \cdot \partial_2 u_1 + \cos \varphi_0 \cdot \partial_2 u_2) \notag \\
&&  =  \sum_{m=0} ^\infty  \Bigg \{ \sin(\varphi_0-\varphi) \Big[\frac{-\bsi k_p^2}{4} a_m \mathrm{e}^{\bsi \left(m-1\right) \varphi} J_{m-2}(k_p r) \mathrm{e}^{\bsi \varphi_0}  \notag \\
& &\quad + \frac{\bsi k_p^2}{4} a_m J_m(k_p r)\mathrm{e}^{\bsi (m-1)\varphi}(\mathrm{e}^{\bsi \varphi_0} -\mathrm{e}^{2 \bsi \varphi}  \mathrm{e}^{-\bsi \varphi_0} )+ \frac{\bsi k_p^2}{4} a_m \mathrm{e}^{\bsi \left(m+1\right) \varphi} J_{m+2}(k_p r) \mathrm{e}^{-\bsi \varphi_0} \notag \\
 && \quad + \frac{ k_s^2}{4} b_m \mathrm{e}^{\bsi \left(m-1\right) \varphi} J_{m-2}(k_s r) \mathrm{e}^{\bsi \varphi_0}  - \frac{ k_s^2}{4} b_m J_m(k_s r) \mathrm{e}^{\bsi (m-1) \varphi} (\mathrm{e}^{\bsi \varphi_0} +\mathrm{e}^{2 \bsi \varphi}  \mathrm{e}^{-\bsi \varphi_0} )\notag \\
  && \quad + \frac{ k_s^2}{4} b_m J_{m+2}(k_s r) \mathrm{e}^{\bsi (m+1) \varphi} J_{m+2}(k_s r) \mathrm{e}^{-\bsi \varphi_0} \Big ] \notag \\
&&\quad  + \frac{ \cos (\varphi_0-\varphi) }{r} \Big [\frac{-(m-1)k_p}{2} a_m \mathrm{e}^{\bsi (m-1)\varphi} J_{m-1}(k_p r) \mathrm{e}^{\bsi \varphi_0}\notag \\
& &\quad - \frac{ (m+1)k_p}{2} a_m \mathrm{e}^{\bsi (m+1)\varphi} J_{m+1}(k_p r) \mathrm{e}^{-\bsi \varphi_0} - \frac{\bsi (m-1)k_s}{2} b_m \mathrm{e}^{\bsi (m-1)\varphi} J_{m-1}(k_s r) \mathrm{e}^{\bsi \varphi_0} \notag \\
 &&\quad + \frac{\bsi (m+1)k_s}{2} b_m \mathrm{e}^{\bsi (m+1)\varphi} J_{m+1}(k_s r) \mathrm{e}^{-\bsi \varphi_0}\Big] \Bigg  \}. \label{eq:u5 par}
%\end{aligned}
%\end{equation}
\end{eqnarray}

By directly calculations, one has
\begin{equation}\label{eq:u6 par}
\begin{aligned}
 &{\nu}^\top \nabla \bmf{u} {\nu}|_{\bmf x \in \Gamma_h^+ }
   =\left(
  \begin{array}{cc}
  \nu_1 & \nu_2
  \end{array}
  \right)
  \left(
  \begin{array}{cc}
  \partial_1 u_1 & \partial_2 u_1\\
  \partial_1 u_2 & \partial_2 u_2\\
  \end{array}
  \right)
  \left(
  \begin{array}{c}
  \nu_1 \\
  \nu_2 \\
  \end{array}
  \right)\\
  & = -\sin \varphi_0 (-\sin \varphi_0 \cdot \partial_1 u_1 +\cos \varphi_0 \cdot \partial_1 u_2 ) + \cos \varphi_0 ( -\sin \varphi_0 \cdot \partial_2 u_1 + \cos \varphi_0 \cdot \partial_2 u_2).
\end{aligned}
\end{equation}
Substituting   \eqref{eq:u5 par} into \eqref{eq:u6 par},  letting $\varphi=\varphi_0$, according to the assumption in \eqref{eq:gradient1}, we can deduce that
\begin{equation}\label{eq:u7 par}
\begin{aligned}
0=& \frac{1}{r} \sum_{m=0} ^\infty \Bigg\{ \frac{-(m-1)k_p}{2} a_m \mathrm{e}^{\bsi m \varphi_0} J_{m-1}(k_p r)  - \frac{ (m+1)k_p}{2} a_m \mathrm{e}^{\bsi m \varphi_0} J_{m+1}(k_p r) \\
&  - \frac{\bsi (m-1)k_s}{2} b_m \mathrm{e}^{\bsi m \varphi_0} J_{m-1}(k_s r) + \frac{\bsi (m+1)k_s}{2} b_m \mathrm{e}^{\bsi m \varphi_0} J_{m+1}(k_s r) \Bigg\}\\
  = & \sum_{m=0} ^\infty \Bigg\{ \frac{-k_p^2}{4} a_m \mathrm{e}^{\bsi m \varphi_0} J_{m-2}(k_p r)  - \frac{k_p^2}{2} a_m \mathrm{e}^{\bsi m \varphi_0} J_m (k_p r) -  \frac{k_p^2}{4} a_m \mathrm{e}^{\bsi m \varphi_0} J_{m+2}(k_p r)\\
&  - \frac{\bsi  k_s^2}{4} b_m \mathrm{e}^{\bsi m \varphi_0} J_{m-2}(k_s r) + \frac{\bsi k_s^2}{4} b_m \mathrm{e}^{\bsi m \varphi_0} J_{m+2}(k_s r) \Bigg\},
\end{aligned}
\end{equation}
where we use the property (cf.\cite{Abr})
\begin{align}\label{eq:property}
	J_m\left(t\right)=\frac{t\left(J_{m-1}\left(t\right)+J_{m+1}\left(t\right)\right)}{2m}.
\end{align}
 Comparing the coefficients of $r^0$ on both sides of \eqref{eq:u7 par}, we can  obtain \eqref{eq:gradient1}.

Next we  shall prove \eqref{eq:gradient2}. According to \eqref{eq:u4b par}, we can obtain that
\begin{eqnarray}
%\begin{equation}\label{eq:u8 par}
\begin{aligned}
& -\sin \varphi_0 (-\cos \varphi_0 \cdot \partial_1 u_1 -\sin \varphi_0 \cdot \partial_1 u_2 ) + \cos \varphi_0 ( -\cos \varphi_0 \cdot \partial_2 u_1 - \sin \varphi_0 \cdot \partial_2 u_2)  \\
& =  \sum_{m=0} ^\infty  \Bigg \{ \sin(\varphi_0-\varphi) \Big[\frac{ k_p^2}{4} a_m \mathrm{e}^{\bsi \left(m-1\right) \varphi} J_{m-2}(k_p r) \mathrm{e}^{\bsi \varphi_0}   \\
& \quad - \frac{ k_p^2}{4} a_m J_m(k_p r)\mathrm{e}^{\bsi (m-1)\varphi}(\mathrm{e}^{\bsi \varphi_0} + \mathrm{e}^{2 \bsi \varphi}  \mathrm{e}^{-\bsi \varphi_0} )+ \frac{ k_p^2}{4} a_m \mathrm{e}^{\bsi \left(m+1\right) \varphi} J_{m+2}(k_p r) \mathrm{e}^{-\bsi \varphi_0}  \\
 & \quad + \frac{\bsi k_s^2}{4} b_m \mathrm{e}^{\bsi \left(m-1\right) \varphi} J_{m-2}(k_s r) \mathrm{e}^{\bsi \varphi_0}  - \frac{\bsi k_s^2}{4} b_m J_m(k_s r) \mathrm{e}^{\bsi (m-1) \varphi} (\mathrm{e}^{\bsi \varphi_0} - \mathrm{e}^{2 \bsi \varphi}  \mathrm{e}^{-\bsi \varphi_0} ) \\
  & \quad - \frac{\bsi k_s^2}{4} b_m J_{m+2}(k_s r) \mathrm{e}^{\bsi (m+1) \varphi} J_{m+2}(k_s r) \mathrm{e}^{-\bsi \varphi_0} \Big ]  \\
&\quad  + \frac{ \cos (\varphi_0-\varphi) }{r} \Big [\frac{-\bsi (m-1)k_p}{2} a_m \mathrm{e}^{\bsi (m-1)\varphi} J_{m-1}(k_p r) \mathrm{e}^{\bsi \varphi_0}  \\
& \quad + \frac{ \bsi (m+1)k_p}{2} a_m \mathrm{e}^{\bsi (m+1)\varphi} J_{m+1}(k_p r) \mathrm{e}^{-\bsi \varphi_0} + \frac{(m-1)k_s}{2} b_m \mathrm{e}^{\bsi (m-1)\varphi} J_{m-1}(k_s r) \mathrm{e}^{\bsi \varphi_0}  \\
 &\quad + \frac{ (m+1)k_s}{2} b_m \mathrm{e}^{\bsi (m+1)\varphi} J_{m+1}(k_s r) \mathrm{e}^{-\bsi \varphi_0}\Big] \Bigg  \}. \label{eq:u8 par}
\end{aligned}
%\end{equation}
\end{eqnarray}
%We know $\boldsymbol{\tau}^\top \nabla \bmf{u} \boldsymbol{\nu}(\bmf{0} )=0$ by condition,
One can readily deduce that
\begin{equation}\label{eq:u9 par}
\begin{aligned}
 &\boldsymbol{\tau}^\top \nabla \bmf{u} {\nu}|_{\mathbf x \in \Gamma_h^+ }
   =\left(
  \begin{array}{cc}
  \tau_1 & \tau_2
  \end{array}
  \right)
  \left(
  \begin{array}{cc}
  \partial_1 u_1 & \partial_2 u_1\\
  \partial_1 u_2 & \partial_2 u_2\\
  \end{array}
  \right)
  \left(
  \begin{array}{c}
  \nu_1 \\
  \nu_2 \\
  \end{array}
  \right)\\
  & = -\sin \varphi_0 (-\cos \varphi_0 \cdot \partial_1 u_1 - \sin \varphi_0 \cdot \partial_1 u_2 ) + \cos \varphi_0 ( -\cos \varphi_0 \cdot \partial_2 u_1 - \sin \varphi_0 \cdot \partial_2 u_2).
\end{aligned}
\end{equation}
Substituting \eqref{eq:u8 par} into \eqref{eq:u9 par}, letting $\varphi=\varphi_0$, from \eqref{eq:property}, according to the assumption in \eqref{eq:gradient2}, it yields that
\begin{equation}\label{eq:u10 par}
\begin{aligned}
 0=& \frac{1}{r} \sum_{m=0} ^\infty \Bigg\{ \frac{- \bsi (m-1)k_p}{2} a_m \mathrm{e}^{\bsi m \varphi_0} J_{m-1}(k_p r)  + \frac{ \bsi (m+1)k_p}{2} a_m \mathrm{e}^{\bsi m \varphi_0} J_{m+1}(k_p r) \\
&  + \frac{ (m-1)k_s}{2} b_m \mathrm{e}^{\bsi m \varphi_0} J_{m-1}(k_s r) + \frac{ (m+1)k_s}{2} b_m \mathrm{e}^{\bsi m \varphi_0} J_{m+1}(k_s r) \Bigg\}\\
  = & \sum_{m=0} ^\infty \Bigg\{ \frac{- \bsi k_p^2}{4} a_m \mathrm{e}^{\bsi m \varphi_0} J_{m-2}(k_p r)  +  \frac{ \bsi k_p^2}{4} a_m \mathrm{e}^{\bsi m \varphi_0} J_{m+2}(k_p r)\\
&  + \frac{ k_s^2}{4} b_m \mathrm{e}^{\bsi m \varphi_0} J_{m-2}(k_s r) +  \frac{ k_s^2}{2} b_m \mathrm{e}^{\bsi m \varphi_0} J_{m+2}(k_s r) + \frac{ k_s^2}{4} b_m \mathrm{e}^{\bsi m \varphi_0} J_{m+2}(k_s r) \Bigg\}.
\end{aligned}
\end{equation}
Comparing the coefficients of $r^0$ on both sides of \eqref{eq:u10 par}, we can  obtain  \eqref{eq:gradient2}.
\end{proof}

\end{document}